\documentclass[12pt,a4paper]{amsart}
\usepackage{amsmath}
\usepackage{amsthm}
\usepackage{bbold}
\usepackage{amsfonts}
\usepackage{xcolor}
\usepackage{amssymb}
\usepackage{float}
\usepackage{amsthm}
\usepackage[all]{xy}
\usepackage{psfrag}
\usepackage{hyperref}
\usepackage{amssymb}

\newcommand{\rsdraw}[3]{\raisebox{-#1\height}{\scalebox{#2}{\includegraphics{#3.eps}}}}
\newtheorem{thm}{Theorem}
\newtheorem{lem}[thm]{Lemma}
\newtheorem{cor}[thm]{Corollary}
\newtheorem{rmk}[thm]{Remark}
\newcommand{\co}{\colon}
\newcommand{\tens}{\otimes}
\newcommand{\id}{\mathrm{id}}
\newcommand{\fs}[1]{[#1]}

\newcommand{\mc}{\mathcal}
\newcommand{\opp}{\mathrm{op}}
\newcommand{\N}{\mathbb{N}}
\newcommand{\C}{\mathbb{C}}

\newcommand{\Hom}{\mathrm{Hom}}
\newcommand{\Ob}[1]{\mathrm{Ob}(#1)}

\newcommand{\Comod}{\textbf{Comod}}
\newcommand{\Mod}{\textbf{Mod}}

\newcommand{\kk}{\Bbbk}
\newcommand{\uu}{\Bbb1}
\newcommand{\evl}{\text{ev}}
\newcommand{\evr}{\widetilde{\text{ev}}}
\newcommand{\coevl}{\text{coev}}
\newcommand{\coevr}{\widetilde{\text{coev}}}

\usepackage{geometry}
 \makeatletter
\@namedef{subjclassname@2020}{\textup{2020} Mathematics Subject Classification}
\makeatother

 \usepackage{etoolbox}
\makeatletter
\patchcmd{\@setaddresses}{\indent}{\noindent}{}{}
\patchcmd{\@setaddresses}{\indent}{\noindent}{}{}
\patchcmd{\@setaddresses}{\indent}{\noindent}{}{}
\patchcmd{\@setaddresses}{\indent}{\noindent}{}{}
\makeatother
 \title{On the braided Connes-Moscovici construction}
 \author{Ivan Bartulovi\'{c}}

\address{Universit\'{e} de Lille, Laboratoire Paul Painlev\'{e}
UMR CNRS 8524, F-59000 Lille, France}
\email{ivan.bartulovic@univ-lille.fr}
\begin{document}

\begin{abstract} In 1998, Connes and Moscovici defined the cyclic cohomology of Hopf algebras. In 2010, Khalkhali and Pourkia proposed a braided generalization: to any Hopf algebra $H$ in a braided category $\mc B$, they associate a paracocyclic object in $\mc B$. In this paper we explicitly compute the powers of the paracocyclic operator of this paracocyclic object. Also, we introduce twisted modular pairs in involution for $H$ and derive (co)cyclic modules from them.
Finally, we relate the paracocyclic object associated with $H$ to that associated with an $H$-module coalgebra via a categorical version of the Connes-Moscovici trace.
\end{abstract}
\maketitle
\setcounter{tocdepth}{1}
 \tableofcontents
\section{Introduction}
Cyclic (co)homology of algebras was introduced in the 1980s by Connes~\cite{Connesext, connes_non-commutative_1985} and Tsygan~\cite{tsygan1983homology} independently.
To any algebra is associated a cocyclic vector space (that is, a cocyclic object in the category of vector spaces) whose cohomology is called the cyclic cohomology of the algebra.
The notion of a (co)cyclic object in a category, introduced by Connes~\cite{Connesext}, is a generalization of the notion of a (co)simplicial object in that category.

Cyclic cohomology has been considered in various versions and generalizations.
In particular, in~\cite{connes_cyclic_1999}, Connes and Moscovici defined the Hopf cyclic cohomology by associating a cocyclic vector space to a Hopf algebra~$H$ over~$\C$ endowed with a modular pair in involution (that is, a pair~$(\delta,\sigma)$ where~$\delta \co H\to \C$ is a character and~$\sigma\in H$ is a grouplike element verifying the modular pair condition~$\delta(\sigma)=1$ and a certain involutivity condition).
Also, in~\cite{connes_hopf_1998}, they relate  the Hopf cyclic cohomology of~$H$ to the cyclic cohomology of an~$H$-module algebra by means of a trace map.

Braided monoidal categories were defined by Street and Joyal in the 1980s and appeared in many areas of mathematics such as low-dimensional topology and representation theo\-ry.
Several generalizations of cyclic (co)homology were introduced in the braided setting.
In this paper, we focus on the braided generalization of the Connes-Moscovici construction due to Khalkhali and Pourkia~\cite{khalkhalipourkia}.
Let~$H$ be a Hopf algebra in a braided monoidal category~$\mc B=(\mc B, \tens, \uu)$. A modular pair for~$H$ is a pair~$(\delta,\sigma)$, where~$\delta\co H\to \uu$ is an algebra morphism and~$\sigma \co \uu\to H$ is a coalgebra morphism such that~$\delta\sigma=\id_\uu$. In~\cite{khalkhalipourkia}, Khalkhali and Pourkia associate to any modular pair~$(\delta,\sigma)$ for~$H$ a paracocyclic object~$\textbf{CM}_\bullet(H,\delta,\sigma)=\{\textbf{CM}_n(H,\delta,\sigma)\}_{n\in \N}$ in~$\mc B$.
This object is cocyclic if it satisfies the cocyclicity condition:  for all~$n\in\N$, $$(\tau_{n}(\delta,\sigma))^{n+1}=\id_{H^{\tens n}},$$ where~$\tau_\bullet(\delta,\sigma)=\{\tau_n(\delta,\sigma)\}_{n\in \N}$ is the paracocyclic operator of $\textbf{CM}_\bullet(H,\delta,\sigma)$.
As already noticed in~\cite{connes_cyclic_1999} (corresponding to the case~$\mc B =\text{Vect}_{\mathbb{C}}$), verifying the cocyclicity condition (if true) is a rather technical task.
Khalkhali and Pourkia proved~\cite[Theorem 7.3]{khalkhalipourkia} that if~$(\delta,\sigma)$ is a so called braided modular pair in involution, then~$(\tau_2(\delta,\sigma))^3$ is equal to the square of the braiding of~$H$ with itself.
In particular, if~$\mc B$ is symmetric, then~$(\tau_2(\delta,\sigma))^3=\id_{H^{\tens 2}}$.
They also state a similar claim about~$(\tau_n(\delta,\sigma))^{n+1}$, which implies the cocyclicity condition when~$\mc B$ is symmetric, see \cite[Remark 7.4]{khalkhalipourkia}.

Our first main result is a complete com\-pu\-ta\-tion (by means of the Penrose graphical calculus) of the powers (up to~$n+1$) of the pa\-ra\-co\-cyclic o\-pe\-ra\-tor~$\tau_n(\delta,\sigma)$ associated with a mo\-du\-lar pair~$(\delta,\sigma)$ for $H$, see Theorem~\ref{powers}.
Next, assume that~$\mc B$ has a twist~$\theta$. We introduce the notion of a~$\theta$\textit{-twisted mo\-du\-lar pair in involution for $H$} (see Section \ref{modpair}) and prove (see Corollary \ref{n+1thpower=twist}) that if $(\delta,\sigma)$ is such a pair, then the associated paracocyclic operator satisfies the following twisted cocyclicity condition: for all $n\in \N$,
$$(\tau_n(\delta,\sigma))^{n+1}=\theta_{H^{\tens n}}.$$
When~$\mc B$ is further~$\kk$-linear, we derive (co)cyclic~$\kk$-modules from a~$\theta$-twisted modular pair~$(\delta,\sigma)$ by composing~$\textbf{CM}_\bullet(H,\delta,\sigma)$ with the functors~$\Hom_{\mc B}(\uu,-)$ and~$\Hom_{\mc B}(-,\uu)$, see Section~\ref{recall}.
Note that if~$\mc B$ is symmetric, then a braided modular pair in involution in the sense of~\cite{khalkhalipourkia} is a~$\id_{\mc B}$-twisted modular pair in involution, where~$\id_{\mc B}$ is the trivial twist of~$\mc B$, and so the associated paracocyclic operator satisfies the cocyclicity condition.

Let $H$ be a Hopf algebra in braided category $\mc B$ with a twist $\theta$. Our second main result is the construction of traces à la Connes-Moscovici.
More precisely, let $C$ be a $H$-module coalgebra, that is, a coalgebra in the category of right $H$-modules in $\mc B$.
Inspired by a construction of Akrami and Majid~\cite{cycliccocycles}, we associate to $C$ a paracocyclic object~$\textbf{C}_\bullet(C)$ in $\mc B$.
We introduce the notion of a $\delta$-\textit{invariant} $\sigma$-\textit{trace} for $C$ and derive from each such trace a natural transformation from $\textbf{CM}_\bullet(H,\delta,\sigma)$ to~$\textbf{C}_\bullet(C)$, see Theorem \ref{CMtrace}. This generalizes the standard Connes-Moscovici trace.
We provide examples of traces in the case where~$\mc B$ is a ribbon category and $H$ is its coend (see Section \ref{tracesdrinfeld}).

The paper is organized as follows.
In Section~\ref{prelimiHopf}, we review braided monoidal categories, Hopf algebras, and graphical calculus. Section~\ref{simplicialprelimi} is devoted to preliminaries on simplicial, paracyclic, and cyclic objects in a category.
In Sections~\ref{res1} and~\ref{res2}, we state our main results and their corollaries. Sections~\ref{proofmain} and~\ref{pfCMtrace} are devoted to the proofs of Theorems~\ref{powers} and~\ref{CMtrace}.
In~\nameref{appendix}, we provide an alternative proof (by using the Penrose graphical calculus) of the fact that the object~$\textbf{CM}_\bullet(H,\delta, \sigma)$ defined in \cite{khalkhalipourkia} is paracocyclic.

Throughout the paper,~$\kk$ denotes any commutative ring.
The class of objects of a category~$\mc B$ is denoted by~$\Ob{\mc B}$.

\subsection*{Acknowledgments} This work was supported by the Labex CEMPI (ANR-11-LABX-0007-01), by the Région Hauts-de-France, and by the FNS-ANR grant OChoTop (ANR-18-CE93-0002-01). The author is particularly thankful for the useful advices of his PhD advisor Alexis Virelizier.

\section{Preliminaries on monoidal categories and braided Hopf algebras} \label{prelimiHopf}
In this section, we recall some algebraic preliminaries used in the paper. We first re\-call some facts about braided monoidal categories and the Penrose gra\-phi\-cal cal\-cu\-lus. Next, we recall de\-fi\-ni\-tions of ca\-te\-go\-ri\-cal Hopf al\-ge\-bras and re\-la\-ted con\-cepts. We finish with a recall on pivotal categories and coends.  For a more com\-pre\-hen\-sive in\-tro\-duc\-tion, see~\cite{moncatstft}.

\subsection{Conventions} In what follows, we suppress in our formulas the associativity and unitality constraints of the monoidal category.
This does not lead to any ambiguity since Mac Lane's coherence theorem (see~\cite{maclane1963natural}) implies that all possible ways of inserting these constraints give the same results. We will denote by $\tens$ and $\uu$ the monoidal product and unit object of a monoidal category. For any objects~$X_1, \dots, X_n$ of a monoidal category with~$n\ge 2$, we set
$$X_1\tens X_2\tens \cdots \tens X_n=(\cdots((X_1\tens X_2)\tens X_3)\tens \cdots\tens X_{n-1})\tens X_n$$ and similarly for morphisms.
A~monoidal ca\-te\-go\-ry is $\kk$-linear if its Hom sets have a struc\-ture of a~$\kk$-mo\-dule such that the com\-po\-si\-tion and monoidal product of morphisms are~$\kk$-bi\-li\-near.

\subsection{Braided categories}
A \textit{braiding} of a monoidal category~$(\mc B, \tens, \uu)$ is a family $\tau=\{\tau_{X,Y}\co X\tens Y \to Y\tens X\}_{X,Y\in \Ob{\mc B}}$ of natural isomorphisms such that
\begin{align}
\tau_{X,Y\tens Z}&= (\id_Y \tens \tau_{X,Z})(\tau_{X,Y} \tens \id_Z) \text{ and }  \label{br1} \\
\tau_{X\tens Y, Z}&= (\tau_{X,Z} \tens \id_Y)(\id_X \tens \tau_{Y,Z}) \label{br2}
\end{align}
for all~$X,Y,Z \in \Ob{\mc B}$.
A \textit{braided category} is a monoidal category endowed with a braiding.

A braiding $\tau$ of $\mc B$ is \textit{symmetric} if for all $X,Y \in \Ob{\mc B}$,
\[\tau_{Y,X}\tau_{X,Y}=\id_{X\tens Y}.\]
A \textit{symmetric category} is a  category endowed with a symmetric braiding.

\subsection{Twists for braided categories} \label{catswithtwist} A \emph{twist} for a braided monoidal category $\mc B$ is a natural isomorphism $\theta=\{\theta_X \co X \to X\}_{X \in \Ob{\mc B}}$ such that
\begin{equation} \label{twistcondition}\theta_{X\tens Y}= (\theta_X \tens \theta_Y)\tau_{Y,X}\tau_{X,Y} \end{equation}
holds for all $X,Y \in \Ob{\mc B}$.
Note that this implies that $\theta_{\uu}=\id_{\uu}$.  For example, any ribbon category (see Section \ref{ribbon}) has a canonical twist.

Note that when~$\mc B$ is symmetric, a twist for~$\mc B$ is nothing but a monoidal natural isomorphism of the identity functor~$\mc B \to \mc B$. In particular,~$\id_{\mc B} = \{\id_X\co X\to X\}_{X \in \Ob{\mc B}}$ is a twist for~$\mc B$.

By a \emph{braided category with a twist}, we mean a braided category endowed with a twist.

\subsection{Graphical calculus} \label{cg}
Throughout this paper, we will use the \emph{Penrose graphical calculus}. For a systematic treatment, one may consult~\cite{moncatstft}.
The diagrams are to be read from bottom to top.
In a monoidal category~$\mc B$, the diagrams are made of arcs colored by objects of~$\mc B$ and of boxes, colored by morphisms of~$\mc B$.
Arcs colored by~$\uu$ may be omitted in the pictures.
The identity morphism of an object~$X$, a morphism~$f\co X\to Y$ in~$\mc B$, and its composition with a morphism~$g\co Y\to Z$ in~$\mc B$ are represented respectively as
\[
\id_X=\,
\psfrag{X}{$X$}
\rsdraw{0.45}{0.75}{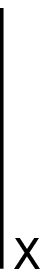}
\;, \quad
f=\,
\psfrag{X}{$X$}
\psfrag{Y}{$Y$}
\psfrag{f}{\hspace{-0.1cm}$f$}
\rsdraw{0.45}{0.75}{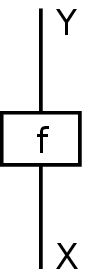}
\;,\quad \text{and} \quad
gf=\,
\psfrag{X}{$X$}
\psfrag{Y}{$Y$}
\psfrag{Z}{\hspace{-0.05cm}$Z$}
\psfrag{f}{\hspace{-0.1cm}$f$}
\psfrag{g}{\hspace{-0.03cm}$g$}
\rsdraw{0.45}{0.75}{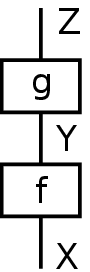}
\;.\]
The tensor product of two morphisms~$f\co X\to Y$ and~$g\co U \to V$ is represented by placing a picture of~$f$ to the left of the picture of~$g$:
\[
f\tens g
=
\,
\psfrag{X}{$X$}
\psfrag{Y}{$Y$}
\psfrag{U}{$U$}
\psfrag{V}{$V$}
\psfrag{f}{\hspace{-0.1cm}$f$}
\psfrag{g}{$g$}
\rsdraw{0.45}{0.75}{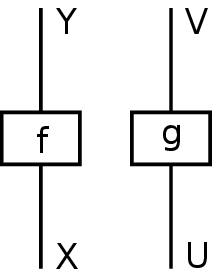}
\;.
\]
Any diagram represents a morphism. For example, the diagram
$$\,
\psfrag{X}{\hspace{-0.1cm}$X$}
\psfrag{Y}{$Y$}
\psfrag{T}{$T$}
\psfrag{U}{$U$}
\psfrag{c}{$c$}
\psfrag{f}{\hspace{-0.1cm}$f$}
\psfrag{h}{\hspace{-0.05cm}$h$}
\psfrag{V}{$V$}
\psfrag{g}{$g$}
\psfrag{Z}{\hspace{-0.14cm}$Z$}
\rsdraw{0.45}{0.75}{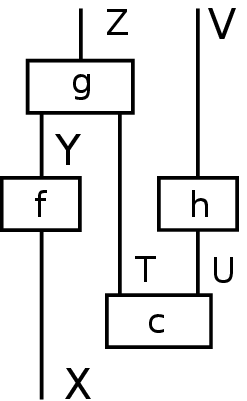}
\;$$
represents $(g\tens \id_V)(f\tens\id_T\tens h)(\id_X\tens c) \co X \to Z\tens V$.
The morphism associated to a diagram depends only on the isotopy class of the diagram representing it.
For example, the following \emph{level-exchange property}:
\[
\, \psfrag{X}{$X$}
\psfrag{Y}{$Y$}
\psfrag{U}{$U$}
\psfrag{V}{$V$}
\psfrag{f}{$f$}
\psfrag{g}{$g$}
\rsdraw{0.45}{0.75}{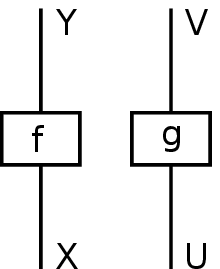}\;= \, \psfrag{X}{$X$}
\psfrag{Y}{$Y$}
\psfrag{U}{$U$}
\psfrag{V}{$V$}
\psfrag{f}{$f$}
\psfrag{g}{$g$}
\rsdraw{0.45}{0.75}{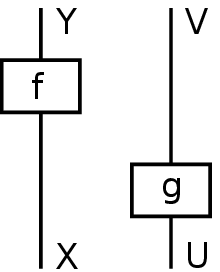}\;=\, \psfrag{X}{$X$}
\psfrag{Y}{$Y$}
\psfrag{U}{$U$}
\psfrag{V}{$V$}
\psfrag{f}{$f$}
\psfrag{g}{$g$}
\rsdraw{0.45}{0.75}{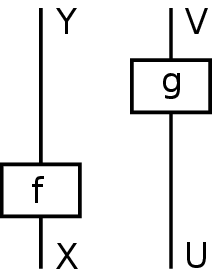}\;,\]
reflects the formula \[f\tens g =(f\tens \id_V)(\id_X \tens g)=(\id_Y \tens g)(f\tens \id_U).\]
When $\mc B$ is braided with braiding $\tau$, we depict
\[
\tau_{X,Y}=
\,
\psfrag{X}{$X$}
\psfrag{Y}{$Y$}
\rsdraw{0.45}{0.75}{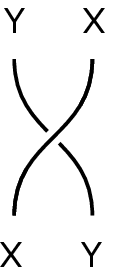}
\; \quad \text{and} \quad
\tau_{X,Y}^{-1}=\,
\psfrag{X}{$X$}
\psfrag{Y}{$Y$}
\rsdraw{0.45}{0.75}{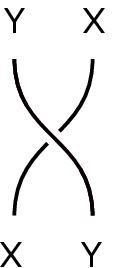}
\;.
\]
Axioms~\eqref{br1} and~\eqref{br2} for $\tau$ say that for all $X,Y,Z \in \Ob{\mc B}$,
$$
\,
\psfrag{X}{$X$}
\psfrag{Y}{\hspace{-0.5cm}$Y\tens Z$}
\rsdraw{0.45}{0.75}{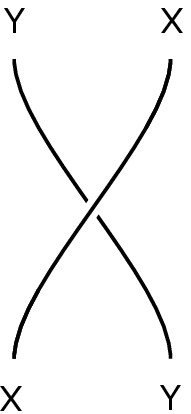}
\; = \,
\psfrag{X}{$X$}
\psfrag{Y}{$Y$}
\psfrag{Z}{$Z$}
\rsdraw{0.45}{0.75}{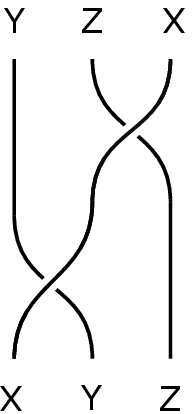}
\; \quad \text{and} \quad \,
\psfrag{X}{\hspace{-0.5cm}$X\tens Y$}
\psfrag{Y}{$Z$}
\rsdraw{0.45}{0.75}{brcgaxiom}
\; = \,
\psfrag{X}{$X$}
\psfrag{Y}{$Y$}
\psfrag{Z}{$Z$}
\rsdraw{0.45}{0.75}{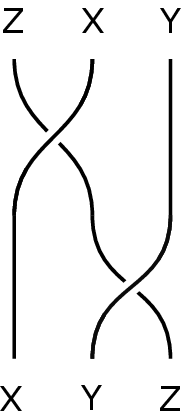}
\;.
$$
Naturality of the braiding and the level-exchange property imply that for any two morphisms~$f \co X\to Y$ and~$g\co U \to V$ in~$\mc B$,
$$\, \psfrag{X}{$X$}
\psfrag{Y}{$Y$}
\psfrag{U}{$U$}
\psfrag{V}{$V$}
\psfrag{f}{$f$}
\psfrag{g}{$g$}
\rsdraw{0.45}{0.75}{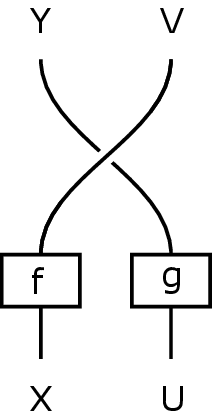}\; = \, \psfrag{X}{$X$}
\psfrag{Y}{$Y$}
\psfrag{U}{$U$}
\psfrag{V}{$V$}
\psfrag{f}{$g$}
\psfrag{g}{$f$}
\rsdraw{0.45}{0.75}{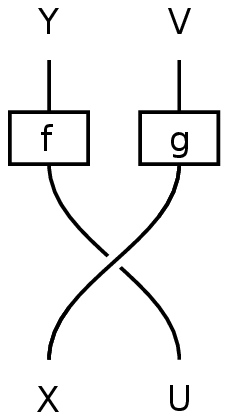}\; = \, \psfrag{X}{$X$}
\psfrag{Y}{$Y$}
\psfrag{U}{$U$}
\psfrag{V}{$V$}
\psfrag{f}{$f$}
\psfrag{g}{$g$}
\rsdraw{0.45}{0.75}{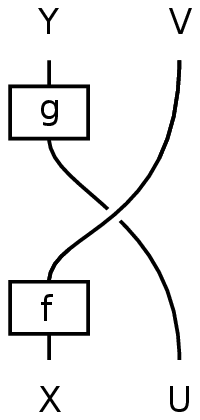}\; = \, \psfrag{X}{$X$}
\psfrag{Y}{$Y$}
\psfrag{U}{$U$}
\psfrag{V}{$V$}
\psfrag{f}{$f$}
\psfrag{g}{$g$}
\rsdraw{0.45}{0.75}{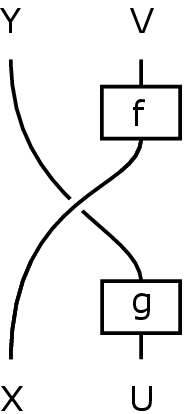}\;.$$
When $\mc B$ is braided with a twist $\theta=\{\theta_X\co X\to X\}_{X\in \Ob{\mc B}}$, we denote the twist by
$$\theta_X = \, \psfrag{X}{$X$}
\rsdraw{0.45}{0.75}{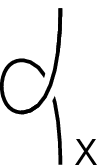}\;.$$
Axiom \eqref{twistcondition} for $\theta$ gives that for any $X,Y \in \Ob{\mc B}$,
$$
\, \psfrag{X}{$X\tens Y$}
\rsdraw{0.45}{0.75}{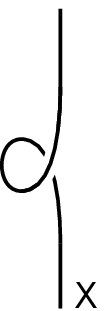}\; \quad \hspace{0.15cm}= \hspace{0.2cm} \, \psfrag{X}{$X$}
\psfrag{U}{$Y$}
\rsdraw{0.45}{0.75}{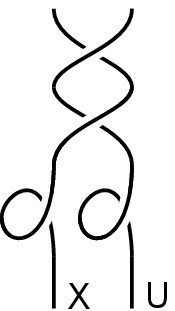}\;.
$$

\subsection{Categorical algebras}
An~\textit{algebra} in a monoidal category~$\mc B$ is a triple~$(A,m,u)$, where~$A$ is an object of~$\mc B$,~$m\co A\tens A \to A$ and~$u \co \uu \to A$ are morphisms in~$\mc B$, called~\textit{multiplication} and~\textit{unit} respectively, which satisfy the associativity and unitality axioms:
\[
m(m\tens \id_A)=m(\id_A \tens m) \quad \text{and} \quad m(u \tens \id_A)= \id_A= m(\id_A\tens u).
\]
The multiplication and unit are depicted by
\[m=
\rsdraw{0.45}{0.75}{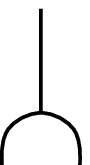}
 \quad \text{and} \quad
u=
\rsdraw{0.45}{0.75}{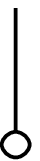},\]
so that the associativity and unitality axioms rewrite graphically as
\[
\rsdraw{0.45}{0.75}{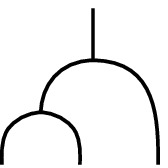}
=
\rsdraw{0.45}{0.75}{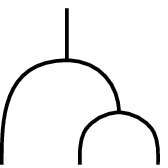}
 \quad \text{and} \quad
\rsdraw{0.45}{0.75}{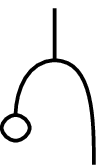}=
\rsdraw{0.45}{0.75}{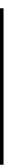}=
\rsdraw{0.45}{0.75}{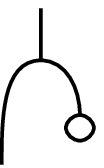}. \]
Here, it is understood that the arcs are colored by the underlying object of the algebra.

An \textit{algebra morphism} between algebras~$(A,m,u)$ and~$(A',m',u')$ in a monoidal category~$\mc B$ is a
morphism~$f\co A \to A'$ in~$\mc B$ such that~$fm = m'(f\tens f)$ and~$fu = u'$. The latter conditions are depicted by
$$
\,
\psfrag{f}{\hspace{-0.1cm}$f$}
\rsdraw{0.45}{0.75}{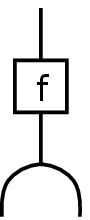}\;=\,
\psfrag{f}{\hspace{-0.1cm}$f$}
\rsdraw{0.45}{0.75}{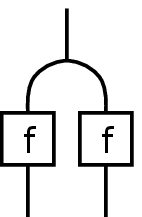}\; \quad \text{and} \quad \,
\psfrag{f}{\hspace{-0.1cm}$f$}
\rsdraw{0.45}{0.75}{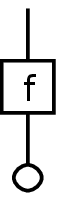}\;=\,
\psfrag{f}{\hspace{-0.1cm}$f$}
\rsdraw{0.45}{0.75}{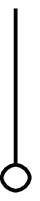}\;.$$

\subsection{Categorical coalgebras}A~\textit{coalgebra} in a mo\-noi\-dal ca\-te\-go\-ry~$\mc B$ is gi\-ven by a triple $(C, \Delta, \varepsilon)$, where~$C$ is an object of~$\mc B$,~$\Delta \co C\to C\tens C$ and~$\varepsilon\co C \to \uu$ are morphisms in~$\mc B$, called~\textit{comultiplication} and~\textit{counit} respectively, which satisfy the coassociativity and counitality axioms:
\[
(\Delta\tens \id_C)\Delta= (\id_C \tens \Delta)\Delta \quad \text{and} \quad  (\id_C \tens \varepsilon)\Delta=\id_C= (\varepsilon \tens \id_C)\Delta.\]
The comultiplication and counit are depicted by
\[
\Delta=\,
\psfrag{C}{$C$}
\rsdraw{0.45}{0.75}{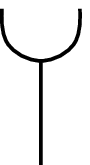}
\; \quad \text{and} \quad
\varepsilon=\,
\psfrag{C}{$C$}
\rsdraw{0.45}{0.75}{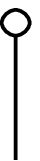}
\;,\]
so that the coassociativity and counitality axioms rewrite graphically as
\[
\rsdraw{0.45}{0.75}{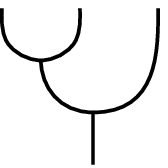}=
\rsdraw{0.45}{0.75}{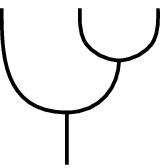} \quad \text{and} \quad
\rsdraw{0.45}{0.75}{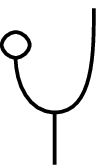}=
\rsdraw{0.45}{0.75}{unitalitygraphic2iponolab}=
\rsdraw{0.45}{0.75}{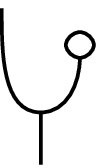}. \]

A \textit{coalgebra morphism} between coalgebras~$(C,\Delta,\varepsilon)$ and~$(C',\Delta',\varepsilon')$ in a monoidal category~$\mc B$ is a
morphism~$f\co C \to C'$ in~$\mc B$ such that~$\Delta' f = (f\tens f)\Delta$ and~$\varepsilon f = \varepsilon'$.
\subsection{Graphical calculus and iterated (co)multiplications} Let~$(A,m,u)$ and~$(C,\Delta, \varepsilon)$ be an algebra and a coalgebra in a monoidal category $\mc B$.
For any~$n\in \N$, we define the~\textit{$n$-th multiplication}~$m_n \co  A^{\tens n} \to A$ and the \textit{$n$-th comultiplication} $\Delta_n \co C\to C^{\tens n}$ inductively by:
$$m_{0}= u, \quad m_{n+1}=m(\id_A\tens m_{n}), \quad  \Delta_0= \varepsilon, \quad \text{and} \quad \Delta_{n+1}=(\id_C\tens \Delta_{n})\Delta.$$
For $n\ge 1$, we depict them as
$$m_n=\underbrace{\,
\psfrag{A}{$A$}
\rsdraw{0.45}{0.75}{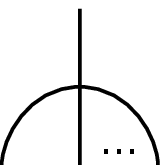}\;
}_{n \text{  times}} \quad \text{and} \quad \Delta_n=\overbrace{\,
\psfrag{A}{$C$}
\rsdraw{0.45}{0.75}{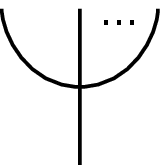}\;}^{n \text{  times}}.$$
The (co)associativity and (co)unitality of $m$ and $\Delta$ imply that
$$m_{k+1}(m_{n_0} \tens \cdots \tens m_{n_k})=m_{n_0+\cdots+ n_k} \quad \text{and} \quad (\Delta_{n_0} \tens \cdots \tens \Delta_{n_k})\Delta_{k+1}=\Delta_{n_0+\cdots+ n_k}$$
for all $k\in \N$ and $n_0,\dots,n_k \in \N$.
For example,
$$m_4=\,
\psfrag{A}{$A$}
\rsdraw{0.45}{0.75}{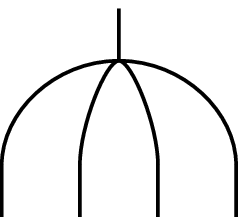}\;=\,
\psfrag{A}{$A$}
\rsdraw{0.45}{0.75}{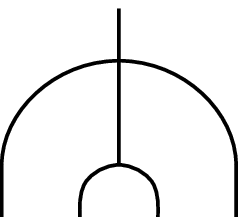}\;=\,
\psfrag{A}{$A$}
\rsdraw{0.45}{0.75}{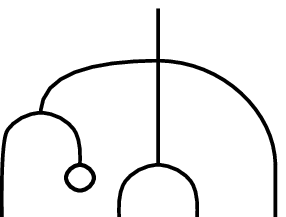}\; \quad \text{and} \quad \Delta_3= \,
\psfrag{C}{$C$}
\rsdraw{0.45}{0.75}{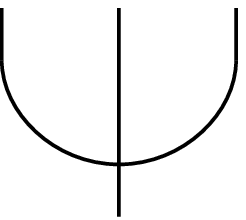}
\;=
\,\psfrag{C}{$C$}
\rsdraw{0.45}{0.75}{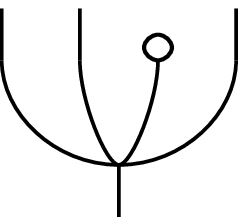}\;
$$

\subsection{Categorical bialgebras}
Let~$\mc B$ be a braided monoidal category.
A \textit{bialgebra} in~$\mc B$ is a quintuple~$(A,m,u, \Delta, \varepsilon)$ such that~$(A,m,u)$ is an algebra in~$\mc B$,~$(A,\Delta, \varepsilon)$ is a coalgebra in~$\mc B$, and the following compatibility relations hold:
\[\Delta m= (m\tens m)(\id_A\tens \tau_{A,A} \tens \id_A)(\Delta\tens \Delta), \quad \Delta u = u\tens u, \quad \varepsilon m = \varepsilon\tens \varepsilon, \quad \text{and} \quad  \varepsilon u= \id_{\uu}.\]
Graphically, these rewrite as
\[
\rsdraw{0.45}{0.75}{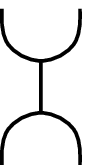} =
\rsdraw{0.45}{0.75}{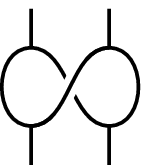}, \qquad
\rsdraw{0.45}{0.75}{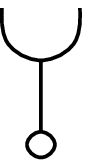} =
\rsdraw{0.45}{0.75}{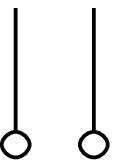}, \qquad
\rsdraw{0.45}{0.75}{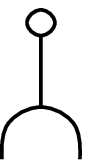} =
\rsdraw{0.45}{0.75}{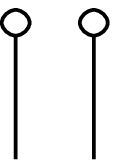},
\qquad \text{and} \qquad
\rsdraw{0.45}{0.75}{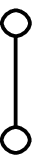}=\emptyset.
\]

A \textit{bialgebra morphism} between two bialgebras~$A$ et~$A'$ is a morphism~$A \to A'$ in $\mc B$ which is both an algebra and a coalgebra morphism.

\subsection{Categorical Hopf algebras}\label{hopfcat}
A~\textit{Hopf algebra} in~$\mc B$ is a sextuple~$(H,m,u, \Delta, \varepsilon, S)$, where~$(H,m,u, \Delta, \varepsilon)$ is a bialgebra in~$\mc B$ and~$S\co H \to H$ is an isomorphism in~$\mc B$, called the \textit{antipode},  which satisfies
\[m(S\tens \id_H)\Delta=u \epsilon = m(\id_H \tens S)\Delta.\]
The antipode and its inverse are depicted by
\[S=
\rsdraw{0.45}{0.75}{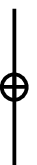}
 \quad \text{and} \quad S^{-1}=
\rsdraw{0.45}{0.75}{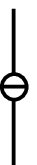}.
 \]
Graphically, the antipode axiom is rewritten as
\[
\rsdraw{0.45}{0.75}{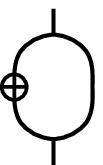}=
\rsdraw{0.45}{0.75}{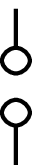}=
\rsdraw{0.45}{0.75}{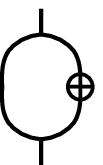}.\]
A useful feature of antipodes is that it is \textit{anti-multiplicative}:
\[
\rsdraw{0.45}{0.75}{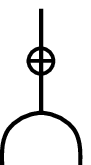}=
\rsdraw{0.45}{0.75}{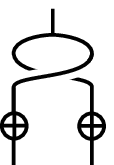} \quad \text{and} \quad
\rsdraw{0.45}{0.75}{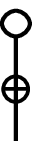}=
\rsdraw{0.45}{0.75}{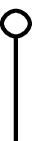},\]
and \textit{anti-comultiplicative}:
\[\,
\psfrag{b}{$H$}
\psfrag{v}{$\tilde{S}$}
\rsdraw{0.55}{0.75}{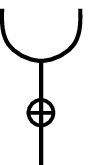}\;=\,
\psfrag{b}{$H$}
\psfrag{v}{$\tilde{S}$}
\rsdraw{0.55}{0.75}{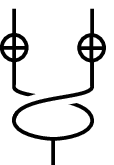}\; \quad \text{and} \quad \,
\psfrag{b}{$H$}
\rsdraw{0.45}{0.75}{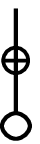}\;=\,
\psfrag{b}{$H$}
\rsdraw{0.45}{0.75}{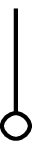}\;.\]

A \textit{Hopf algebra morphism} between two Hopf algebras is a bialgebra morphism between them.

\subsection{Categorical modules} Let~$(A,m,u)$ an algebra in a monoidal category~$\mc B$.
A \emph{left~$A$-module} in $\mc B$ is a pair~$(M,r)$, where~$r \co A\tens M \to M$ is a morphism in~$\mc B$, called~\emph{the action of~$A$ on~$M$}, which satisfies
$$r(m\tens \id_M)=r(\id_A \tens r) \quad \text{and} \quad r(u\tens \id_M)=\id_M.$$
Graphically, the action $r\co A\tens M \to M$ is denoted by
$$r=\,
\psfrag{u}{\hspace{-0.2cm}$A$}
\psfrag{A}{$M$}
\rsdraw{0.45}{0.75}{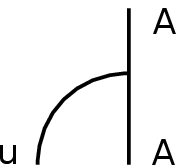}\;,$$
so that the axioms of a left $A$-module rewrite as
$$\,
\psfrag{u}{\hspace{-0.2cm}$A$}
\psfrag{A}{\hspace{-0.15cm}$M$}
\rsdraw{0.45}{0.75}{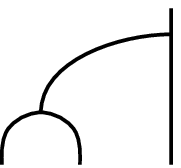}\;=\,
\psfrag{u}{\hspace{-0.2cm}$A$}
\psfrag{A}{\hspace{-0.15cm}$M$}
\rsdraw{0.45}{0.75}{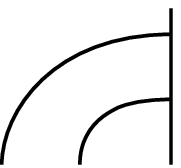}\; \quad \text{and} \quad \psfrag{u}{\hspace{-0.05cm}$A$}
\psfrag{A}{\hspace{-0.15cm}$M$}
\rsdraw{0.45}{0.75}{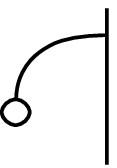}\;=\,
\psfrag{u}{\hspace{-0.2cm}$A$}
\psfrag{A}{\hspace{-0.15cm}$M$}
\rsdraw{0.45}{0.75}{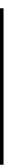}\;. $$

A morphism~$f\co (M,r) \to(M',r')$ between two left~$A$-modules~$(M,r)$ and~$(M',r')$ is a morphism~$f\co M \to M'$ in~$\mc B$ such that~$fr=r'(\id_A\tens f)$.
With composition inherited from~$\mc B$, left~$A$-modules and morphisms between them form a category~$_{A}{\Mod}$.

When $\mc B$ is braided and~$A$ is a bialgebra in $\mc B$, the category~$_{A}{\Mod}$ is monoidal: the unit object of~$_{A}{\Mod}$ is the pair~$(\uu,\varepsilon)$, the monoidal product of two left~$A$-modules~$(M,r)$ and~$(M',r')$ is given by the pair~$(M\tens M', s)$, where $$s= (r\tens r')(\id_A\tens \tau_{A,M} \tens \id_{M'})(\Delta \tens \id_{M\tens M'})=\,
\psfrag{u}{\hspace{-0.2cm}$A$}
\psfrag{B}{\hspace{-0.2cm}$M$}
\psfrag{A}{\hspace{-0.15cm}$M'$}
\rsdraw{0.45}{0.75}{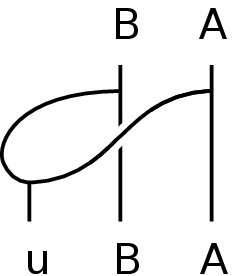}\;,$$
and the monoidal product of morphisms is inherited from $\mc B$.
\subsection{Categorical comodules} Let~$(C,\Delta,\varepsilon)$ a coalgebra in a monoidal category~$\mc B$.
A \emph{left~$C$-comodule} in $\mc B$ is a pair~$(N,\gamma)$, where~$\gamma \co N \to C\tens N$ is a morphism in~$\mc B$, called~\emph{the coaction of~$C$ on~$N$}, which satisfies
$$(\Delta \tens \id_N)\gamma=(\id_C \tens \gamma)\gamma \quad \text{and} \quad (\varepsilon \tens \id_N)\gamma=\id_N.$$
Graphically, the coaction~$\gamma\co N \to C\tens N$ is denoted by
$$\gamma=\,
\psfrag{u}{$C$}
\psfrag{A}{$N$}
\rsdraw{0.45}{0.75}{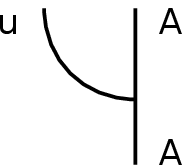}\;,$$
so that the axioms of a left $A$-comodule rewrite as
$$\,
\psfrag{u}{\hspace{-0.05cm}$C$}
\psfrag{A}{\hspace{-0.1cm}$N$}
\rsdraw{0.45}{0.75}{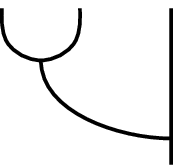}\;=\,
\psfrag{u}{\hspace{-0.05cm}$C$}
\psfrag{A}{\hspace{-0.1cm}$N$}
\rsdraw{0.45}{0.75}{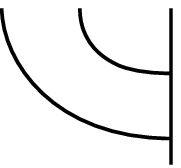}\; \quad \text{and} \quad \psfrag{u}{\hspace{-0.05cm}$C$}
\psfrag{A}{\hspace{-0.1cm}$N$}
\rsdraw{0.45}{0.75}{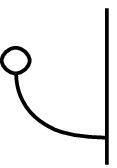}\;=\,
\psfrag{A}{\hspace{-0.1cm}$N$}
\rsdraw{0.45}{0.75}{action4nolab}\;. $$

A morphism $f\co (N,\gamma) \to(N',\gamma')$ between two left $A$-comodules $(N,\gamma)$ and $(N',\gamma')$ is a morphism $f\co N \to N'$ in $\mc B$ such that $\gamma'f=(\id_C\tens f)\gamma$.
With composition inhe\-rited from~$\mc B$, left~$C$-comodules and morphisms between them form a category~$_{C}{\Comod}$.

When $\mc B$ is braided and~$C$ is a bialgebra in $\mc B$, the category~$_{C}{\Comod}$ is monoidal:
the unit object of~$_{C}{\Comod}$ is the pair~$(\uu,u)$, the monoidal product of two left~$C$-comodules~$(N,\gamma)$ and~$(N',\gamma')$ is given by the pair~$(N\tens N', \delta)$, where $$\delta= (m \tens \id_{N\tens N'})(\id_C\tens \tau_{N,C} \tens \id_{N'})(\gamma\tens \gamma')=\,
\psfrag{u}{\hspace{-0.05cm}$C$}
\psfrag{B}{\hspace{-0.2cm}$N$}
\psfrag{A}{\hspace{-0.15cm}$N'$}
\rsdraw{0.45}{0.75}{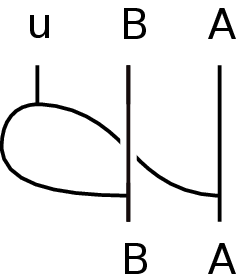}\;,$$
and the monoidal product of morphisms is inherited from $\mc B$.

\subsection{Diagonal actions} \label{diagonalacts} Let $H$ be a bialgebra in a braided category $\mc B$. The \emph{left diagonal action} of $H$ on $H^{\tens n}$ is defined inductively by
\[ \,
\psfrag{u}{$\uu$}
\psfrag{A}{$H$}
\rsdraw{0.45}{0.75}{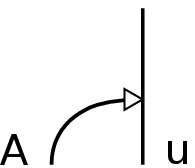}\; = \,
\psfrag{C}{$H$}
\rsdraw{0.45}{0.75}{counitCMnolab}
\;  \quad \text{ and } \quad
\,
\psfrag{u}{\hspace{-0.1cm}$H^{\tens n}$}
\psfrag{A}{$H$}
\rsdraw{0.45}{0.75}{leftH0actdiag2}\;\hspace{0.2cm}= \,
\psfrag{u}{$H^{\tens n-1}$}
\psfrag{A}{$H$}
\rsdraw{0.55}{0.75}{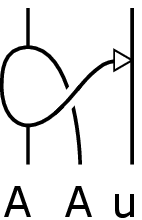}\;
\quad \text{ for } n\ge 1.
\]
Note that $$\,
\psfrag{u}{$H$}
\psfrag{A}{$H$}
\rsdraw{0.45}{0.75}{leftH0actdiag2}\; = \,
\psfrag{A}{$H$}
\rsdraw{0.45}{0.75}{multiplicationCMnolab}
\; \quad \text{and} \quad
 \,
\psfrag{u}{\hspace{-0.1cm}$H^{\tens n}$}
\psfrag{A}{$H$}
\rsdraw{0.45}{0.75}{leftH0actdiag2}\; \hspace{0.2cm}=
\,
\psfrag{A}{$H$}
\psfrag{a}{$H$}
\psfrag{e}{$H$}
\psfrag{n}{$H$}
\rsdraw{0.55}{0.75}{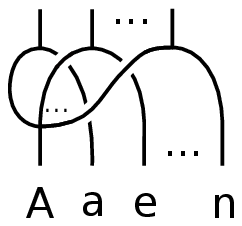}\; \quad \text{ for } n\ge 2.$$
Similarly, the \emph{right diagonal action} of $H$ on $H^{\tens n}$ is defined inductively by
\[\,
\psfrag{u}{\hspace{0.15cm}$\uu$}
\psfrag{A}{$H$}
\rsdraw{0.45}{0.75}{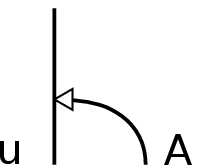}\; = \,
\psfrag{C}{$H$}
\rsdraw{0.45}{0.75}{counitCMnolab}
\;
 \quad \text{ and } \quad \,
\psfrag{u}{\hspace{-0.35cm}$H^{\tens n}$}
\psfrag{A}{$H$}
\rsdraw{0.45}{0.75}{rightH0actdiag2}\;=\quad \,
\psfrag{u}{\hspace{-0.65cm}$H^{\tens n-1}$}
\psfrag{A}{$H$}
\rsdraw{0.55}{0.75}{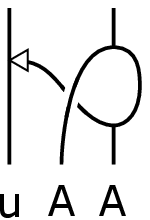}\; \quad \text{ for } n\ge 1.
\]
Note that
$$  \,
\psfrag{u}{$H$}
\psfrag{A}{$H$}
\rsdraw{0.45}{0.75}{rightH0actdiag2}\; =\,
\psfrag{A}{$H$}
\rsdraw{0.45}{0.75}{multiplicationCMnolab}
\;
\quad \text{and} \quad
\,
\psfrag{u}{\hspace{-0.35cm}$H^{\tens n}$}
\psfrag{A}{$H$}
\rsdraw{0.45}{0.75}{rightH0actdiag2}\;= \,
\psfrag{A}{$H$}
\psfrag{a}{$H$}
\psfrag{e}{$H$}
\psfrag{n}{$H$}
\rsdraw{0.55}{0.75}{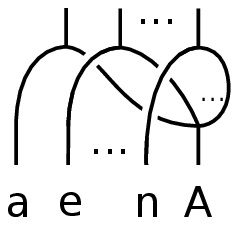}\;\quad \text{ for } n\ge 2.
$$
It follows from the definitions, that if $\sigma\co \uu \to H$ is a coalgebra morphism, then
$$\,
\psfrag{u}{\hspace{-0.3cm}$H^{\tens n}$}
\psfrag{c}{$\sigma$}
\psfrag{A}{$H$}
\rsdraw{0.55}{0.75}{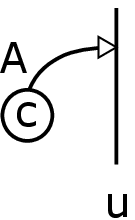}\; =
\underbrace{\,
\psfrag{u}{\hspace{-0.1cm}$H$}
\psfrag{c}{$\sigma$}
\rsdraw{0.55}{0.75}{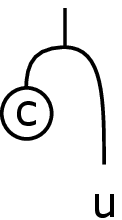}\; \cdots \,
\psfrag{u}{\hspace{-0.1cm}$H$}
\psfrag{c}{$\sigma$}
\rsdraw{0.55}{0.75}{diagsigma3}\;}_{n \text{ times}} \quad \text{and} \quad
\,
\psfrag{u}{\hspace{-0.3cm}$H^{\tens n}$}
\psfrag{c}{$\sigma$}
\psfrag{A}{$H$}
\rsdraw{0.55}{0.75}{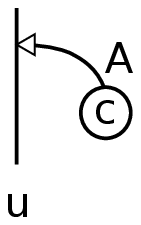}\; =
\underbrace{\,
\psfrag{u}{\hspace{-0.1cm}$H$}
\psfrag{c}{$\sigma$}
\rsdraw{0.55}{0.75}{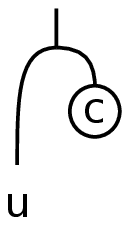}\; \cdots \,
\psfrag{u}{\hspace{-0.1cm}$H$}
\psfrag{c}{$\sigma$}
\rsdraw{0.55}{0.75}{diagsigma1}\;}_{n \text{ times}}.
$$

\subsection{Adjoint actions} \label{adjact} Let $H$ be a Hopf algebra in a braided category $\mc B$.
The \emph{left adjoint action} of $H$ on $H^{\tens n}$ is defined inductively by
\[\,
\psfrag{u}{\hspace{-0.1cm}$H$}
\psfrag{A}{$\uu$}
\rsdraw{0.45}{0.75}{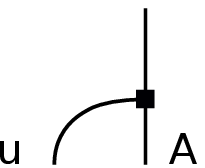}\;=
\,
\psfrag{C}{$H$}
\rsdraw{0.45}{0.75}{counitCMnolab}
\;
\qquad \text{and} \qquad
\,
\psfrag{u}{\hspace{-0.1cm}$H$}
\psfrag{A}{$H^{\tens n}$}
\rsdraw{0.45}{0.75}{leftadjact02}\; \hspace{0.1cm}= \,
\psfrag{u}{\hspace{-0.1cm}$H$}
\psfrag{A}{$H$}
\psfrag{H}{\hspace{-0.2cm}$H^{\tens n-1}$}
\rsdraw{0.55}{0.75}{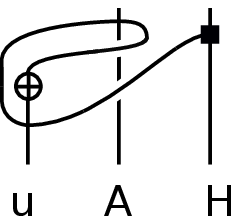}\; \qquad \text{ for } n\ge 1.
\]
Similarly, the \emph{right adjoint action} of $H$ on $H^{\tens n}$ is defined inductively by
$$
\,
\psfrag{u}{\hspace{-0.1cm}$\uu$}
\psfrag{A}{$H$}
\rsdraw{0.45}{0.75}{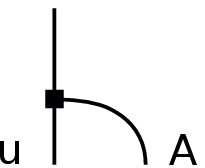}\;=\,
\psfrag{C}{$H$}
\rsdraw{0.45}{0.75}{counitCMnolab}
\; \qquad \text{ and } \qquad  \,
\psfrag{u}{\hspace{-0.5cm}$H^{\tens n}$}
\psfrag{A}{$H$}
\rsdraw{0.45}{0.75}{rightadjact02}\;= \,
\psfrag{A}{$H$}
\psfrag{H}{\hspace{-0.1cm}$H$}
\psfrag{u}{\hspace{-0.4cm}$H^{\tens n-1}$}
\rsdraw{0.55}{0.75}{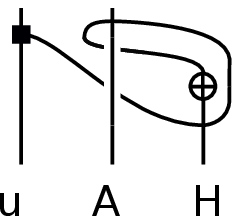}\; \qquad \text{ for } n\ge 1.
$$
Note that $$
\,
\psfrag{u}{\hspace{-0.1cm}$H$}
\psfrag{A}{$H$}
\rsdraw{0.45}{0.75}{leftadjact02}\;=
\,
\psfrag{u}{\hspace{-0.1cm}$H$}
\psfrag{A}{$H$}
\rsdraw{0.55}{0.75}{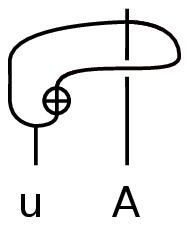}\;  \qquad \text{and} \qquad
\,
\psfrag{u}{\hspace{-0.2cm}$H$}
\psfrag{A}{$H$}
\rsdraw{0.45}{0.75}{rightadjact02}\;
=\,
\psfrag{u}{\hspace{-0.1cm}$H$}
\psfrag{A}{$H$}
\rsdraw{0.55}{0.75}{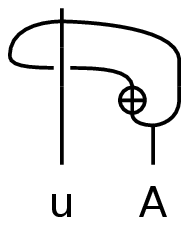}\;.
$$
It follows from the definition, that if~$\sigma\co \uu \to H$ is a coalgebra morphism, then
$$
\,
\psfrag{u}{\hspace{-0.3cm}$H^{\tens n}$}
\psfrag{c}{$\sigma$}
\psfrag{A}{$H$}
\psfrag{l}{\hspace{-0.05cm}$\delta$}
\rsdraw{0.55}{0.75}{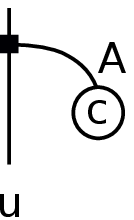}\; =
\underbrace{\,
\psfrag{u}{$H$}
\psfrag{c}{$\sigma$}
\psfrag{A}{$H$}
\psfrag{l}{\hspace{-0.05cm}$\delta$}
\rsdraw{0.55}{0.75}{adsigmaright}\;  \cdots \,
\psfrag{u}{$H$}
\psfrag{c}{$\sigma$}
\psfrag{A}{$H$}
\psfrag{l}{\hspace{-0.05cm}$\delta$}
\rsdraw{0.55}{0.75}{adsigmaright}\;}_{n \text{ times}}.$$
\subsection{Coadjoint coactions} \label{coadjcoact} Let $H$ be a Hopf algebra in a braided category $\mc B$.
The \emph{left coadjoint coaction} of $H$ on $H^{\tens n}$ is defined inductively by
\[
\,
\psfrag{u}{\hspace{-0.1cm}$H$}
\psfrag{A}{$\uu$}
\rsdraw{0.45}{0.75}{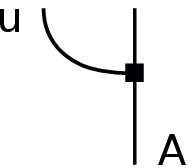}\;=\rsdraw{0.45}{0.75}{unitnolab}
\qquad \text{and} \qquad
\,
\psfrag{u}{\hspace{-0.1cm}$H$}
\psfrag{A}{$H^{\tens n}$}
\rsdraw{0.45}{0.75}{leftcoadjcoact02}\; \hspace{0.15 cm}
=
\,
\psfrag{u}{$H$}
\psfrag{A}{$H$}
\psfrag{H}{\hspace{-0.25cm}$H^{\tens n-1}$}
\rsdraw{0.55}{0.75}{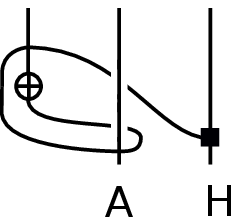}\; \qquad \text{ for } n\ge 1.
\]
Similarly, the \emph{right coadjoint coaction} of $H$ on $H^{\tens n}$ is defined inductively by
$$\,
\psfrag{u}{$\uu$}
\psfrag{A}{$H$}
\rsdraw{0.45}{0.75}{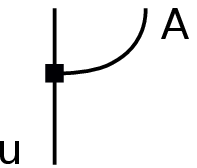}\; =\rsdraw{0.45}{0.75}{unitnolab} \qquad \text{and} \qquad \,
\psfrag{u}{\hspace{-0.45cm}$H^{\tens n}$}
\psfrag{A}{$H$}
\rsdraw{0.45}{0.75}{rightcoadjcoact02}\; = \,
\psfrag{A}{\hspace{-0.4cm}$H^{\tens n-1}$}
\psfrag{H}{\hspace{-0.15cm}$H$}
\rsdraw{0.55}{0.75}{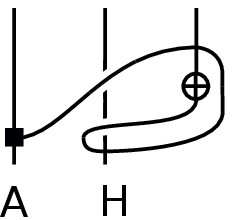}\; \qquad \text{ for } n\ge 1.$$
Note that
$$
\,
\psfrag{u}{\hspace{-0.1cm}$H$}
\psfrag{A}{$H$}
\rsdraw{0.45}{0.75}{leftcoadjcoact02}\; = \,
\psfrag{u}{$H$}
\psfrag{A}{$H$}
\rsdraw{0.45}{0.75}{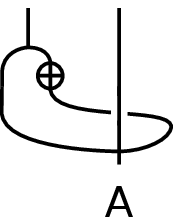}\;
 \qquad \text{and} \qquad \,
\psfrag{h}{$H$}
\psfrag{A}{$H$}
\rsdraw{0.45}{0.75}{rightcoadjcoact02}\;=\,
\psfrag{u}{\hspace{-0.1cm}$H$}
\psfrag{A}{$H$}
\rsdraw{0.45}{0.75}{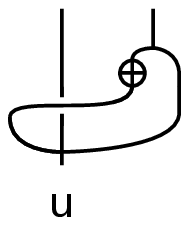}\;.
$$
It follows from the definition, that if~$\delta\co H\to \uu$ is an algebra morphism, then
$$\,
\psfrag{A}{\hspace{-0.3cm}$H^{\tens n}$}
\psfrag{c}{$\sigma$}
\psfrag{u}{\hspace{0.55cm}$H$}
\psfrag{l}{\hspace{-0.05cm}$\delta$}
\rsdraw{0.55}{0.75}{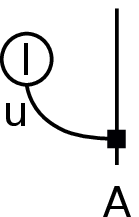}\; =
\underbrace{\,
\psfrag{A}{$H$}
\psfrag{c}{$\sigma$}
\psfrag{u}{\hspace{0.55cm}$H$}
\psfrag{l}{\hspace{-0.05cm}$\delta$}
\rsdraw{0.55}{0.75}{coaddelta}\;  \cdots \,
\psfrag{A}{$H$}
\psfrag{c}{$\sigma$}
\psfrag{u}{\hspace{0.55cm}$H$}
\psfrag{l}{\hspace{-0.05cm}$\delta$}
\rsdraw{0.55}{0.75}{coaddelta}\;}_{n \text{ times}}.$$

\subsection{Pivotal categories}
A \emph{pivotal} category is a monoidal category $\mc B$ such that each object $X$ of
 $\mc B$ has a \emph{dual} object $X^*$ and four morphisms
\begin{align*}
  &\evl_X \co X^* \tens X \to \uu, \quad \coevl_X \co \uu \to X\tens X^*, \\
  &\evr_X \co X \tens X^* \to \uu, \quad \coevr_X \co \uu \to X^* \tens X,
\end{align*}
satisfying some conditions. Briefly, these say that the associated left/right
dual functors coincide as monoidal functors (see \cite[Chapter 1]{moncatstft} for more details). The latter implies that the dual morphism $f^* \co Y^* \to X^*$ of a morphism $f\co X\to Y$ in $\mc B$ is computed by
\begin{align*}
  f^*&=(\id_{X^*} \tens \evr_Y )(\id_{X^*} \tens f \tens \id_{Y^*})(\coevr_X \tens \id_{Y^*})= \\
  & =(\evl_Y \tens \id_{{X^*}})(\id_{{Y^*}} \tens f \tens \id_{{X^*}})(\id_{{Y^*}} \tens \coevl_X).
\end{align*}
We extend the graphical calculus for monoidal categories (see Section~\ref{cg}) to pivotal ca\-te\-gories by orienting arcs. If an arc colored by~$X$ is oriented upwards, the represented object in source/target of corresponding morphism is~$X^*$.
For example,~$\id_{X}, \id_{X^*}$, and a morphism~$f\co X\tens Y^* \tens Z \to U \tens V^*$ are depicted by
\[
\id_{X}=
\,
\psfrag{X}{$X$}
\rsdraw{0.45}{0.75}{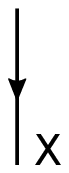}
\;,\quad
\id_{X^*}=
\,
\psfrag{X}{$X$}
\rsdraw{0.45}{0.75}{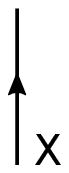}
\;
=
\,
\psfrag{X}{$X^*$}
\rsdraw{0.45}{0.75}{idX}
\;,
\quad \text{and} \quad
f=
\,
\psfrag{X}{$X$}
\psfrag{Y}{$Y$}
\psfrag{Z}{$Z$}
\psfrag{U}{$U$}
\psfrag{V}{$V$}
\psfrag{f}{\hspace{0.2cm}$f$}
\rsdraw{0.45}{0.75}{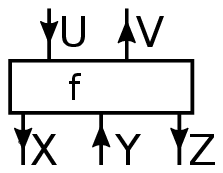}
\;. \]
The morphisms $\evl_X, \evr_X, \coevl_X$, and $\coevr_X$ are respectively depicted by
\[
\,
\psfrag{X}{$X$}
\rsdraw{0.45}{0.75}{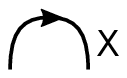}
\;, \quad
\,
\psfrag{X}{$X$}
\rsdraw{0.45}{0.75}{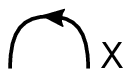}
\;, \quad
\,
\psfrag{X}{$X$}
\rsdraw{0.45}{0.75}{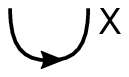}
\;,  \quad \text{and} \quad
\,
\psfrag{X}{$X$}
\rsdraw{0.45}{0.75}{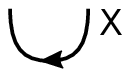}
\;.
\]

\subsection{Left and right twists} \label{ribbon} Let $\mc B$ be a braided pivotal category.
The \emph{left twist} of an object~$X$ of $\mc B$ is defined by
\[\theta^l_X= \, \psfrag{X}{$X$}
\rsdraw{0.45}{0.75}{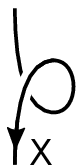}\;= (\id_X\tens \evr_X)(\tau_{X,X}\tens \id_{X^*})(\id_X\tens \coevl_X)\co X\to X,\]
while the \emph{right twist} of $X$ is defined by
\[\theta^r_X= \, \psfrag{X}{$X$}
\rsdraw{0.45}{0.75}{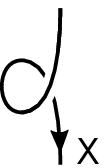}\;=(\evl_X\tens \id_{X})(\id_{X^*}\tens \tau_{X,X})(\coevr_X\tens \id_X)\co X\to X.\]
The left and the right twist are natural isomorphisms with inverses
$$
(\theta_X^l)^{-1}= \, \psfrag{X}{$X$}
\rsdraw{0.45}{0.75}{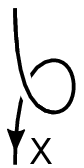}\; \quad \text{and} \quad (\theta_X^r)^{-1}= \, \psfrag{X}{$X$}
\rsdraw{0.45}{0.75}{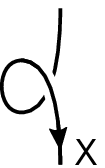}\;.$$
The left twist $\theta^l=\{\theta^l_X\co X\to X\}_{X\in\Ob{\mc B}}$ and the right twist $\theta^r=\{\theta^r_X\co X\to X\}_{X\in\Ob{\mc B}}$ are twists for $\mc B$ in the sense of Section \ref{catswithtwist}.

A \emph{ribbon} category is a braided pivotal category~$\mc B$ whose left and right twist coincide. Then~$\theta=\theta_l = \theta_r$ is called the \emph{twist} of $\mc B$.

\subsection{Coends}
Let~$\mc C$ and~$\mc D$ be any categories and~$F\co \mc C^{\opp} \times \mc C \to \mc D$ a functor.
A \textit{dinatural transformation} between~$F$ and an object $D$ in $\mc D$ is a function~$d$ that assigns to any object~$X$ in~$\mc C$ a morphism~$d_X\co F(X,X) \to D$ such that for all morphisms~$f\co X \to Y$ in $\mc C$ the following diagram commutes:
\[
\xymatrix@R+1pc@C+4pc
{F(Y,X)\ar[d]_{F(\id_Y,f)} \ar[r]^{F(f,\id_X)} & F(X,X) \ar[d]^{d_X} \\
F(Y,Y) \ar[r]_-{d_Y} &D.
}
\]

A \emph{coend} of a functor~$F\co \mc C^{\opp} \times \mc C \to \mc D$ is a pair~$(C,i)$ where~$C$ is an object of~$\mc D$ and~$i$ is a dinatural transformation from~$F$ to~$C$, which is universal among all dinatural transformations.
More precisely, for any dinatural transformation~$d$ from~$F$ to~$D$, there exists a unique morphism~$\varphi\co C\to D$ in~$\mc D$ such that~$d_X=\varphi i_X$ for all~$X \in \Ob{\mc C}$.
A coend~$(C,i)$ of a functor~$F$, if it exists, is unique up to a unique isomorphism commuting with the dinatural transformation.

\subsection{Coend of a pivotal category}\label{CoendRibbon}
Let $\mc B$ be a pivotal category. The \emph{coend} of~$\mc B$, if it exists, is the coend $(H,i)$ of the functor~$F\co \mc{B}^\opp \times \mc B \to \mc B$ defined by
\[F(X, Y)=X^{*}\tens Y \quad \text{and} \quad F(f,g)=f^*\tens g.\]
We depict the universal dinatural transformation $i=\{i_X\co X^*\tens X \to H\}_{X\in \Ob{\mc B}}$ as
\[i_X=
\,
\psfrag{X}{$X$}
\psfrag{C}{$H$}
\psfrag{G}{$X$}
\rsdraw{0.45}{0.75}{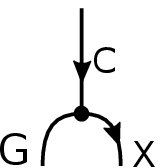}
\;.
\]
Note that $H$ is a co\-al\-ge\-bra in $\mc B$ with co\-mul\-ti\-pli\-ca\-tion $\Delta \co H\to H\tens H$ and counit $\varepsilon\co H\to \uu$, which are unique morphisms such that, for all $X\in \Ob{\mc B}$,
$$ \,
\psfrag{X}{$X$}
\psfrag{Y}{$Y$}
\psfrag{C}{$H$}
\psfrag{f}{\hspace{-0.15cm}$\Delta$}
\psfrag{G}{$X$}
\rsdraw{0.45}{0.75}{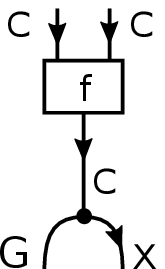}
\;=~~
\,
\psfrag{X}{$X$}
\psfrag{G}{$X$}
\psfrag{Y}{$Y$}
\psfrag{C}{$H$}
\rsdraw{0.45}{0.75}{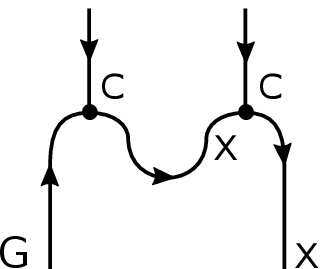}
\;
\qquad \text{and} \qquad
\,
\psfrag{g}{\hspace{0.03cm}$\varepsilon$}
\psfrag{X}{$X$}
\psfrag{C}{$H$}
\rsdraw{0.45}{0.75}{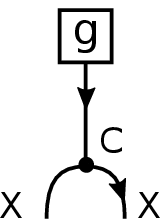}
\;
=~~
\,
\psfrag{X}{$X$}
\rsdraw{0.45}{0.75}{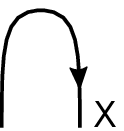}
\;.$$

The coalgebra~$(H,\Delta, \varepsilon)$ coacts on the objects in $\mc B$ via the~\emph{universal coaction} defined for any~$X\in \Ob{\mc B}$ by
$$\delta_X=(\id_X\tens i_X)\circ(\coevl_X\tens \id_X).$$
We will denote it graphically as
  $$\delta_X=\,
    \psfrag{X}{\hspace{-0.1cm}$X$}
    \psfrag{C}{$H$}
    \rsdraw{0.55}{0.75}{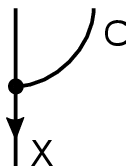}
  \;.$$
Note that~$\delta_H$ is the right coadjoint coaction of~$H$ on~$H$ (see Section~\ref{coadjcoact}).
If~$\mc B$ is braided, then~$H$ is a Hopf algebra in~$\mc B$. Its unit is~$u=\delta_{\uu} \co \uu \to H$ and its multiplication~$m \co H\tens H\to H$ and antipode~$S \co H\to H$ are characterized as follows: for all~$X,Y \in \Ob{\mc B}$,
$$\,
    \psfrag{X}{$X$}
    \psfrag{Y}{$Y$}
    \psfrag{f}{\hspace{-0.1cm}$m$}
    \psfrag{C}{$H$}
    \rsdraw{0.45}{0.75}{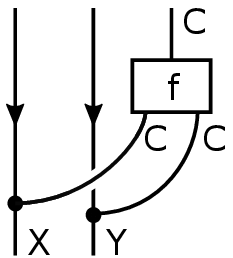}
  \; =
  \,
    \psfrag{XY}{$X\tens Y$}
    \psfrag{C}{$H$}
    \rsdraw{0.45}{0.75}{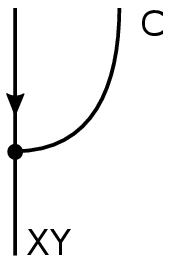}
  \; \qquad \text{and} \qquad \,
    \psfrag{X}{$X$}
    \psfrag{C}{$H$}
\psfrag{g}{\hspace{-0.03cm}$S$}
    \rsdraw{0.45}{0.75}{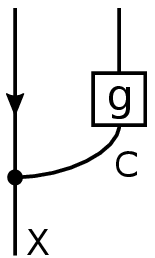}
  \; =\,
    \psfrag{X}{$X$}
    \psfrag{C}{$H$}
    \rsdraw{0.45}{0.75}{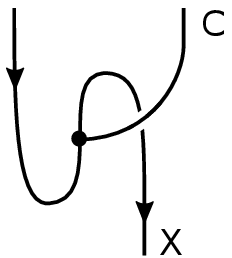}
  \;                           .$$
We refer to \cite[Chapter 6]{moncatstft} for details.

\section{Simplicial, paracyclic, and cyclic objects} \label{simplicialprelimi}
In this section we recall the notions of (co)simplicial, para(co)cyclic, and (co)cyclic objects in a category.

\subsection{The simplicial category} \label{simplicial}
The \emph{simplicial category}~$\Delta$ is defined as fol\-lows.
The ob\-jects of~$\Delta$ are the non\-ne\-ga\-tive in\-te\-gers~$n \in \N$.
A morphism~$n\to m$ in~$\Delta$ is an in\-creasing map bet\-ween sets~$\fs{n}=\{0,\dots,n\}$ and~$\fs{m}=\{0, \dots,m\}$.
For~$n\in \N^*$ and~$0\le i \le n$, the~$i$-th \emph{coface}~$\delta_i^n \co n-1 \to n$ is the unique increasing in\-jection from~$\fs{n-1}$ into~$\fs{n}$ which misses~$i \in\fs{n}$.
For~$n\in \N$ and~$0\le j \le n$, the~$j$-th~\emph{codegeneracy}~$\sigma_j^n \co n+1 \to n$ is the unique increasing sur\-jection from~$\fs{n+1}$ onto~$\fs{n}$ which sends both~$j$ and~$j+1$ to~$j$.

It is well known (see~\cite[Lemma 5.1]{MCLHomology}) that mor\-phisms in~$\Delta$ are ge\-ne\-rated by co\-fa\-ces~$\{\delta_i^n\}_{0\leq i \leq n, n\in \N^*}$ and codegeneracies~$\{\sigma_j^n\}_{0\leq j \leq n, n\in \N}$ subject to the simplicial relations~\eqref{sr}:
\begin{equation}
\tag{SR} \label{sr} \
\begin{aligned}
\delta_j \delta_i& = \delta_i \delta_{j-1} \text{\quad for }i<j, \\
\sigma_j\sigma_i&=\sigma_i\sigma_{j+1} \text{\quad for } i \leq j,
\\
\sigma_j\delta_i&= \begin{cases}
\delta_i \sigma_{j-1} & \text{for } i<j, \\
\id_{n} & \text{for } i=j,~~ i=j+1,\\
\delta_{i-1}\sigma_j & \text{for } i>j+1.
\end{cases}
\end{aligned}
\end{equation}

\subsection{The paracyclic category}\label{paracyccat}
The~\emph{paracyclic category}~$\Delta C_{\infty}$ is defined as follows.
The objects of~$\Delta C_{\infty}$ are the nonnegative integers~$n\in \N$.
The morphisms are generated by morphisms~$\{\delta_i^n\}_{n\in \N^*, 0\le i \le n}$, called \emph{cofaces}, morphisms~$\{\sigma_j^n\}_{n\in \N, 0\le j \le n }$, called \emph{codegeneracies}, and isomorphisms~$\{\tau_n\co n \to n\}_{n\in \N}$, called~\emph{paracocyclic operators}, satisfying the simplicial relations \eqref{sr} and the following~\emph{paracyclic compatibility relations} \eqref{pcr}:
\begin{equation}
\tag{PCR} \label{pcr} \
\begin{aligned}
\tau_n \delta_i &= \delta_{i-1} \tau_{n-1} \text{\quad for } 1 \leq i \leq n, \\
\tau_n\delta_0&=\delta_n, \\
\tau_n\sigma_i&=\sigma_{i-1}\tau_{n+1} \text{\quad for } 1 \leq i \leq n, \text{ and }\\
\tau_n\sigma_0&=\sigma_n\tau_{n+1}^2.
\end{aligned}
\end{equation}
Note that $\Delta$ is a subcategory of $\Delta C_{\infty}$.

\subsection{The cyclic category}
The~\emph{cyc\-lic ca\-te\-gory}~$\Delta C$ is defined as follows.
The ob\-jects of~$\Delta C$ are the non\-ne\-ga\-tive in\-te\-gers~$n\in \N$.
The mor\-phisms are ge\-ne\-ra\-ted by  morphisms~$\{\delta_i^n\}_{n\in \N^*, 0\le i \le n}$, called \emph{cofaces}, morphisms~$\{\sigma_j^n\}_{n\in \N, 0\le j \le n }$, called \emph{codegeneracies}, and isomorphisms~$\{\tau_n\co n \to n\}_{n\in \N}$, called~\emph{cocyc\-lic ope\-ra\-tors}, which satisfy the relations~\eqref{sr},~\eqref{pcr}, and the~\textit{cycli\-city condition}~\eqref{cc}:
\begin{equation}
\tag{CC} \label{cc}\tau_n ^{n+1} = \id_{n}.
\end{equation}
Note that $\Delta C$ is a quotient of $\Delta C_\infty$.

\subsection{(Co)simplicial, para(co)cyclic, and (co)cyclic objects in a category} \label{simplicial&co}

Let $\mc C$ be any category.
A \textit{simplicial object} in~$\mc C$ is a functor~$X\co \Delta^\opp \to \mc C$, a \textit{paracyclic} object in~$\mc C$ is a functor~$\Delta C_\infty^\opp \to \mc C$, and a \emph{cyclic object} in~$\mc C$ is a functor~$ \Delta C^{\opp} \to \mc C.$
Dually, a~\emph{cosimplicial object} in~$\mc{C}$  is a functor~$\Delta \to \mc C$, a~\textit{paracocyclic} object in~$\mc C$ is a functor~$\Delta C_\infty \to \mc C$, and a~\emph{cocyclic object} in~$\mc C$ is a functor~$\Delta C \to \mc C.$ A (co)simplicial/para(co)cyclic/(co)cyclic object in the category of sets (respectively, of $\kk$-modules) are called \emph{(co)simplicial/para(co)cyclic/(co)cy\-clic sets} (respectively, \emph{$\kk$-modules}).

A \emph{morphism} between two (co)simplicial/para(co)cyclic/(co)cyclic objects is a natural transformation between them.
One often denotes the image of a morphism~$f$ under a (co)simplicial/para(co)cyclic/(co)cy\-clic by the same letter~$f$.

Since the categories $\Delta, \Delta C_\infty, \Delta C$ are defined by generators and relations, a (co)sim\-pli\-cial/para(co)cyclic/(co)cyclic object in a category is entirely determined by the images of the generators satisfying the corresponding relations.
For example, a pa\-ra\-co\-cy\-clic object~$Y$ in~$\mc C$ may be seen as a family~$Y_\bullet=\{Y_n\}_{n\in \N}$ of objects of~$\mc C$ equipped with morphisms~$\{\delta_i^n \co Y_{n-1} \to Y_{n}\}_{n\in \N^*, 0\le i \le n}$, called \textit{cofaces}, morphisms~$\{\sigma_j^n \co Y_{n+1} \to Y_{n}\}_{n\in \N, 0\le j \le n}$, called~\textit{codegeneracies}, and isomorphisms~$\{\tau_n\co Y_n \to Y_n\}_{n\in \N}$, called~\textit{paracocyclic operators}, subject to the relations \eqref{sr} and \eqref{pcr}.
Similarly, a morphism~$\beta\co Y \to Y'$ between two paracocyclic objects~$Y_\bullet$ and~$Y'_\bullet$ in~$\mc C$ is explicited as a family~$\{\beta_n\co Y_n \to Y'_n\}_{n\in \N}$ of morphisms in~$\mc C$ satisfying
\begin{align*}
\beta_n\delta_i^n&= \delta_i^n\beta_{n-1} \quad \text{for any }  n\in \N^* \text{ and } 0\le i \le n, \\
\beta_n \sigma_j^n&=\sigma_j^n\beta_{n+1} \quad \text{for any }  n\in \N \text{ and } 0\le j \le n,\\
\beta_n\tau_n&=\tau_n\beta_n  \quad \text{for any }  n\in \N.
\end{align*}

Clearly, the composition of a (co)simplicial/para(co)cyclic/(co)cyclic object $X$ in $\mc C$ with a functor $F\co \mc C \to \mc D$ is a (co)simplicial/para(co)cyclic/(co)cyclic object $FX$ in $\mc D$.  In particular, useful examples are provided by the co\-va\-riant and the con\-tra\-va\-riant~$\Hom$-func\-tors~$\Hom_{\mc C}(I,-)$ and~$\Hom_{\mc C}(-,I)$, where $I$ is an ob\-ject of~$\mc C$. In this case, we denote: $$\Hom_{\mc C}(I,X)=\Hom_{\mc C}(I,-) \circ X \quad \text{and} \quad \Hom_{\mc C}(X,I)=\Hom_{\mc C}(-,I) \circ X.$$

\begin{lem}\label{twistinginducescyclic} Let~$\mc B$ be a~$\kk$-li\-ne\-ar brai\-ded ca\-te\-go\-ry with a twist $\theta$. If~$Y_\bullet=\{Y_n\}_{n\in \N}$ is a pa\-ra\-co\-cy\-clic object in~$\mc B$ such that its paracocyclic operator satisfies~$\tau_n^{n+1}=\theta_{Y_n}$ for each~$n\in \N$, then
\begin{enumerate}
\item[(a)] $\Hom_{\mc B}(\uu,Y_\bullet)=\{\Hom_{\mc B}(\uu,Y_n)\}_{n\in \N}$ is a cocyclic $\kk$-module,
\item[(b)] $\Hom_{\mc B}(Y_\bullet,\uu)=\{\Hom_{\mc B}(Y_n,\uu)\}_{n\in \N}$ is a cyclic $\kk$-module.
\end{enumerate}
\end{lem}

\begin{proof} Let us prove~$(a)$. By composition with~$\Hom_{\mc B}(\uu, -) \co \mc B \to \Mod_\kk$, we obtain that~$\{\Hom_{\mc B}(\uu,Y_n)\}_{n\in \N}$ is a paracocyclic~$\kk$-module. Let us now verify that~\eqref{cc} holds.
Let~$n\in \N$.
The morphism~$\Hom_{\mc B}(\uu,-)(\theta_{Y_n}) \co \Hom_{\mc B}(\uu,Y_n)\to \Hom_{\mc B}(\uu,Y_n)$ is the~$\kk$-linear morphism given by~$f\mapsto \theta_{Y_n}f$.
The naturality of~$\theta$ and the fact that~$\theta_\uu=\id_\uu$ imply that~$\theta_{Y_n}f=f\theta_{\uu}=f$ for all~$f\in \Hom_{\mc B}(\uu,Y_n)$.
Then, using the functoriality of~$Y_\bullet$ and the hypothesis that~$Y_\bullet(\tau_n^{n+1})=\theta_{Y_n}$, we have
$$\left(\Hom_{\mc B}(\uu,Y_\bullet)(\tau_n)\right)^{n+1}=\Hom_{\mc B}(\uu,-)\left(Y_\bullet(\tau_n^{n+1})\right)=\Hom_{\mc B}(\uu,-)(\theta_{Y_n})=\id_{\Hom_{\mc B}(\uu,Y_n)}.$$
Part~$(b)$ is proved similarly.
\end{proof}

\subsection{Cyclic (co)homology} To any cyclic~$\kk$-module~$X\co \Delta C^\opp \to \Mod_\kk$, one can associate a bicomplex~$\textbf{CC}(X)$ (see \cite{weibel}).
The~$n$-th \emph{cyclic homology}~$HC_n(X)$ of $X$ is defined as the~$n$-th homology of the total chain complex associated to the chain bicomplex~$\textbf{CC}(X)$.
A morphism between cyclic~$\kk$-modules induces a levelwise morphism in cyclic homology.

Similarly, to any cocyclic~$\kk$-module~$Y\co \Delta C\to \Mod_\kk$, one can associate a cochain bicomplex~$\textbf{CC}(Y)$, obtained by a construction dual to the one of a chain bicomplex.
The~$n$-th \emph{cyclic cohomology}~$HC^n(Y)$ of~$Y$ is defined as the~$n$-th cohomology of the total cochain complex associated to the cochain bicomplex~$\textbf{CC}(Y)$.
A morphism between cocyclic~$\kk$-modules induces a levelwise morphism in cyclic cohomology.

\section{Modular pairs and braided Connes-Moscovici construction} \label{res1}

In this section,~$\mc B$ is a braided monoidal category and~$H$ is a Hopf algebra in~$\mc B$. We provide a braided generalization of the notion of a modular pair in involution for $H$ and then compute the powers of the paracocyclic operator associated to such a pair (see Theorem~\ref{powers} and its corollaries).

\subsection{Modular pairs} \label{modpair}
A~\emph{modular pair} for~$H$ is a pair~$(\delta,\sigma)$ where~$\delta \co H \to \uu$ is an algebra morphism and~$\sigma \co \uu \to H$ is a coalgebra morphism such that~$\delta \sigma = \id_{\uu}$.
For instance,~$(\varepsilon, u)$ is a modular pair for $H$, where $\varepsilon \co H\to \uu$ and $u \co \uu \to H$ are the counit and unit of $H$, respectively.

Given a twist $\theta$ for $\mc B$, a~$\theta$-\emph{twisted modular pair in involution} for~$H$ is a modular pair~$(\delta, \sigma)$ for~$H$ such that
$$\,
\psfrag{A}{\hspace{-0.05cm}$H$}
\rsdraw{0.55}{0.75}{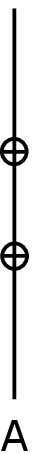}
\;=
\,
\psfrag{A}{\hspace{-0.05cm}$H$}
\psfrag{c}{$\sigma$}
\psfrag{l}{\hspace{-0.07cm}$\delta$}
\rsdraw{0.55}{0.75}{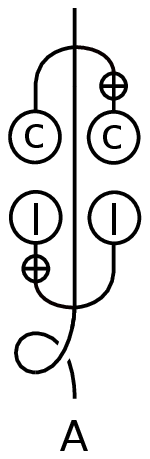}
\;.
$$

Here we use the graphical conventions from Section \ref{cg}.
If~$H$ is \emph{involutive} Hopf algebra in~$\mc B$ in the sense that~$S^2=\theta_H$, then~$(\varepsilon, u)$ is a~$\theta$-twisted modular pair in involution for~$H$.

Note that if $\mc B$ is symmetric with a trivial twist $\id_{\mc B}$ (see Section \ref{catswithtwist}), then an $\id_{\mc B}$-twisted modular pair in involution corresponds to a braided modular pair in involution in the sense of \cite{khalkhalipourkia}.

\subsection{Powers of the paracocyclic operators} \label{maincomputation}

Let $(\delta, \sigma)$ be a modular pair for $H$. For any $n\ge 0$, define the \emph{paracocyclic operator} $\tau_n(\delta, \sigma)\co H^{\tens n} \to H^{\tens n}$ by
\[\tau_0(\delta, \sigma)= \id_{\uu},\quad  \tau_1(\delta, \sigma)= \,
\psfrag{b}{\hspace{-0.05cm}$H$}
\psfrag{i}{\hspace{-0.25cm}$H^{\tens n-1}$}
\psfrag{v}{$\tilde{S}$}
\psfrag{c}{$\sigma$}
\psfrag{l}{\hspace{-0.05cm}$\delta$}
\rsdraw{0.55}{0.75}{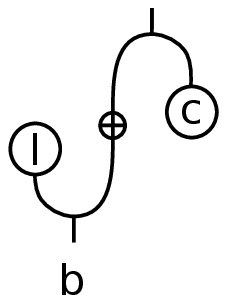}
\;
\quad \text{and} \quad
\tau_n(\delta, \sigma)=\,
\psfrag{b}{\hspace{-0.05cm}$H$}
\psfrag{i}{\hspace{-0.25cm}$H^{\tens n-1}$}
\psfrag{v}{$\tilde{S}$}
\psfrag{c}{$\sigma$}
\psfrag{l}{\hspace{-0.05cm}$\delta$}
\rsdraw{0.55}{0.75}{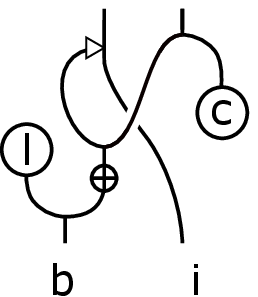}
\; \quad \text{for} \quad n\geq 2.
\]
Here we use the diagonal actions defined in Section~\ref{diagonalacts}.
Note that the operators~$\tau_n(\delta,\sigma)$ are the paracocyclic operators of a paracocyclic object in~$\mc B$ associated with~$H$ and ~$(\delta,\sigma)$ (see Section \ref{recall}).
In the following theorem we compute the powers (up to~$n+1$) of~$\tau_n(\delta, \sigma)$.
\begin{thm}\label{powers}
For $n\ge 2$ and $2\leq k \leq n$, we have:
\begin{equation}\label{kthpower}
\left(\tau_n(\delta, \sigma)\right)^k=\,
\psfrag{z}{\hspace{-0.75cm}$H^{\tens k-2}$}
\psfrag{c}{\hspace{-0.25cm}$k-1$}
\psfrag{c+}{$k$}
\psfrag{n}{$n$}
\psfrag{f}{\hspace{-0.6cm}$\left(\tau_{n-1}(\varepsilon,u)\right)^{k-1}$}
\psfrag{g}{\hspace{-0.85cm}$\tau_1(\delta, \sigma)$}
\psfrag{l}{\hspace{-0.1cm}$\delta$}
\psfrag{s}{$\sigma$}
\psfrag{h}{\hspace{-0.6cm}$H^{\tens n-k}$}
\psfrag{d}{\hspace{-0.5cm}$H^{\tens k-1}$}
\psfrag{v}{$\tilde{S}$}
\rsdraw{0.55}{0.75}{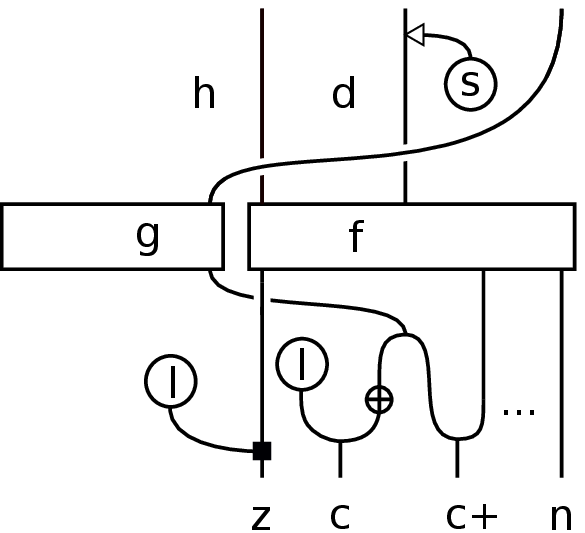}  \; \text{  }.
\end{equation}
In addition,
\begin{equation}\label{n+1thpower}\phantom{blabla}(\tau_n(\delta, \sigma))^{n+1}=\begin{cases}
\,
\psfrag{A}{$H$}
\psfrag{l}{\hspace{-0.05cm}$\delta$}
\psfrag{v}{\hspace{0.05cm}$\sigma$}
\psfrag{u}{$H$}
\rsdraw{0.55}{0.75}{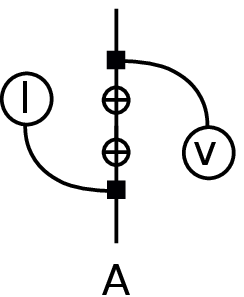}  \; &\text{ if } n=1,
\\
\,
\psfrag{z}{\hspace{-0.4cm}$H^{\tens n-1}$}
\psfrag{c}{\hspace{0.2cm}$H$}
\psfrag{f}{\hspace{-0.55cm}$\left(\tau_{n-1}(\varepsilon,u)\right)^{n}$}
\psfrag{g}{\hspace{-1.1cm}$\left(\tau_1(\delta,\sigma)\right)^{2}$}
\psfrag{l}{\hspace{-0.1cm}$\delta$}
\psfrag{s}{$\sigma$}
\psfrag{d}{\hspace{-0.65cm}$H^{\tens n-1}$}
\rsdraw{0.55}{0.75}{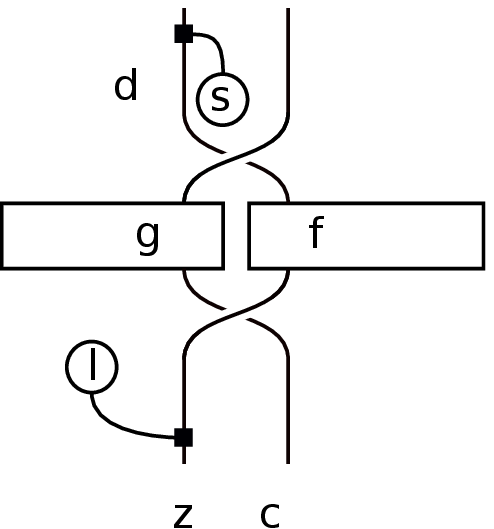}\; & \text{ if } n\ge 2.\end{cases} \end{equation}
\end{thm}
In the statements of the theorem, we use the diagonal actions together with the (co)adjoint (co)actions defined in Sections~\ref{adjact} and~\ref{coadjcoact}. Also, an integer~$k$ below an arc denotes the~{$k$-th} tensorand of~$H^{\tens n}.$
We prove Theorem~\ref{powers} in Section~\ref{proofmain} by induction and by using pro\-per\-ties of modular pairs and twisted antipodes.

In the next corollary, we compute the~$(n+1)$-th power of the paracocyclic operator~$\tau_n(\delta,\sigma)$ in terms of the~$(n+1)$-th power of the paracocyclic operator~$\tau_n(\varepsilon,u)$, where~$\varepsilon$ and~$u$ are the counit and unit of~$H$.
\begin{cor}\label{remarkeuds}
For any modular pair $(\delta,\sigma)$ and any $n\in \N$,
\begin{equation*}
(\tau_{n}(\delta,\sigma))^{n+1}=\,
\psfrag{c}{\hspace{-0.1cm}$H^{\tens n}$}
\psfrag{f}{\hspace{-0.5cm}$\left(\tau_{n}(\varepsilon,u)\right)^{n+1}$}
\psfrag{l}{\hspace{-0.1cm}$\delta$}
\psfrag{s}{$\sigma$}
\rsdraw{0.5}{0.75}{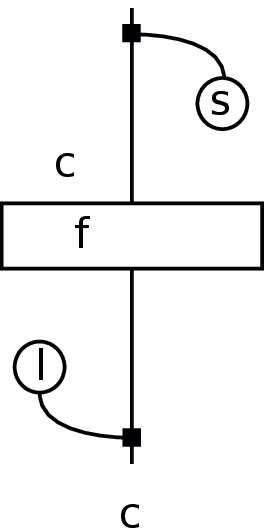}\;.
\end{equation*}
\end{cor}

\begin{proof}
We show the result by induction.
For~$n=0$, this follows since the adjoint action on~$\uu$ is given by counit, since the coadjoint coaction on~$\uu$ is given by unit and the fact that~$\tau_0(\varepsilon,u)=\tau_0(\delta,\sigma)=\id_\uu$. Let us check the case~$n=1$. Using~\eqref{n+1thpower} and the fact that~$\tau_1(\varepsilon,u)=S$, we obtain
$$(\tau_1(\delta,\sigma))^2=\,
\psfrag{b}{\hspace{-0.2cm}$H$}
\psfrag{v}{$\sigma$}
\psfrag{l}{\hspace{-0.1cm}$\delta$}
\rsdraw{0.55}{0.75}{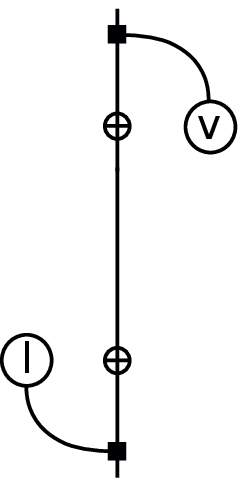}\;=\,
\psfrag{b}{\hspace{-0.2cm}$H$}
\psfrag{v}{$\sigma$}
\psfrag{l}{\hspace{-0.1cm}$\delta$}
\psfrag{g}{\hspace{-0.3cm}$(\tau_1(\varepsilon,u))^2$}
\rsdraw{0.55}{0.75}{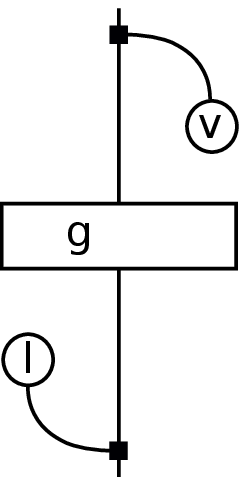}\;.$$
Suppose that the result is true for an $n\ge 1$ and let us show it for $n+1$.
Indeed, we have
$$(\tau_{n+1}(\delta,\sigma))^{n+2}\overset{(i)}{=}\,
\psfrag{z}{\hspace{-0.2cm}$H^{\tens n}$}
\psfrag{c}{\hspace{0.2cm}$H$}
\psfrag{f}{\hspace{-0.55cm}$\left(\tau_{n}(\varepsilon,u)\right)^{n+1}$}
\psfrag{g}{\hspace{-1.1cm}$\left(\tau_1(\delta,\sigma)\right)^{2}$}
\psfrag{l}{\hspace{-0.1cm}$\delta$}
\psfrag{s}{$\sigma$}
\psfrag{d}{\hspace{-0.35cm}$H^{\tens n}$}
\rsdraw{0.55}{0.75}{taunn+1DS}\;\overset{(ii)}{=}\,
\psfrag{z}{\hspace{-0.2cm}$H^{\tens n}$}
\psfrag{c}{\hspace{0.2cm}$H$}
\psfrag{f}{\hspace{-0.55cm}$\left(\tau_{n}(\varepsilon,u)\right)^{n+1}$}
\psfrag{g}{\hspace{-1.1cm}$\left(\tau_1(\varepsilon,u)\right)^{2}$}
\psfrag{l}{\hspace{-0.1cm}$\delta$}
\psfrag{s}{$\sigma$}
\psfrag{d}{\hspace{-0.35cm}$H^{\tens n}$}
\rsdraw{0.55}{0.75}{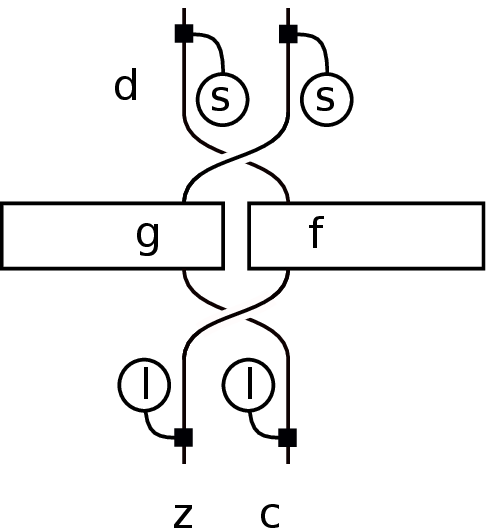}\;\overset{(iii)}{=}\,
\psfrag{z}{\hspace{-0.2cm}$H^{\tens n}$}
\psfrag{c}{\hspace{0.2cm}$H$}
\psfrag{f}{\hspace{-0.55cm}$\left(\tau_{n}(\varepsilon,u)\right)^{n+1}$}
\psfrag{g}{\hspace{-0.8cm}$\left(\tau_{n+1}(\varepsilon,u)\right)^{n+2}$}
\psfrag{l}{\hspace{-0.1cm}$\delta$}
\psfrag{s}{$\sigma$}
\psfrag{d}{\hspace{-0.35cm}$H^{\tens n}$}
\rsdraw{0.55}{0.75}{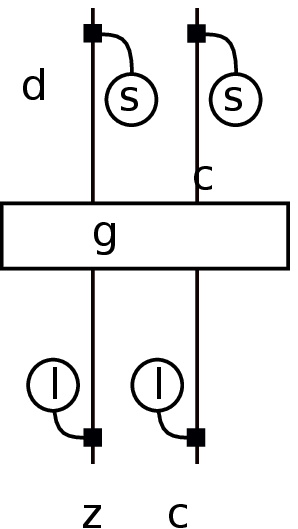}\;.$$
Here~$(i)$ follows by applying~\eqref{n+1thpower} for the modular pair~$(\delta,\sigma)$,~$(ii)$ by applying the result for~$n=1$, and~$(iii)$ by applying~\eqref{n+1thpower} for the modular pair~$(\varepsilon,u)$.
\end{proof}

The next corollary states that the paracocyclic operator associated with a twisted mo\-du\-lar pair in involution satisfies the ``twisted cocyclicity condition''.
\begin{cor} \label{n+1thpower=twist} If $\mc B$ has a twist $\theta$ and $(\delta,\sigma)$ is a $\theta$-twisted modular pair in involution for $H$, then $(\tau_n(\delta,\sigma))^{n+1}=\theta_{H^{\tens n}}$ for all $n\in \N$.
\end{cor}

\begin{proof}
The equality~$(\tau_n(\delta,\sigma))^{n+1}=\theta_{H^{\tens n}}$ is shown by induction.
For~$n=0$, this follows by definition and the fact that~$\theta_\uu=\id_\uu$. Indeed,~$\tau_0(\delta,\sigma)=\id_\uu=\theta_\uu=\theta_{H^{\tens 0}}$.
For~$n=1$, we have
$$(\tau_1(\delta,\sigma))^{2}\overset{(i)}{=}
\,
\psfrag{b}{\hspace{-0.2cm}$H$}
\psfrag{v}{$\sigma$}
\psfrag{l}{\hspace{-0.05cm}$\delta$}
\rsdraw{0.55}{0.75}{tau123DSnolab}\;\overset{(ii)}{=}
\,
\psfrag{b}{\hspace{-0.2cm}$H$}
\psfrag{v}{$\sigma$}
\psfrag{c}{$\sigma$}
\psfrag{l}{\hspace{-0.05cm}$\delta$}
\rsdraw{0.55}{0.75}{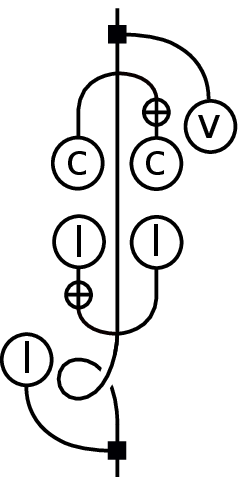}\;\overset{(iii)}{=}\,
\psfrag{b}{\hspace{-0.2cm}$H$}
\psfrag{v}{$\sigma$}
\psfrag{c}{$\sigma$}
\psfrag{l}{\hspace{-0.07cm}$\delta$}
\rsdraw{0.55}{0.75}{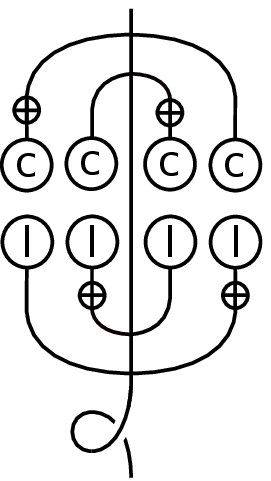}\;\overset{(iv)}{=}\,
\psfrag{b}{\hspace{-0.2cm}$H$}
\psfrag{v}{$\sigma$}
\psfrag{c}{$\sigma$}
\psfrag{l}{\hspace{-0.07cm}$\delta$}
\rsdraw{0.55}{0.75}{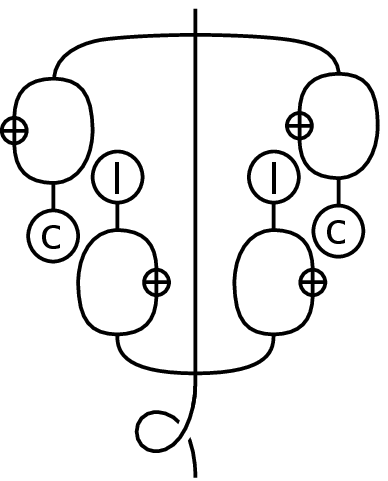}\;\overset{(v)}{=}\theta_H.$$
Here $(i)$ follows by Formula \eqref{n+1thpower} of Theorem \ref{powers}, $(ii)$ follows by the fact that $(\delta,\sigma)$ is a $\theta$-twisted modular pair in involution for $H$, $(iii)$ follows by the naturality of the twist and the de\-fi\-ni\-tions of left coadjoint coaction and right adjoint action, $(iv)$ follows by (co)associativity and the fact that $\delta$ is an algebra morphism and $\sigma$ is a coalgebra morphism, $(v)$ follows by the antipode axiom and (co)unitality.

Suppose that the statement is true for an $n\ge 1$ and let us show it for $n+1$. We have
$$(\tau_{n+1}(\delta,\sigma))^{n+2}\overset{(i)}{=}\,
\psfrag{z}{\hspace{-0.4cm}$H^{\tens n}$}
\psfrag{c}{\hspace{0.2cm}$H$}
\psfrag{f}{\hspace{-0.55cm}$\left(\tau_{n}(\varepsilon,u)\right)^{n+1}$}
\psfrag{g}{\hspace{-1.1cm}$\left(\tau_1(\delta,\sigma)\right)^{2}$}
\psfrag{l}{\hspace{-0.1cm}$\delta$}
\psfrag{s}{$\sigma$}
\psfrag{d}{\hspace{-0.65cm}$H^{\tens n}$}
\rsdraw{0.55}{0.75}{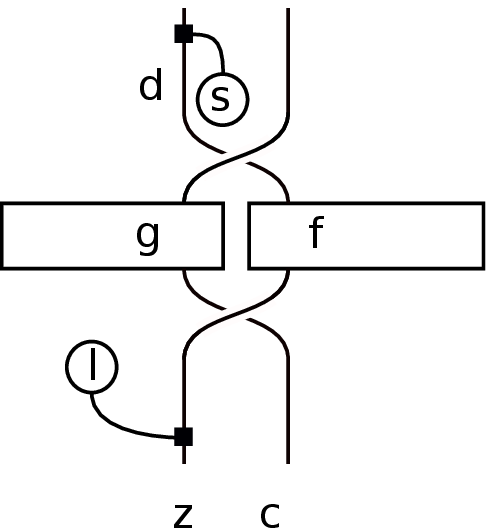}\;\overset{(ii)}{=}\,
\psfrag{z}{\hspace{-0.4cm}$H^{\tens n}$}
\psfrag{c}{\hspace{0.2cm}$H$}
\psfrag{f}{\hspace{-0.55cm}$\left(\tau_{n}(\delta,\sigma)\right)^{n+1}$}
\psfrag{g}{\hspace{-1.1cm}$\left(\tau_1(\delta,\sigma)\right)^{2}$}
\psfrag{l}{\hspace{-0.1cm}$\delta$}
\psfrag{s}{$\sigma$}
\psfrag{d}{\hspace{-0.65cm}$H^{\tens n}$}
\rsdraw{0.55}{0.75}{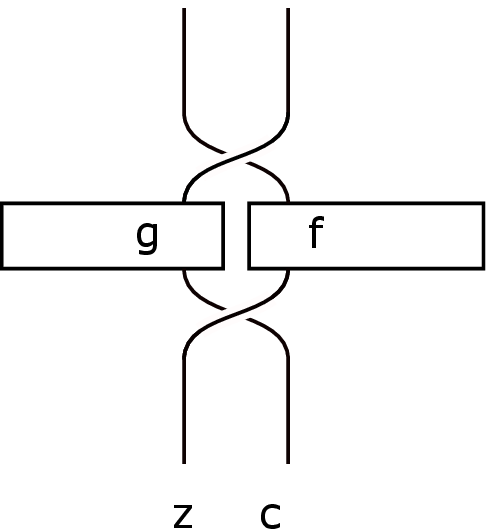}\;\overset{(iii)}{=}\,
\psfrag{z}{\hspace{-0.4cm}$H^{\tens n}$}
\psfrag{c}{\hspace{0.2cm}$H$}
\psfrag{f}{\hspace{-0.55cm}$\left(\tau_{n}(\delta,\sigma)\right)^{n+1}$}
\psfrag{g}{\hspace{-1.1cm}$\left(\tau_1(\delta,\sigma)\right)^{2}$}
\psfrag{l}{\hspace{-0.1cm}$\delta$}
\psfrag{s}{$\sigma$}
\psfrag{d}{\hspace{-0.65cm}$H^{\tens n}$}
\rsdraw{0.55}{0.75}{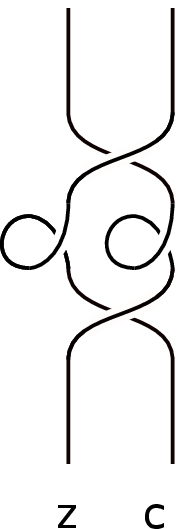}\;\overset{(iv)}{=}\theta_{H^{\tens n+1}}.$$
Here $(i)$ follows from Formula \eqref{n+1thpower} of Theorem \ref{powers}, $(ii)$ follows from Corollary \ref{remarkeuds}, $(iii)$ follows from the statement for $n=1$ and the induction hypothesis, $(iv)$ follows by the naturality of the braiding and from the axiom of the twist.
\end{proof}

\begin{rmk}
\normalfont If~$\mc B$ is a symmetric monoidal category endowed with the trivial twist~$\id_{\mc B}$ and~$(\delta,\sigma)$ is a~$\id_{\mc B}$-twisted modular pair for~$H$, then Corollary~\ref{n+1thpower=twist} gives that~$(\tau_n(\delta,\sigma))^{n+1}=\id_{H^{\tens n}}$ for all~$n\in \N$. This was first proved by Khalkhali and Pourkia in~\cite{khalkhalipourkia}.
\end{rmk}

\subsection{Paracocyclic objects associated with modular pairs} \label{recall}
Let~$(\delta,\sigma)$ be a modular pair for~$H$.
Let us recall the paracocyclic object~$\textbf{CM}_\bullet(H,\delta,\sigma)$ in~$\mc B$ from~\cite{khalkhalipourkia} associated to this data.
For any~$n\ge 0$, define~$$\textbf{CM}_n(H,\delta,\sigma)=H^{\tens n}.$$
For any $n\ge 1$, define the cofaces~$\{\delta_i^n(\sigma) \co  H^{\tens n-1} \to H^{\tens n}\}_{0\le i \le n}$ by setting~$\delta_0^1=u$,~$\delta_1^1=\sigma$, and for any $n\ge 2,$
\[\delta_i^n(\sigma)=\begin{cases}
\,
\psfrag{b}{\hspace{-0.1cm}$1$}
\psfrag{e}{\hspace{-0.4cm}$n-1$}
\rsdraw{0.55}{0.75}{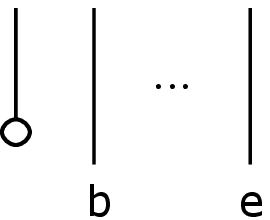}
\;  &\quad \quad \text{if } i=0,\\ \\
\,
\psfrag{b}{$1$}
\psfrag{j}{$i$}
\psfrag{e}{\hspace{-0.4cm}$n-1$}
\rsdraw{0.55}{0.75}{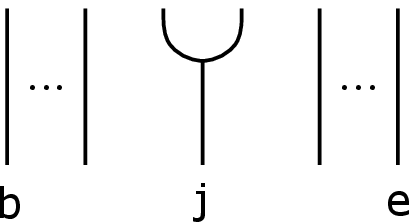}\; &\quad \quad \text{if }1\le i \le n-1,\\ \\
\,
\psfrag{b}{$1$}
\psfrag{n}{$n$}
\psfrag{c}{$\sigma$}
\psfrag{e}{\hspace{-0.4cm}$n-1$}
\rsdraw{0.55}{0.75}{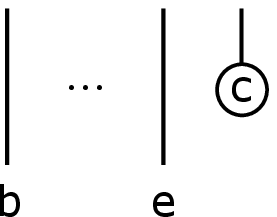}
\; &\quad \quad \text{if } i=n. \end{cases}\]
For any $n\ge 0$, define the codegeneracies $\{\sigma_j^n \co  H^{\tens n+1} \to H^{\tens n}\}_{0\le j \le n}$ by
\[\sigma_j^n=\,
\psfrag{b}{$1$}
\psfrag{i}{\hspace{-0.45cm}$j+1$}
\psfrag{e}{\hspace{-0.35cm}$n+1$}
\rsdraw{0.55}{0.75}{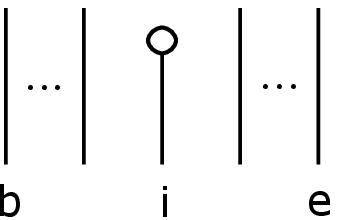}
\;.\]
For any $n\ge 0$, the paracocyclic operators~$\tau_n(\delta, \sigma)\co H^{\tens n} \to H^{\tens n}$ of~$\textbf{CM}_\bullet(H,\delta,\sigma)$ are those defined in Section~\ref{maincomputation}.

Theorem \ref{powers} is useful to prove that $\textbf{CM}_\bullet(H,\delta,\sigma)$ is a paracocyclic object in~$\mc B$. We prove this in \nameref{appendix}. In particular, we prove that for all $n\in \N$, $$\tau_n(\delta, \sigma)\sigma_0^n=\sigma_n^n(\tau_{n+1}(\delta, \sigma))^2.$$ The next corollary derives (co)cyclic $\kk$-modules from $\textbf{CM}_\bullet(H,\delta,\sigma)$. It follows directly from Lemma \ref{twistinginducescyclic} and Corollary \ref{n+1thpower=twist}.
\begin{cor}
If~$\mc B$ has a twist $\theta$ and~$(\delta, \sigma)$ a $\theta$-twisted mo\-du\-lar pair in invo\-lu\-tion for $H$, then
\begin{itemize}
\item[(a)] $\{\Hom_{\mc B}(\uu,\textbf{\emph{CM}}_n(H,\delta,\sigma))\}_{n\in \N}$ is a cocyclic~$\kk$-module,
\item[(b)] $\{\Hom_{\mc B}(\textbf{\emph{CM}}_n(H,\delta,\sigma),\uu)\}_{n\in \N}$ is a cyclic~$\kk$-module.
\end{itemize}
\end{cor}

\section{Categorical Connes-Moscovici trace} \label{res2}

In this section,~$\mc B$ is a braided category with a twist~$\theta$ and~$H$ is a Hopf algebra in~$\mc B$. We  in\-tro\-duce traces (à la Connes-Moscovici) be\-tween pa\-ra\-co\-cy\-clic ob\-jects as\-so\-ci\-a\-ted with~$H$ (as in Section~\ref{recall}) and pa\-ra\-co\-cy\-clic ob\-jects as\-so\-ci\-a\-ted with an~$H$-module co\-al\-ge\-bra.
We provide an explicit example of such traces using coends.

\subsection{Paracocyclic objects associated with coalgebras} \label{alasolotar}
Let $C$ be a coalgebra in $\mc B$.
We associate with $C$ a paracocyclic object $\textbf{C}_\bullet(C)$ in $\mc B$.
It is inspired by the construction of Akrami and Majid from \cite{cycliccocycles}. When $\mc B=\Mod_\kk$ is the category of $\kk$-modules, one recovers the cocyclic $\kk$-module implicitly defined in the work of Farinati and Solotar~\cite{farinatisolotar}.
When~$\mc B$ is a symmetric monoidal category endowed with the trivial twist, then the underlying cosimplicial object of~$\textbf{C}_{\bullet}(C)$ is equal to the one considered in~\cite[Definition 2.2]{symmcohoch}.

For~$n\ge 0$, set~$\textbf{C}_n(C)= C^{\tens n+1}$.
For $n\ge 1$, define the cofaces $\{\delta_i^n \co  C^{\tens n} \to  C^{\tens n+1}\}_{0\le i \le n}$ by
\[\delta_i^n=\begin{cases}\,
\psfrag{b}{$1$}
\psfrag{j}{\hspace{-0.3cm}$i+1$}
\psfrag{e}{$n$}
\rsdraw{0.55}{0.75}{degenAC} \; & \text{ if } 0 \le i \le n-1, \\ \\
\,
\psfrag{b}{$1$}
\psfrag{j}{$2$}
\psfrag{e}{$n$}
\rsdraw{0.55}{0.75}{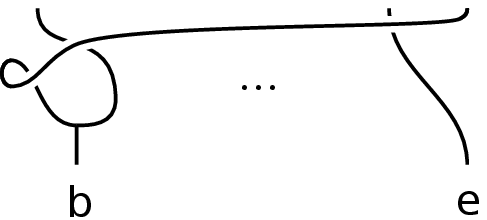} \;& \text{ if } i = n.
 \end{cases}
\]
For $n\ge 0$, define the codegeneracies $\{\sigma_j^n \co   C^{\tens n+2} \to C^{\tens n+1}\}_{0\le j \le n}$ by
$$\sigma_j^n=\,
\psfrag{b}{$0$}
\psfrag{i}{\hspace{-0.45cm}$j+1$}
\psfrag{e}{\hspace{-0.4cm}$n+1$}
\rsdraw{0.55}{0.75}{faceAC}\;.$$
For $n\ge 0$, define the paracocyclic operators $\tau_n\co C^{\tens n+1} \to C^{\tens n+1}$ by
$$\tau_0=\, \psfrag{X}{$C$} \rsdraw{0.55}{0.75}{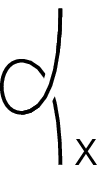}\;  \quad \text{and} \quad
\tau_n=
\, \psfrag{b}{$0$}
\psfrag{bn}{\hspace{0.15cm}$1$}
\psfrag{pe}[l]{\hspace{-0.2cm}$n-1$}
\psfrag{e}{$n$}
\rsdraw{0.55}{0.75}{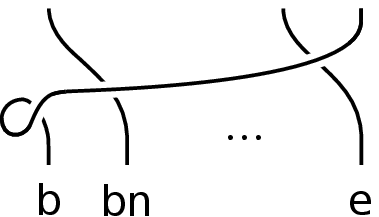}\; \quad \text{if}\quad n\geq 1.
$$
It follows directly from the definition of a twist (see Section \ref{catswithtwist}) that the paracocyclic operator for $\textbf{C}_\bullet(C)$ satisfies the relation $\tau_n^{n+1}=\theta_{C^{\tens n+1}}$ for all $n\in \N$.

\subsection{Traces}
Let~$C$ be an \emph{$H$-module coalgebra} in~$\mc B$, that is, a coalgebra in the category of right~$H$-modules in~$\mc B$. In other words,~$C$ is a coalgebra in~$\mc B$ endowed with a right action~$r\co C\tens H \to C$ of~$H$ on~$C$ such that the comultiplication~$\Delta_C$ and the counit~$\varepsilon_C$ of~$C$ are both $H$-linear, that is, morphisms of right~$H$-modules.
By depicting the right action by $$r=\,
\psfrag{u}{\hspace{-0.2cm}$H$}
\psfrag{A}{\hspace{-0.15cm}${\color{red} C}$}
\rsdraw{0.45}{0.75}{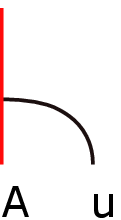}\;,$$ the $H$-linearity of $\Delta_C$ and $\varepsilon_C$ depicts as
$$
\,
\psfrag{u}{\hspace{-0.2cm}$H$}
\psfrag{A}{\hspace{-0.15cm}$C$}
\rsdraw{0.45}{0.75}{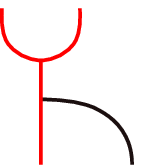}\;=\,
\psfrag{u}{\hspace{-0.2cm}$H$}
\psfrag{A}{\hspace{-0.15cm}$C$}
\rsdraw{0.45}{0.75}{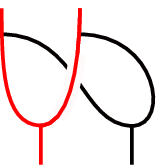}\; \qquad \text{and} \qquad \,
\psfrag{u}{\hspace{-0.2cm}$H$}
\psfrag{A}{\hspace{-0.15cm}$C$}
\rsdraw{0.45}{0.75}{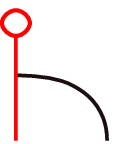}\;=\,
\psfrag{u}{\hspace{-0.2cm}$H$}
\psfrag{A}{\hspace{-0.15cm}$C$}
\rsdraw{0.45}{0.75}{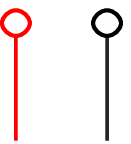}\;.
$$
In this pictures, the red strands are colored by $C$ and the black ones by $H$.

Let $\delta \co H\to \uu$ be an algebra morphism and let $\sigma \co \uu \to H$ be a coalgebra morphism.
A~$\delta$-\emph{invariant}~$\sigma$-\emph{trace} for $C$ is a morphism~$\alpha\co \uu \to C$ in~$\mc B$ satisfying
$$
\,
\psfrag{g}{$\color{red} \alpha$}
\psfrag{u}{\hspace{-0.1cm}$H$}
\psfrag{A}{\hspace{-0.05cm}$\color{red} C$}
\rsdraw{0.45}{0.75}{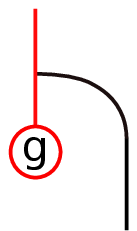}\;=\,
\psfrag{u}{\hspace{-0.1cm}$H$}
\psfrag{A}{\hspace{-0.05cm}$\color{red}C$}
\psfrag{c}{$\color{red} \alpha$}
\psfrag{l}{\hspace{-0.07cm}$\delta$}
\rsdraw{0.45}{0.75}{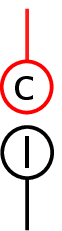}\; \qquad \text{and} \qquad \,
\psfrag{g}{$\color{red} \alpha$}
\psfrag{A}{\hspace{-0.05cm}$\color{red}C$}
\psfrag{s}{$\sigma$}
\rsdraw{0.45}{0.75}{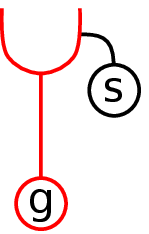}\;=\,
\psfrag{g}{$\color{red} \alpha$}
\psfrag{A}{\hspace{-0.05cm}$\color{red}C$}
\psfrag{s}{$\sigma$}
\rsdraw{0.45}{0.75}{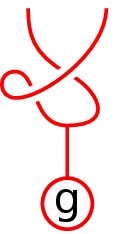}\;.
$$
Given such a morphism, define for any~$n\in \N$ the morphism~$\alpha_n \co H^{\tens n} \to C^{\tens n+1}$ in $\mc B$ by setting
$$\alpha_0=\alpha \quad \text{and} \quad \alpha_n= \,
\psfrag{g}{$\color{red}\alpha$}
\psfrag{A}{\hspace{-0.15cm}$C$}
\psfrag{b}{$1$}
\psfrag{e}{$n$}
\rsdraw{0.55}{0.75}{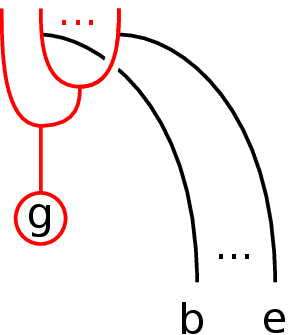}\; \quad \text{for} \quad n\ge 1.$$

Consider the paracocyclic object $\textbf{CM}_\bullet(H,\delta,\sigma)$ in $\mc B$ (see Section \ref{recall}) and the paracocyclic object $\textbf{C}_\bullet(C)$ in $\mc B$ associated to the coalgebra $C$ in $\mc B$ (see Section \ref{alasolotar}).
\begin{thm}\label{CMtrace} Let~$(\delta, \sigma)$ be a modular pair for~$H$ and~$\alpha$ be a~$\delta$-invariant~$\sigma$-trace for~$C$. Then the family~$\{\alpha_n \co H^{\tens n} \to C^{\tens n+1}\}_{n\in \N}$ is a morphism of paracocyclic objects from~$\emph{\textbf{CM}}_\bullet(H,\delta,\sigma)$ to~$\emph{\textbf{C}}_\bullet(C)$.
\end{thm}
We prove Theorem \ref{CMtrace} in Section \ref{pfCMtrace}. The next corollary relates the cyclic (co)homologies associated with $\textbf{CM}_\bullet(H,\delta,\sigma)$ and $\textbf{C}_\bullet(C)$.

\begin{cor} Assume that~$\mc B$ is~$\kk$-linear and has a twist $\theta$. Let~$(\delta, \sigma)$ be a $\theta$-twisted mo\-du\-lar pair in invo\-lu\-tion for $H$ and $\alpha$ be a $\delta$-invariant $\sigma$-invariant trace for $C$. Then
\begin{itemize}
  \item[(a)] \label{traceandco} The family $\{\Hom_{\mc B}(\uu,\alpha_n)\}_{n\in \N}$ induces a morphism in cyclic cohomology $$\alpha^* \co HC^*\left(\Hom_{\mc B}(\uu,\textbf{\emph{CM}}_\bullet(H,\delta,\sigma))\right) \to HC^*\left(\Hom_{\mc B}(\uu,\textbf{\emph{C}}_\bullet(C))\right).$$
  \item[(b)] \label{traceandco1} The family $\{\Hom_{\mc B}(\alpha_n,\uu)\}_{n\in \N}$ induces a morphism in cyclic homology $$\alpha_* \co HC_*\left(\Hom_{\mc B}(\textbf{\emph{C}}_\bullet(C),\uu)\right) \to HC_*\left(\Hom_{\mc B}(\textbf{\emph{CM}}_\bullet(H,\delta,\sigma),\uu)\right).$$
\end{itemize}
\end{cor}
\begin{proof}
Since~$(\textbf{C}_\bullet(C)(\tau_{n}))^{n+1}= \theta_{C^{\tens n+1}}$, Lemma~\ref{twistinginducescyclic} implies that~$\Hom_{\mc B}(\uu,\textbf{C}_\bullet(C))$ is a cocyclic~$\kk$-module and that~$\Hom_{\mc B}(\textbf{C}_\bullet(C),\uu)$ is a cyclic~$\kk$-module.
Next, by an application of Theorem~\ref{CMtrace}, the family~$\{\Hom_{\mc B}(\uu,\alpha_n)\}_{n\in \N}$ is a natural transformation between cocyclic~$\kk$-modules~$\Hom_{\mc B}(\uu,\textbf{CM}_\bullet(H,\delta,\sigma))$ and~$\Hom_{\mc B}(\uu,\textbf{C}_\bullet(C))$.
This means that for any~$n\in \N$, there is a morphism
\begin{align*}
  \alpha^n \co HC^n\left(\Hom_{\mc B}(\uu,\textbf{CM}_\bullet(H,\delta,\sigma))\right) &\to HC^n\left(\Hom_{\mc B}(\uu,\textbf{C}_\bullet(C))\right)\\
  \left[f\right]&\mapsto \left[\alpha_nf\right],
\end{align*}
where~$\left[f\right]$ is a representative class of an~$n$-th cyclic cocycle.
This finishes the proof of the part~$(a)$.
The proof of~$(b)$ is similar.
\end{proof}

\subsection{Traces from coends} \label{tracesdrinfeld}
Let~$\mc B$ be a rib\-bon ca\-te\-go\-ry with a coend~$(H,i)$, see Sec\-tion~\ref{CoendRibbon}. By~\cite[Chapter 6]{moncatstft}, the ob\-ject~$H$ is a Hopf al\-ge\-bra in~$\mc B$ which is involutive (that is~$S^2=\theta_H$, where~$\theta$ is the twist of~$\mc B$, see Section~\ref{ribbon}). By Section~\ref{modpair}, since~$H$ is involutive, the pair~$(\varepsilon, u)$ is a~$\theta$-twisted mo\-du\-lar pair in in\-vo\-lu\-tion for~$H$. Since~$\Mod_H$ is braided isomorphic to the center of~$\mc B$ (see \cite[Section 6.5.3]{moncatstft} for details), we obtain that~$\Mod_H$ is ribbon. The braiding of~$\Mod_H$ is given by $$
\tau_{(M,r), (N,s)}= \,
    \psfrag{M}{\hspace{-0.1cm}$M$}
    \psfrag{N}{$N$}
    \psfrag{H}{$H$}
    \rsdraw{0.55}{0.75}{braidmodH}
  \;.
$$
Here, the coaction denoted with a black dot is the universal coaction of~$H$ (see Section~\ref{CoendRibbon}).
The dual of $(M,r) \in \Ob{\Mod_H}$ is given by $(M^*, r^\dag)$, where
$$r^\dag=\,
    \psfrag{M}{$M$}
    \psfrag{f}{$r$}
    \psfrag{H}{$H$}
    \rsdraw{0.55}{0.75}{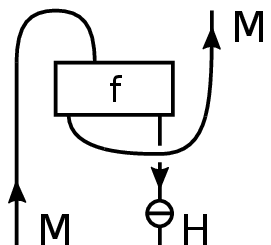}
  \; = \,
    \psfrag{M}{$M$}
    \psfrag{f}{$r$}
    \psfrag{H}{$H$}
    \rsdraw{0.55}{0.75}{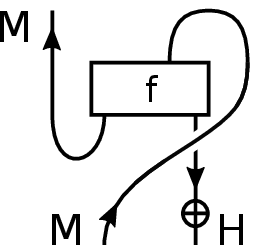}
  \; \co M^* \tens H \to M^*,$$ together with the (co)evaluation morphisms inherited from $\mc B$:
 $$\evl_{(M,r)}=\evl_M, \quad \evr_{(M,r)}=\evr_M, \quad \coevl_{(M,r)}=\coevl_M, \quad \coevr_{(M,r)}=\coevr_M.$$
Note that the last equality in the definition of $r^\dag$ follows from the involutivity of $H$.

The ca\-te\-go\-ry~$\Mod_H$ has a coend~$((C,a),j)$, where~$C=H^*\tens H$, the ac\-tion~$a\co C \tens H \to C$ of~$H$ on~$C$ is computed by $$a=\,
\psfrag{H}{$H$}
\rsdraw{0.45}{0.75}{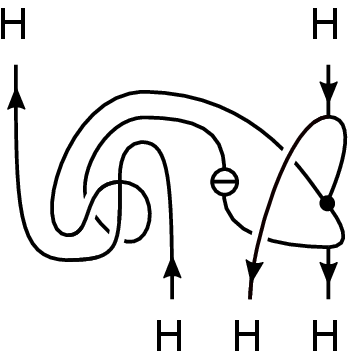}\;,$$
and the universal dinatural transformation $j=\{j_{(M,r)} \co (M,r)^*\tens (M,r) \to (C,a)\}_{(M,r)\in \Mod_H}$ is given by
$$j_{(M,r)}=\,
\psfrag{H}{$H$}
\psfrag{M}{$M$}
\rsdraw{0.45}{0.75}{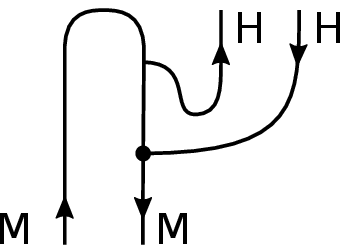}\;.$$
By Section \ref{CoendRibbon}, the coend $C$ is a Hopf algebra in $\Mod_H$. In particular, it is a coalgebra in $\Mod_H$.
The comultiplication~$\Delta_C\co C \to C\tens C$ and the counit~$\varepsilon_C\co C\to \uu$ of~$C$ are computed by $$\Delta_C=
\,
\psfrag{H}{$H$}
\rsdraw{0.45}{0.75}{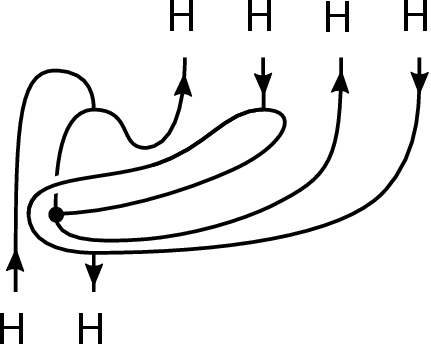}\; \quad \text{and} \quad \varepsilon_C=\,
\psfrag{H}{$H$}
\rsdraw{0.45}{0.75}{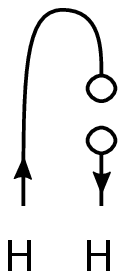}\;.
$$

The following lemma gives a way to produce an~$\varepsilon$-invariant~$u$-trace~$\uu \to C$, where~$\varepsilon$ and~$u$ denote the counit and unit of~$H$.
\begin{lem} \label{centersandtraces} If a morphism~$\kappa \co \uu \to H$ in $\mc B$ satisfies \begin{equation*}
\,
\psfrag{g}{$\kappa$}
\psfrag{A}{\hspace{-0.15cm}$H$}
\psfrag{s}{$\sigma$}
\rsdraw{0.45}{0.75}{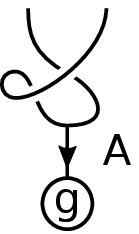}\;=\,
\psfrag{g}{$\kappa$}
\psfrag{A}{\hspace{-0.15cm}$H$}
\psfrag{s}{$\sigma$}
\rsdraw{0.45}{0.75}{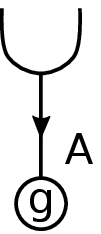}\;,\end{equation*}
then the morphism $$\alpha=
\,
\psfrag{C}{$H$}
\psfrag{f}{$\kappa$}
\rsdraw{0.45}{0.75}{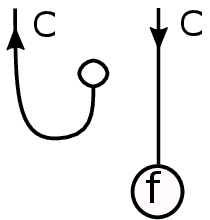}
 \;
 \co \uu \to C$$ is an $\varepsilon$-invariant $u$-trace.
\end{lem}
\begin{proof}
Denote~$f=\varepsilon^*\tens \id_H \co H \to C$. Let us first check that~$f$ is a morphism between right~$H$-modules~$(H,\varepsilon)$ and~$(C,a)$. Indeed,
\begin{align*}
  \,
\psfrag{A}{$\color{red}C$}
\psfrag{u}{\hspace{-0.1cm}$H$}
\psfrag{a}{$a$}
\psfrag{g}{$f$}
\rsdraw{0.35}{0.75}{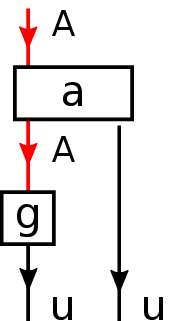}\;\overset{(i)}{=} & \,
\psfrag{H}{$H$}
\rsdraw{0.45}{0.75}{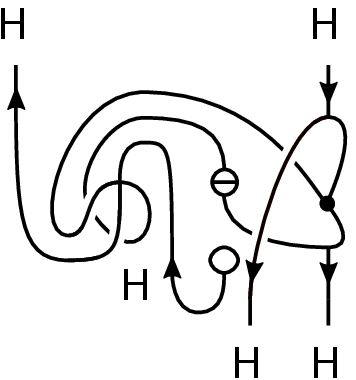}\;\overset{(ii)}{=} \,
\psfrag{H}{$H$}
\rsdraw{0.45}{0.75}{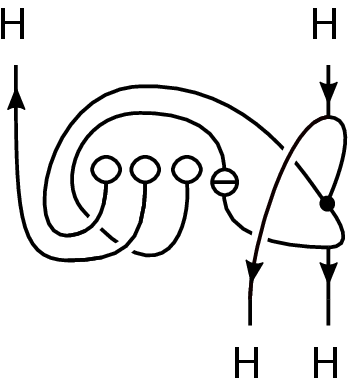}\;\overset{(iii)}{=} \,
\psfrag{H}{$H$}
\rsdraw{0.45}{0.75}{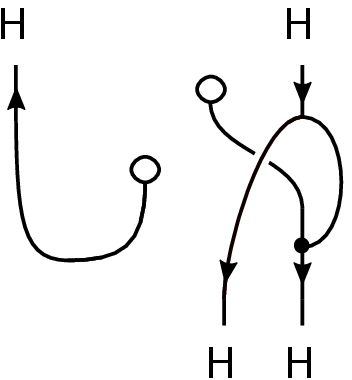}\;\overset{(iv)}{=}  \\
  \overset{(iv)}{=}  & \,
\psfrag{H}{$H$}
\rsdraw{0.45}{0.75}{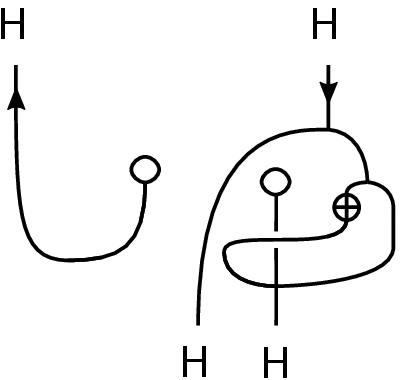}\; \overset{(v)}{=} \,
\psfrag{H}{$H$}
\rsdraw{0.45}{0.75}{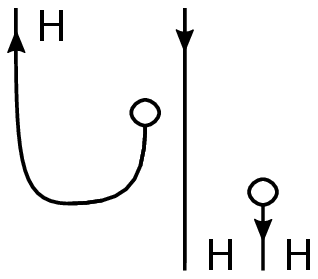}\; \overset{(vi)}{=} \,
\psfrag{A}{$\color{red} C$}
\psfrag{u}{\hspace{-0.1cm}$H$}
\psfrag{a}{$a$}
\psfrag{g}{$f$}
\rsdraw{0.45}{0.75}{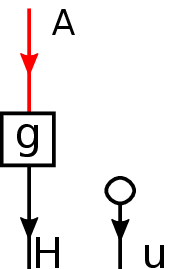}\;.
\end{align*}
Here~$(i)$ follows by definition of~$f$ and~$a$,~$(ii)$ follows from the isotopy of the graphical calculus for pivotal categories and the multiplicativity of the counit,~$(iii)$ by the naturality of the braiding, the fact that~$\varepsilon S=\varepsilon$, and the counitality,~$(iv)$ since universal coaction of~$H$ on itself is the right coadjoint coaction,~$(v)$ by the naturality of the braiding, the (co)unitality and the antipode axiom,~$(vi)$ follows by definition of~$f$.

Next, $f \co H\to C$ is a coalgebra morphism. Indeed, we have:
\begin{align*}
  \,
\psfrag{H}{$H$}
\psfrag{A}{$\color{red} C$}
\psfrag{g}{$f$}
\rsdraw{0.45}{0.75}{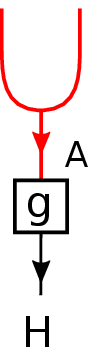}\;&\overset{(i)}{=} \,
\psfrag{H}{$H$}
\rsdraw{0.45}{0.75}{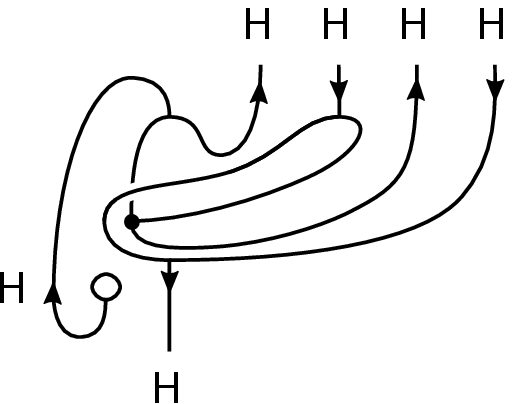}\;\overset{(ii)}{=}\,
\psfrag{H}{$H$}
\rsdraw{0.45}{0.75}{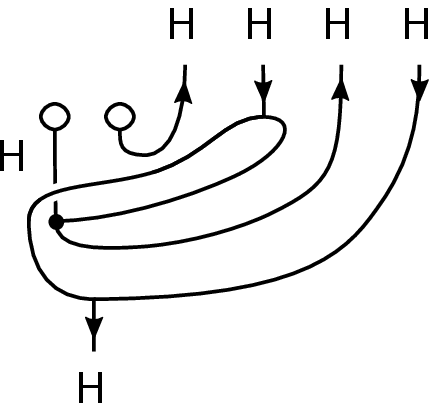}\;\overset{(iii)}{=} \\
   &\overset{(iii)}{=} \hspace{0.5cm}\,
\psfrag{H}{$H$}
\rsdraw{0.45}{0.75}{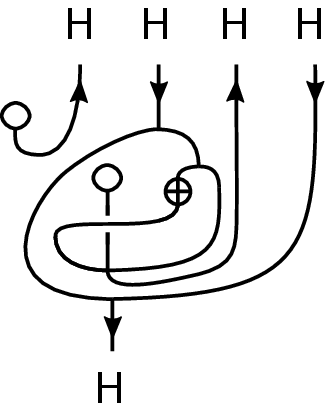}\; \hspace{0.7cm}\overset{(iv)}{=}\hspace{0.5cm}\,
\psfrag{H}{$H$}
\rsdraw{0.45}{0.75}{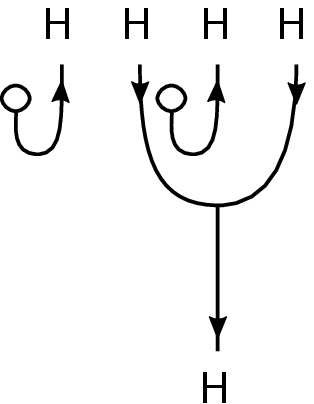}\;\hspace{0.35cm}\overset{(v)}{=}\,
\psfrag{H}{$H$}
\psfrag{A}{$\color{red} C$}
\psfrag{g}{$f$}
\rsdraw{0.45}{0.75}{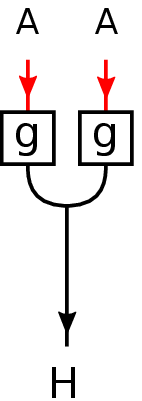}\;.
\end{align*}
Here $(i)$ and $(v)$ follow by definition of $f$, $(ii)$ from the isotopy of the graphical
calculus for pivotal categories and the multiplicativity of the counit, $(iii)$ by the naturality of the braiding and the fact that universal coaction of $H$ on itself is the right coadjoint coaction, $(iv)$ follows by the naturality of the braiding, the (co)unitality, and the antipode axiom. Also,
$$\,
\psfrag{H}{$H$}
\psfrag{A}{$\color{red}C$}
\psfrag{g}{$f$}
\rsdraw{0.45}{0.75}{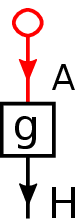}\;\overset{(i)}{=}\,
\psfrag{H}{$H$}
\psfrag{A}{$C$}
\psfrag{g}{$f$}
\rsdraw{0.45}{0.75}{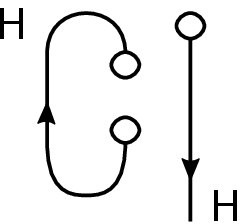}\;\overset{(ii)}{=}\,
\psfrag{H}{$H$}
\psfrag{A}{$C$}
\psfrag{g}{$f$}
\rsdraw{0.45}{0.75}{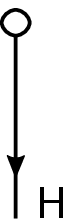}\;.$$
Here $(i)$ follow from definitions of $f$ and $\varepsilon_C$ and $(ii)$ from the fact that $\varepsilon u=\id_\uu$.

Finally, let us show that~$\alpha$ is an~$\varepsilon$-invariant~$u$-trace.
By definition of~$\alpha$ and the fact that~$f\co H \to C$ is an~$H$-module morphism, we have that
$$\,
\psfrag{g}{$\color{red}\alpha$}
\psfrag{a}{$a$}
\psfrag{A}{\hspace{-0.1cm}$\color{red}C$}
\psfrag{u}{\hspace{-0.1cm}$H$}
\rsdraw{0.45}{0.75}{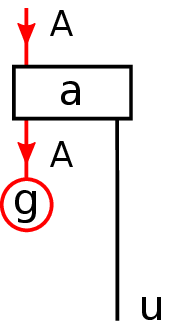}\;=\,
\psfrag{g}{$f$}
\psfrag{A}{\hspace{-0.1cm}$\color{red}C$}
\psfrag{u}{\hspace{-0.1cm}$H$}
\psfrag{r}{\hspace{-0.07cm}$\kappa$}
\rsdraw{0.45}{0.75}{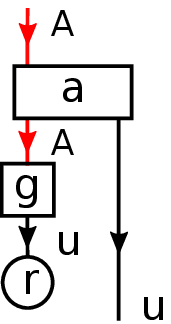}\;=\,
\psfrag{g}{$f$}
\psfrag{A}{\hspace{-0.1cm}$\color{red}C$}
\psfrag{u}{\hspace{-0.1cm}$H$}
\psfrag{r}{\hspace{-0.07cm}$\kappa$}
\rsdraw{0.45}{0.75}{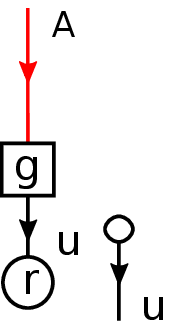}\;=\,
\psfrag{g}{$\color{red}\alpha$}
\psfrag{A}{\hspace{-0.1cm}$\color{red}C$}
\psfrag{u}{\hspace{-0.1cm}$H$}
\rsdraw{0.45}{0.75}{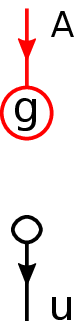}\;.$$
Thus the morphism~$\alpha$ is $\varepsilon$-invariant.
It remains to show that~$\alpha=f\kappa$ is a $u$-trace. Indeed,
$$
\,
\psfrag{f}{$\color{red}\alpha$}
\psfrag{C}{$\color{red}C$}
\rsdraw{0.45}{0.75}{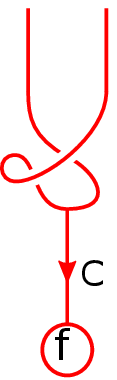}\; \overset{(i)}{=}
\,
\psfrag{g}{$f$}
\psfrag{r}{\hspace{-0.1cm}$\kappa$}
\psfrag{u}{$H$}
\psfrag{A}{$\color{red}C$}
\rsdraw{0.45}{0.75}{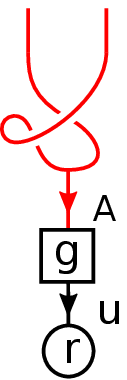}\;\overset{(ii)}{=}
\,
\psfrag{g}{$f$}
\psfrag{r}{\hspace{-0.1cm}$\kappa$}
\psfrag{u}{$H$}
\psfrag{A}{$\color{red}C$}
\rsdraw{0.45}{0.75}{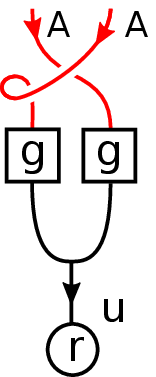}\;\overset{(iii)}{=}
\,
\psfrag{g}{$f$}
\psfrag{r}{\hspace{-0.1cm}$\kappa$}
\psfrag{u}{$H$}
\psfrag{A}{$\color{red}C$}
\rsdraw{0.45}{0.75}{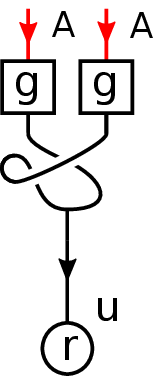}\;\overset{(iv)}{=}
\,
\psfrag{g}{$f$}
\psfrag{r}{\hspace{-0.1cm}$\kappa$}
\psfrag{u}{$H$}
\psfrag{A}{$\color{red}C$}
\rsdraw{0.45}{0.75}{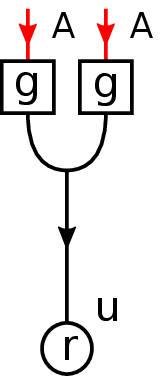}\;\overset{(v)}{=}\,
\psfrag{g}{$f$}
\psfrag{r}{\hspace{-0.1cm}$\kappa$}
\psfrag{u}{$H$}
\psfrag{A}{$\color{red}C$}
\rsdraw{0.45}{0.75}{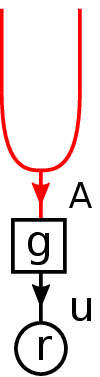}\;\overset{(vi)}{=}\,
\psfrag{f}{$\color{red}\alpha$}
\psfrag{r}{\hspace{-0.1cm}$\kappa$}
\psfrag{u}{$H$}
\psfrag{A}{$\color{red}C$}
\rsdraw{0.45}{0.75}{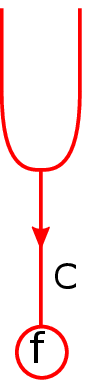}\;.
$$
Here~$(i)$ and~$(vi)$ follow from definition,~$(ii)$ and~$(v)$ follow by the fact that~$f\co H\to C$ is a coalgebra morphism,~$(iii)$ follows by the naturality of twist and the braiding, and~$(iv)$ follows by hypothesis on~$\kappa$.
\end{proof}

Any coalgebra morphism $\uu \to H$ satisfies the condition of Lemma \ref{centersandtraces}. Another family of examples satisfying the condition of Lemma \ref{centersandtraces} is given as follows: for any $X \in \Ob{\mc B}$, set
$$
\kappa^X=\,
\psfrag{l}{\hspace{0.25cm}$\sigma^C_X$}
\psfrag{C}{$H$}
\psfrag{X}{$X$}
\psfrag{H}{$H$}
\rsdraw{0.45}{0.75}{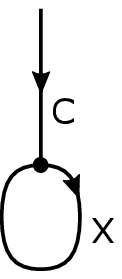}
\;=i_X\coevr_X\co \uu \to H.
$$
Then,
$$
\,
\psfrag{g}{\hspace{-0.1cm}$\kappa^X$}
\psfrag{C}{$H$}
\psfrag{X}{$X$}
\rsdraw{0.45}{0.75}{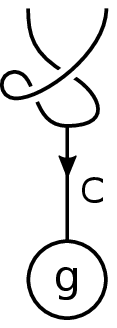}
\;\overset{(i)}{=}\,
\psfrag{C}{$H$}
\psfrag{X}{$X$}
\rsdraw{0.45}{0.75}{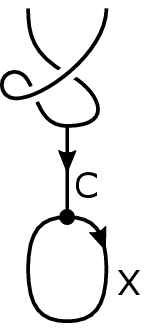}
\;\overset{(ii)}{=}\,
\psfrag{C}{$H$}
\psfrag{X}{$X$}
\rsdraw{0.45}{0.75}{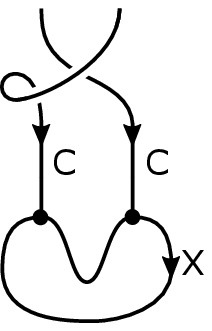}
\;\overset{(iii)}{=}\,
\psfrag{C}{$H$}
\psfrag{X}{$X$}
\rsdraw{0.45}{0.75}{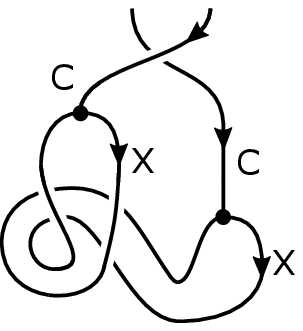}
\;\overset{(iv)}{=}\,
\psfrag{C}{$H$}
\psfrag{X}{$X$}
\rsdraw{0.45}{0.75}{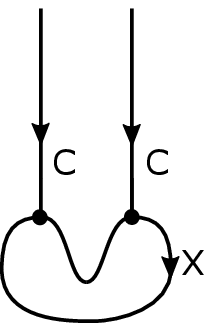}
\; \overset{(v)}{=}
\,
\psfrag{X}{$X$}
\psfrag{C}{$H$}
\rsdraw{0.45}{0.75}{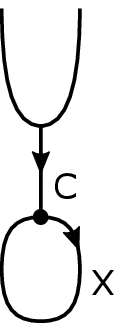}
\;\overset{(vi)}{=}\,
\psfrag{C}{$H$}
\psfrag{g}{\hspace{-0.1cm}$\kappa^X$}
\rsdraw{0.45}{0.75}{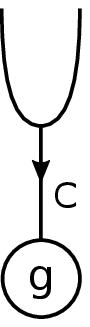}
\;.$$
Here, $(i)$ and $(vi)$ follow by definition of~$\kappa^X$, $(ii)$ and $(v)$ follow by definition of comultiplication of $H$, $(iii)$ by the naturality of twists, and $(iv)$ by the naturality of the braiding and isotopy invariance of graphical calculus.

Thus $\kappa^X$ satisfies the condition of Lemma \ref{centersandtraces} and so
\begin{equation} \label{alphaX}\alpha^X= \,
\psfrag{C}{$H$}
\psfrag{X}{$X$}
\rsdraw{0.45}{0.75}{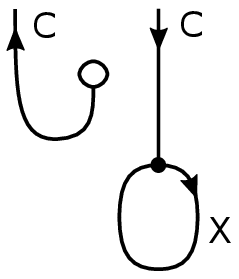}
\;\co \uu \to C\end{equation} is an $\varepsilon$-invariant $u$-trace for $C$.
\begin{rmk} \normalfont Note that if $\mc B$ is $\kk$-linear, then any linear combination of $\delta$-invariant $\sigma$-traces is a $\delta$-invariant $\sigma$-trace. In particular, an in\-te\-re\-sting e\-xam\-ple of an~$\varepsilon$-in\-var\-i\-ant~$u$-trace comes from to\-po\-lo\-gi\-cal field the\-ory: if~$\mc B$ is a ribbon fu\-sion $\kk$-linear ca\-te\-gory and~$I$ is a re\-pre\-sen\-ta\-tive set of simple ob\-jects of~$\mc B$, then~$\alpha=\sum_{k\in I}\dim(k)\alpha^k$ is an $\varepsilon$-invariant $u$-trace. Here $\alpha^k$ is defined in \eqref{alphaX} and~$\dim(k)=\evr_k\coevl_k=\evl_k\coevr_k$ is the~dimension of~$k$.
\end{rmk}

\section{Proof of Theorem \ref*{powers}} \label{proofmain}
Our strategy to compute the $(n+1)$-th power of the paracocyclic operator $\tau_n(\delta,\sigma)$ is similar to the proof of cocyclicity condition from Connes and Moscovici in~\cite{connes_cyclic_1999}, where Hopf algebras over $\mathbb{C}$ are considered. We indeed proceed by induction. The difficulty  here is that the paracocyclic operators involve the braiding.
In our approach, based on graphical calculus, we manage to keep track the powers of paracocyclic operators. In Section~\ref{prepa} we list algebraic properties used in our proof of the equalities from Theo\-rem~\ref{powers}. In Section~\ref{subsec: kthpower} we show Formula~\eqref{kthpower}. In Section~\ref{finalstep} we show Formula~\eqref{n+1thpower}.

Recall that $H$ denotes a Hopf algebra in the braided monoidal ca\-te\-go\-ry $\mc B$, $\delta \co H\to \uu$ is an algebra morphism and $\sigma \co \uu \to H$ is a coalgebra morphism such that $\delta \sigma=\id_\uu$.
Given such a pair, we define the \emph{twisted antipode}~$\tilde{S} \co H \to H$ by \[\tilde{S}=
\,
\psfrag{b}{\hspace{-0.075cm}$H$}
\psfrag{l}{\hspace{-0.075cm}$\delta$}
\rsdraw{0.45}{0.75}{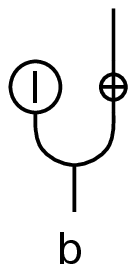}
\;.\]
For brevity, we denote the twisted antipode $\tilde{S}$ graphically by $\,
\psfrag{b}{$H$}
\psfrag{v}[tl][tl]{$\thicksim$}
\rsdraw{0.45}{0.75}{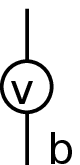}
\;$.
With this notation, we will rewrite
$$\tau_n(\delta,\sigma)=\,
\psfrag{b}{\hspace{-0.1cm}$H$}
\psfrag{i}{\hspace{-0.25cm}$H^{\tens n-1}$}
\psfrag{v}[tl][tl]{$\thicksim$}
\psfrag{c}{$\sigma$}
\rsdraw{0.55}{0.75}{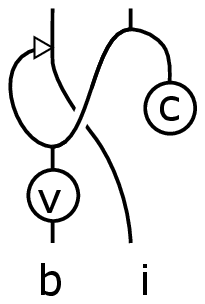}
\; \quad \text{if} \quad n\geq 1.$$
Similarly, equation \eqref{kthpower}, which is to be proven, rewrites as
$$\left(\tau_n(\delta, \sigma)\right)^k=\,
\psfrag{z}{\hspace{-0.75cm}$H^{\tens k-2}$}
\psfrag{c}{\hspace{-0.25cm}$k-1$}
\psfrag{c+}{$k$}
\psfrag{n}{$n$}
\psfrag{f}{\hspace{-0.6cm}$\left(\tau_{n-1}(\varepsilon,u)\right)^{k-1}$}
\psfrag{g}{\hspace{-0.8cm}$\tau_1(\delta, \sigma)$}
\psfrag{l}{\hspace{-0.1cm}$\delta$}
\psfrag{s}{$\sigma$}
\psfrag{h}{\hspace{-0.6cm}$H^{\tens n-k}$}
\psfrag{d}{\hspace{-0.5cm}$H^{\tens k-1}$}
\psfrag{v}[tl][tl]{$\thicksim$}
\rsdraw{0.55}{0.75}{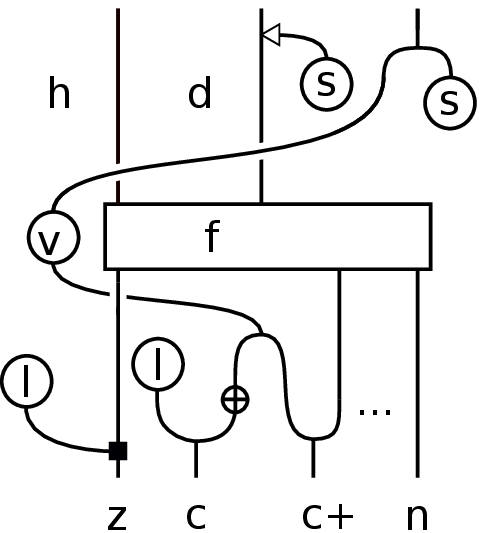}  \;.$$

\subsection{Preliminary facts} \label{prepa}
In this section we state several lemmas, which are used in the proof of Theorem \ref{powers}. We mention that equalities $(a)$ and $(b)$ from Lemma \ref{algebraicppties} and the equality from Remark \ref{twistedantipodesigma} are already stated in \cite[Proposition 4.3]{khalkhalipourkia}.
In the lemma that follows, some properties of the twisted antipode are established.
\begin{lem} \label{algebraicppties} The following equalities hold:
$$
\begingroup
\allowdisplaybreaks
\begin{tabular}{ccccccc} \hspace{1.25cm}
&\,
\psfrag{b}{$H$}
\psfrag{v}[tl][tl]{$\thicksim$}
\rsdraw{0.55}{0.75}{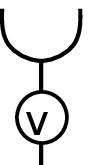}\;  =  \,
\psfrag{b}{$H$}
\psfrag{v}[tl][tl]{$\thicksim$}
\rsdraw{0.55}{0.75}{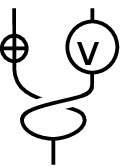}\;&, &\,
\psfrag{b}{$H$}
\psfrag{v}[tl][tl]{$\thicksim$}
\rsdraw{0.55}{0.75}{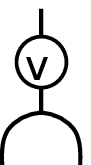}\;=\,
\psfrag{b}{$H$}
\psfrag{v}[tl][tl]{$\thicksim$}
\rsdraw{0.55}{0.75}{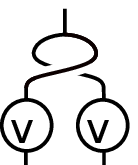}\;&, &\,
\psfrag{b}{$H$}
\psfrag{v}[tl][tl]{$\thicksim$}
\rsdraw{0.55}{0.75}{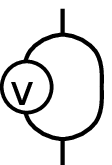}\;=\,
\psfrag{b}{$H$}
\psfrag{v}{\hspace{0.075cm}$\delta$}
\rsdraw{0.55}{0.75}{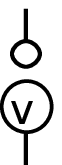}\;,& \\ &&&&&&\\
  &(a)& &(b)&  &(c)&\\ &&&&&& \\
&\,
\psfrag{b}{$H$}
\psfrag{l}{\hspace{-0.05cm}$\delta$}
\psfrag{v}[tl][tl]{$\thicksim$}
\rsdraw{0.55}{0.75}{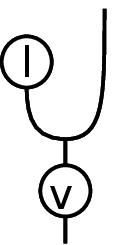}\;=
\,
\psfrag{b}{$H$}
\psfrag{l}{\hspace{-0.05cm}$\delta$}
\psfrag{v}{\hspace{0.05cm}$\delta$}
\rsdraw{0.55}{0.75}{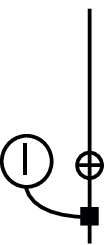}\;&, &\,
\psfrag{b}{$H$}
\psfrag{v}[tl][tl]{$\thicksim$}
\rsdraw{0.55}{0.75}{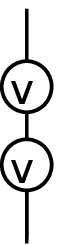}\;=\,
\psfrag{b}{$H$}
\psfrag{l}{\hspace{-0.05cm}$\delta$}
\rsdraw{0.55}{0.75}{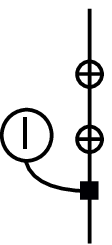}\;&, &\,
\psfrag{b}{$H$}
\psfrag{v}[tl][tl]{$\thicksim$}
\rsdraw{0.55}{0.75}{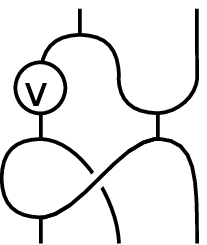}\;=\,
\psfrag{b}{$H$}
\psfrag{l}{\hspace{-0.1cm}$\delta$}
\psfrag{v}[tl][tl]{$\thicksim$}
\rsdraw{0.55}{0.75}{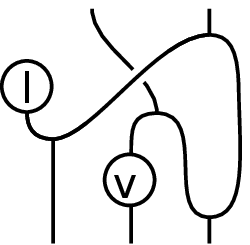}\;&.
\\&&&&&&\\
&(d)& &(e)&  &(f)&
\end{tabular}
\endgroup
$$
\end{lem}

\begin{proof}
Let us first show the relation~$(a)$.
Indeed, by definition of $\tilde{S}$, the anti-comultiplicativity of the antipode, the coassociativity, and the naturality of the braiding, we have
\[\,
\psfrag{b}{$H$}
\psfrag{v}[tl][tl]{$\thicksim$}
\rsdraw{0.55}{0.75}{twistedantipodecomult0nolab}\;=\,
\psfrag{b}{$H$}
\psfrag{l}{\hspace{-0.05cm}$\delta$}
\rsdraw{0.55}{0.75}{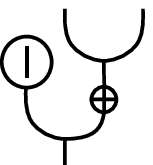}\;=\,
\psfrag{b}{$H$}
\psfrag{l}{\hspace{-0.05cm}$\delta$}
\rsdraw{0.55}{0.75}{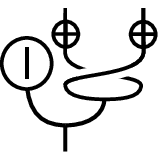}\;=\,
\psfrag{b}{$H$}
\psfrag{l}{\hspace{-0.05cm}$\delta$}
\rsdraw{0.55}{0.75}{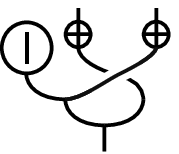}\;=\,
\psfrag{b}{$H$}
\psfrag{v}[tl][tl]{$\thicksim$}
\rsdraw{0.55}{0.75}{twistedantipodecomult1nolab}\;.\]
Let us show the relation~$(b)$.
Indeed, by definition of $\tilde{S}$, the fact that comultiplication is an algebra morphism, the fact that $\delta$ is an algebra morphism, and the naturality of the braiding we have
\[\,
\psfrag{b}{$H$}
\psfrag{v}[tl][tl]{$\thicksim$}
\rsdraw{0.55}{0.75}{twistedantipodemult0nolab}\;=\,
\psfrag{b}{$H$}
\psfrag{l}{\hspace{-0.05cm}$\delta$}
\rsdraw{0.55}{0.75}{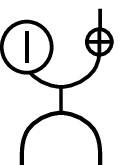}\;=\,
\psfrag{b}{$H$}
\psfrag{l}{\hspace{-0.05cm}$\delta$}
\rsdraw{0.55}{0.75}{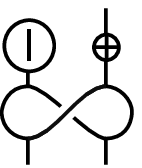}\;=\,
\psfrag{b}{$H$}
\psfrag{l}{\hspace{-0.05cm}$\delta$}
\rsdraw{0.55}{0.75}{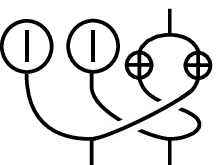}\;=\,
\psfrag{b}{$H$}
\psfrag{v}[tl][tl]{$\thicksim$}
\rsdraw{0.55}{0.75}{twistedantipodemult1nolab}\;.\]
Let us show the relation~$(c)$.
Indeed, this relation follows by the definition of $\tilde{S}$, the coassociativity, the antipode axiom, and the counitality:
\[\,
\psfrag{b}{$H$}
\psfrag{v}[tl][tl]{$\thicksim$}
\rsdraw{0.55}{0.75}{antipodeaxiomtwistednolab}\;=
\,\psfrag{b}{$H$}
\psfrag{l}{\hspace{-0.05cm}$\delta$}
\rsdraw{0.55}{0.75}{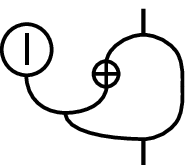}\;=
\,\psfrag{b}{$H$}
\psfrag{l}{\hspace{-0.05cm}$\delta$}
\rsdraw{0.55}{0.75}{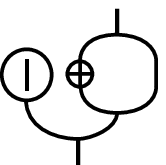}\;=
\,\psfrag{b}{$H$}
\psfrag{l}{\hspace{-0.05cm}$\delta$}
\rsdraw{0.55}{0.75}{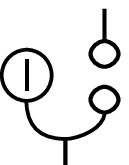}=
\,\psfrag{b}{$H$}
\psfrag{v}{\hspace{0.075cm}$\delta$}
\rsdraw{0.55}{0.75}{antipodeaxiomtwisted1nolab}\;.\]
Now we show the equality~$(d)$.
It follows by the part~$(a)$, the naturality of the braiding, the definition of~$\tilde{S}$, the fact that~$\delta$ is an algebra morphism, and the definition of left coadjoint coaction:
\[\,
\psfrag{b}{$H$}
\psfrag{l}{\hspace{-0.05cm}$\delta$}
\psfrag{v}[tl][tl]{$\thicksim$}
\rsdraw{0.55}{0.75}{twistedantipodeadjoint0nolab}\;=\,
\psfrag{b}{$H$}
\psfrag{l}{\hspace{-0.05cm}$\delta$}
\psfrag{v}[tl][tl]{$\thicksim$}
\rsdraw{0.55}{0.75}{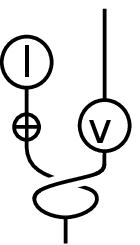}\;=\,
\psfrag{b}{$H$}
\psfrag{l}{\hspace{-0.05cm}$\delta$}
\psfrag{v}[tl][tl]{$\thicksim$}
\rsdraw{0.55}{0.75}{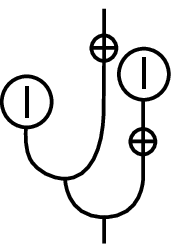}\;=\,
\psfrag{b}{$H$}
\psfrag{l}{\hspace{-0.05cm}$\delta$}
\psfrag{v}[tl][tl]{$\thicksim$}
\rsdraw{0.55}{0.75}{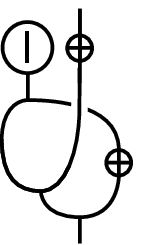}\;=\,
\psfrag{b}{$H$}
\psfrag{l}{\hspace{-0.05cm}$\delta$}
\rsdraw{0.55}{0.75}{twistedantipodeadjoint1nolab}\;.\]
The equality~$(e)$ is a consequence of the equality~$(d)$.
To see this, compose the left hand side of~$(d)$ with the antipode $S$ of $H$ and use the definition of $\tilde{S}$.

\noindent Finally, let us show the equation~$(f)$.
Indeed, this equation follows by the part~$(b)$, the fact that comultiplication is an algebra morphism, the (co)associativity, the naturality of the braiding, the part~$(c)$, and the unitality:
\[\,
\psfrag{b}{$H$}
\psfrag{v}[tl][tl]{$\thicksim$}
\rsdraw{0.55}{0.75}{specialnolab}\;=\,
\psfrag{b}{$H$}
\psfrag{v}[tl][tl]{$\thicksim$}
\rsdraw{0.55}{0.75}{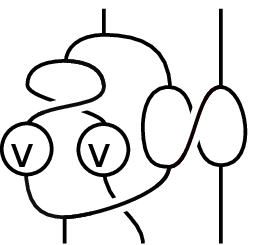}\;=\,
\psfrag{b}{$H$}
\psfrag{v}[tl][tl]{$\thicksim$}
\rsdraw{0.55}{0.75}{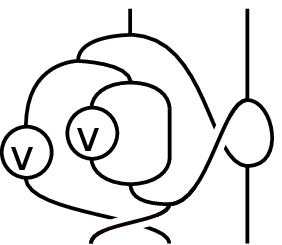}\;=\,
\psfrag{b}{$H$}
\psfrag{v}[tl][tl]{$\thicksim$}
\psfrag{l}{\hspace{-0.05cm}$\delta$}
\rsdraw{0.55}{0.75}{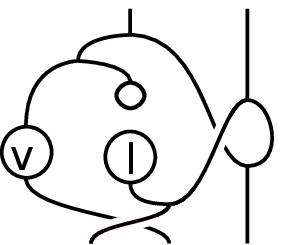}\;=\,
\psfrag{b}{$H$}
\psfrag{v}[tl][tl]{$\thicksim$}
\psfrag{l}{\hspace{-0.05cm}$\delta$}
\rsdraw{0.55}{0.75}{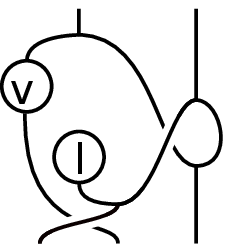}\;=\,
\psfrag{b}{$H$}
\psfrag{v}[tl][tl]{$\thicksim$}
\psfrag{l}{\hspace{-0.05cm}$\delta$}
\rsdraw{0.55}{0.75}{special1nolab}\;.\]
\end{proof}
\begin{rmk} \label{twistedantipodeeps} \normalfont
Another useful property of the twisted antipode $\tilde{S}$ is that $\varepsilon\tilde{S} =\delta$.
It follows by the definition of $\tilde{S}$, the fact that $\varepsilon S = \varepsilon$, and the counitality.
\end{rmk}
The following lemma gives the expression of the paracocyclic operator~$\tau_n(\delta,\sigma)$ in terms of~$\tau_{n-1}(\varepsilon,u)$.
\begin{lem} \label{recurrencerelations} If $n \ge 2$, then
\begingroup
\allowdisplaybreaks
\begin{itemize}
\item[\textit{(a)}] \hspace{3.5cm}$\tau_n(\delta,\sigma)=\,
\psfrag{b}{$H$}
\psfrag{i}{\hspace{-0.5cm}$H^{\tens n-2}$}
\psfrag{v}[tl][tl]{$\thicksim$}
\psfrag{s}{$\sigma$}
\psfrag{f}{$\hspace{-0.55cm}\tau_{n-1}(\varepsilon,u)$}
\rsdraw{0.55}{0.75}{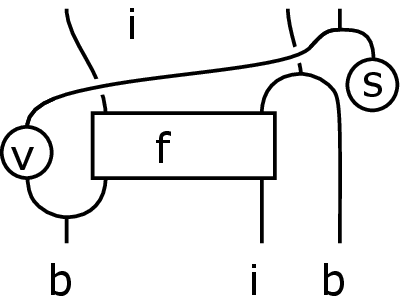}\;$,\vspace{0.5cm}
\item[\textit{(b)}] \hspace{3.5cm} $\tau_n(\delta,\sigma)=\,
\psfrag{b}{$H$}
\psfrag{i}{\hspace{-0.5cm}$H^{\tens n-2}$}
\psfrag{v}{\hspace{0.05cm}$\delta$}
\psfrag{s}{$\sigma$}
\psfrag{f}{$\hspace{-0.55cm}\tau_{n-1}(\varepsilon,u)$}
\rsdraw{0.55}{0.75}{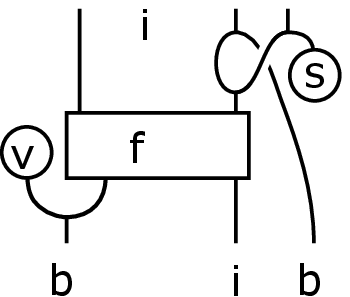}\;.$
\end{itemize}
\endgroup
\end{lem}
\begin{proof}
Let us first show the equation~$(a)$.
Indeed, by definition of $\tau_n(\delta,\sigma)$, Lemma~\ref{algebraicppties}$(a)$, the naturality of the braiding, inductive definition of the left diagonal action, and the definition of $\tau_{n-1}(\varepsilon,u)$, we have
\[\, \tau_n(\delta,\sigma)= \psfrag{b}{\hspace{-0.15cm}$H$}
\psfrag{i}{\hspace{-0.25cm}$H^{\tens n-1}$}
\psfrag{v}[tl][tl]{$\thicksim$}
\psfrag{c}{$\sigma$}
\rsdraw{0.55}{0.75}{CMtaunDS} \;= \,
\psfrag{b}{\hspace{-0.15cm}$H$}
\psfrag{i}{\hspace{-0.5cm}$H^{\tens n-1}$}
\psfrag{v}[tl][tl]{$\thicksim$}
\psfrag{s}{$\sigma$}
\rsdraw{0.55}{0.75}{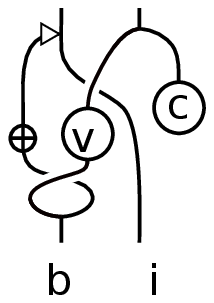}\;=\,
\psfrag{b}{\hspace{-0.05cm}$H$}
\psfrag{i}{\hspace{-0.5cm}$H^{\tens n-1}$}
\psfrag{v}[tl][tl]{$\thicksim$}
\psfrag{s}{$\sigma$}
\rsdraw{0.55}{0.75}{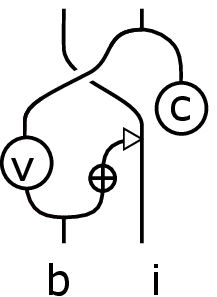}\; =\,
\psfrag{b}{\hspace{-0.1cm}$H$}
\psfrag{i}{\hspace{-0.75cm}$H^{\tens n-2}$}
\psfrag{v}[tl][tl]{$\thicksim$}
\psfrag{s}{$\sigma$}
\rsdraw{0.55}{0.75}{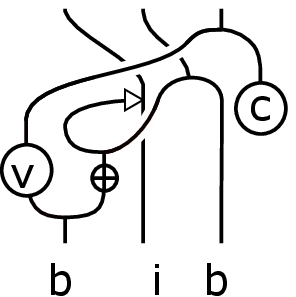}\;=\,
\psfrag{b}{$H$}
\psfrag{i}{\hspace{-0.5cm}$H^{\tens n-2}$}
\psfrag{v}[tl][tl]{$\thicksim$}
\psfrag{s}{$\sigma$}
\psfrag{f}{$\hspace{-0.55cm}\tau_{n-1}(\varepsilon,u)$}
\rsdraw{0.55}{0.75}{recctaun0}\;.\]
Further, we show the part~$(b)$.
For~$n=2$, the statement follows by definition.
From now on, suppose that~$n\ge 3$.
By the definition of~$\tau_n(\delta,\sigma)$, the definition of the left diagonal action, the coassociativity, and the definition of~$\tau_{n-1}(\varepsilon,u)$, we have:
\[\tau_n(\delta,\sigma)= \psfrag{b}{\hspace{-0.15cm}$H$}
\psfrag{i}{\hspace{-0.25cm}$H^{\tens n-1}$}
\psfrag{v}[tl][tl]{$\thicksim$}
\psfrag{c}{$\sigma$}
\rsdraw{0.55}{0.75}{CMtaunDS} \; = \,
\psfrag{b}{\hspace{0.05cm}$1$}
\psfrag{i}{\hspace{-0.1cm}$2$}
\psfrag{v}[tl][tl]{$\thicksim$}
\psfrag{c}{$\sigma$}
\psfrag{e}{\hspace{-0.4cm}$n-1$}
\psfrag{r}{\hspace{-0.1cm}$n$}
\rsdraw{0.55}{0.75}{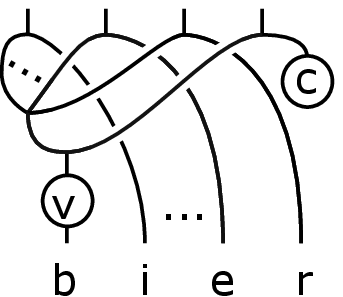}\;=\,
\psfrag{b}{\hspace{0.05cm}$1$}
\psfrag{i}{\hspace{-0.05cm}$2$}
\psfrag{v}[tl][tl]{$\thicksim$}
\psfrag{c}{$\sigma$}
\psfrag{e}{\hspace{-0.3cm}$n-1$}
\psfrag{r}{$n$}
\psfrag{l}{\hspace{-0.05cm}$\delta$}
\rsdraw{0.55}{0.75}{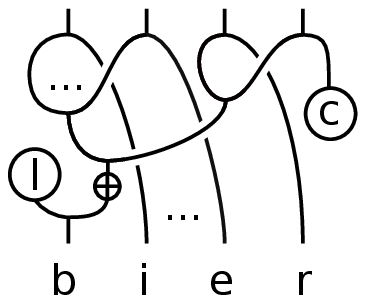}\;=\,
\psfrag{b}{$H$}
\psfrag{i}{\hspace{-0.5cm}$H^{\tens n-2}$}
\psfrag{v}{\hspace{0.05cm}$\delta$}
\psfrag{s}{$\sigma$}
\psfrag{f}{$\hspace{-0.55cm}\tau_{n-1}(\varepsilon,u)$}
\rsdraw{0.55}{0.75}{recctaun1}\;.\]
\end{proof}
The equalities stated in the following lemma are used in computation of squares of the paracocyclic operator $\tau_n(\delta,\sigma)$ in the case $n\ge 3$.
\begin{lem} \label{squarelemma} For any $n\ge 2$, we have:
$$
\begin{tabular}{ccccc}
&\,
\psfrag{b}{$H^{\tens n}$}
\psfrag{i}{\hspace{-0.25cm}$H^{\tens n-1}$}
\psfrag{f}{$\hspace{-0.35cm}\tau_{n}(\varepsilon,u)$}
\rsdraw{0.55}{0.75}{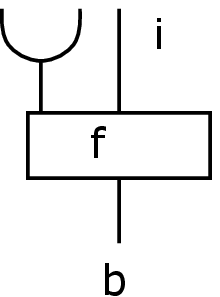}\;=\,
\psfrag{b}{$H$}
\psfrag{i}{\hspace{-0.25cm}$H^{\tens n-2}$}
\psfrag{r}{\hspace{-0.25cm}$H^{\tens n}$}
\psfrag{f}{$\hspace{-0.2cm}\tau_{n}(\varepsilon,u)$}
\rsdraw{0.55}{0.75}{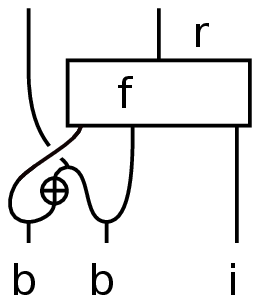}\;&,
&\,
\psfrag{b}{$H$}
\psfrag{i}{\hspace{-0.45cm}$H^{\tens n-1}$}
\psfrag{r}{\hspace{-0.1cm}$H^{\tens n}$}
\psfrag{f}{$\hspace{-0.35cm}\tau_{n}(\varepsilon,u)$}
\rsdraw{0.55}{0.75}{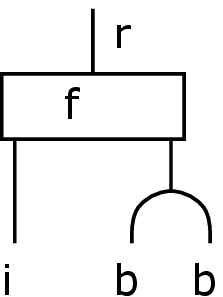}\;=
\,
\psfrag{i}{\hspace{-0.25cm}$H^{\tens n}$}
\psfrag{b}{$H$}
\psfrag{r}{\hspace{-0.1cm}$H^{\tens n-2}$}
\psfrag{f}{$\hspace{-0.25cm}\tau_{n}(\varepsilon,u)$}
\rsdraw{0.55}{0.75}{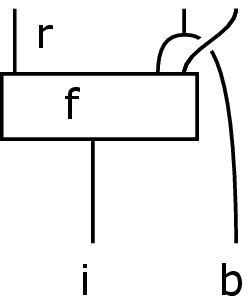}\;&, \\ &&&&\\&(a)& &(b)&\\&&&&\\
&\,
\psfrag{i}{\hspace{-0.25cm}$H^{\tens n-1}$}
\psfrag{b}{$H$}
\psfrag{r}{\hspace{-0.1cm}$H^{\tens n-1}$}
\psfrag{f}{$\hspace{-0.35cm}\tau_{n}(\varepsilon,u)$}
\rsdraw{0.55}{0.75}{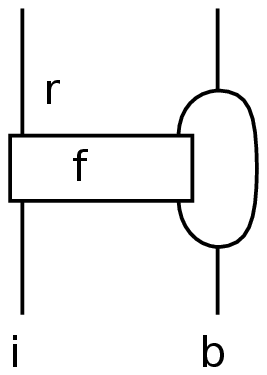}\;=\,
\psfrag{i}{\hspace{-0.25cm}$H^{\tens n-1}$}
\psfrag{b}{$H$}
\psfrag{r}{\hspace{-0.1cm}$H^{\tens n-2}$}
\psfrag{f}{$\hspace{-0.35cm}\tau_{n-1}(\varepsilon,u)$}
\rsdraw{0.55}{0.75}{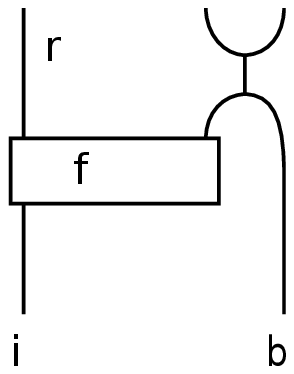}\;&,
&\,
\psfrag{i}{\hspace{-0.15cm}$H^{\tens n}$}
\psfrag{b}{$H$}
\psfrag{l}{\hspace{-0.15cm}$H^{\tens n-1}$}
\psfrag{r}{\hspace{-0.1cm}$H^{\tens n-2}$}
\psfrag{f}{$\hspace{-0.35cm}\tau_{n-1}(\varepsilon,u)$}
\psfrag{g}{$\hspace{-0.35cm}\tau_{n}(\varepsilon,u)$}
\rsdraw{0.55}{0.75}{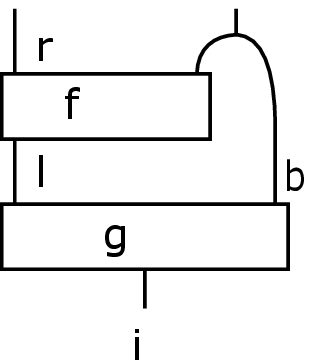}\; = \,
\psfrag{i}{\hspace{-0.15cm}$H^{\tens n}$}
\psfrag{b}{$H$}
\psfrag{r}{\hspace{-0.1cm}$H^{\tens n-1}$}
\psfrag{f}{$\hspace{-0.35cm}\tau_{n-1}(\varepsilon,u)$}
\rsdraw{0.55}{0.75}{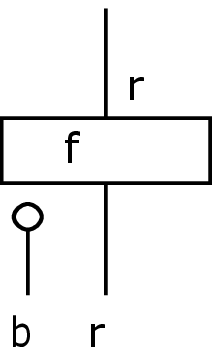}\;&. \\ &&&&\\&(c)& &(d)&
\end{tabular}$$
\end{lem}

\begin{proof}
We begin by showing the equality~$(a)$.
Let us first inspect the case~$n=2$.
To see that the equality is true in this case, we use the definition of~$\tau_2(\varepsilon,u)$, the fact that comultiplication is an algebra morphism, the coassociativity, the anti-comultiplicativity of the antipode, and the naturality of the braiding:
\[\,
\psfrag{b}{$H^{\tens 2}$}
\psfrag{f}{$\hspace{-0.35cm}\tau_{2}(\varepsilon,u)$}
\rsdraw{0.55}{0.75}{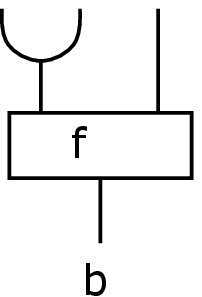}\;=\,
\psfrag{b}{$H$}
\psfrag{f}{$\hspace{-0.35cm}\tau_{2}(\varepsilon,u)$}
\rsdraw{0.55}{0.75}{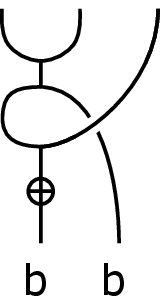}\;=
\,
\psfrag{h}{$H$}
\psfrag{f}{$\hspace{-0.35cm}\tau_{2}(\varepsilon,u)$}
\rsdraw{0.55}{0.75}{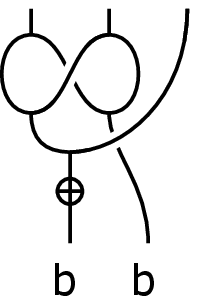}\;=\,
\psfrag{h}{$H$}
\psfrag{f}{$\hspace{-0.35cm}\tau_{2}(\varepsilon,u)$}
\rsdraw{0.55}{0.75}{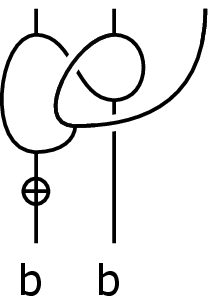}\;=\,
\psfrag{h}{$H$}
\psfrag{f}{$\hspace{-0.35cm}\tau_{2}(\varepsilon,u)$}
\rsdraw{0.55}{0.75}{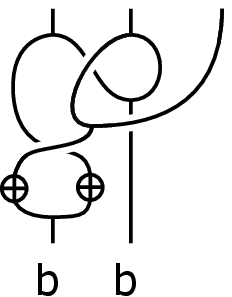}\;=\,
\psfrag{b}{$H$}
\psfrag{r}{\hspace{-0.1cm}$H^{\tens 2}$}
\psfrag{f}{$\hspace{-0.25cm}\tau_{2}(\varepsilon,u)$}
\rsdraw{0.55}{0.75}{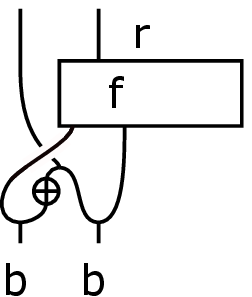}\;.\] From now on, suppose that $n\ge 3$. By definition of $\tau_n(\varepsilon,u)$, the fact that comultiplication is an algebra morphism, the naturality of the braiding, the coassociativity, and the anti-comultiplicativity of the antipode we have:
\begingroup
\allowdisplaybreaks
\begin{align*}
\,
\psfrag{b}{$H^{\tens n}$}
\psfrag{i}{\hspace{-0.25cm}$H^{\tens n-1}$}
\psfrag{f}{$\hspace{-0.35cm}\tau_{n}(\varepsilon,u)$}
\rsdraw{0.55}{0.75}{intertauncomult0}\;&=\,
\psfrag{b}{$1$}
\psfrag{i}{$2$}
\psfrag{e}{$n-1$}
\psfrag{r}{\hspace{0.1cm}$n$}
\rsdraw{0.55}{0.75}{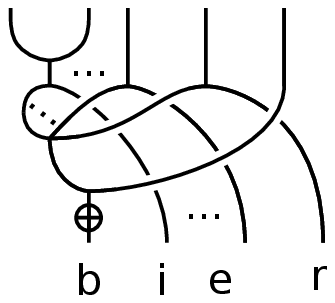}\;=\,
\psfrag{b}{$1$}
\psfrag{i}{$2$}
\psfrag{e}{$n-1$}
\psfrag{r}{\hspace{0.1cm}$n$}
\rsdraw{0.55}{0.75}{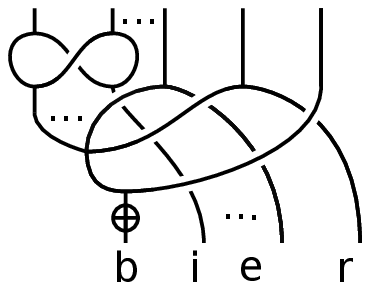}\;=\,
\psfrag{b}{$1$}
\psfrag{i}{$2$}
\psfrag{e}{$n-1$}
\psfrag{r}{\hspace{0.1cm}$n$}
\rsdraw{0.55}{0.75}{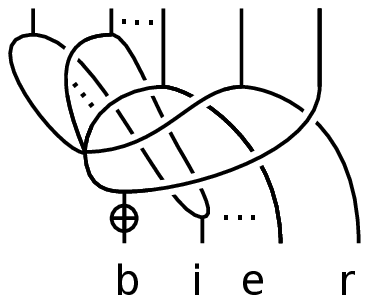}\;\\
&=\,
\psfrag{b}{$1$}
\psfrag{i}{$2$}
\psfrag{e}{$n-1$}
\psfrag{r}{\hspace{0.1cm}$n$}
\rsdraw{0.55}{0.75}{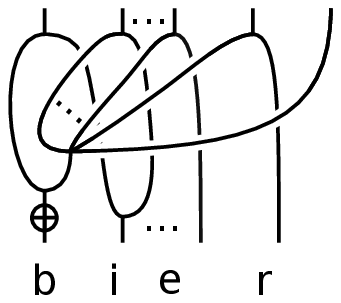}\;=\,
\psfrag{b}{$1$}
\psfrag{i}{$2$}
\psfrag{e}{$n-1$}
\psfrag{r}{\hspace{0.1cm}$n$}
\rsdraw{0.55}{0.75}{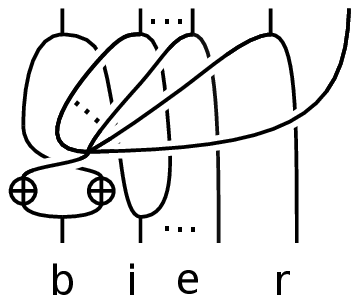}\;=\,
\psfrag{b}{$H$}
\psfrag{i}{\hspace{-0.25cm}$H^{\tens n-2}$}
\psfrag{r}{\hspace{-0.25cm}$H^{\tens n}$}
\psfrag{f}{$\hspace{-0.2cm}\tau_{n}(\varepsilon,u)$}
\rsdraw{0.55}{0.75}{intertauncomult1}\;.
\end{align*}
\endgroup
Let us show the equality~$(b)$. Indeed, it follows by definition of $\tau_n(\varepsilon,u)$, the naturality of the braiding and the associativity:
\[\,
\psfrag{b}{$H$}
\psfrag{i}{\hspace{-0.45cm}$H^{\tens n-1}$}
\psfrag{r}{\hspace{-0.1cm}$H^{\tens n}$}
\psfrag{f}{$\hspace{-0.35cm}\tau_{n}(\varepsilon,u)$}
\rsdraw{0.55}{0.75}{intertaunmult0}\;=\,
\psfrag{b}{$1$}
\psfrag{i}{$2$}
\psfrag{e}{$n-1$}
\psfrag{n}{\hspace{0.1cm}$n$}
\psfrag{n+1}{$n+1$}
\rsdraw{0.55}{0.75}{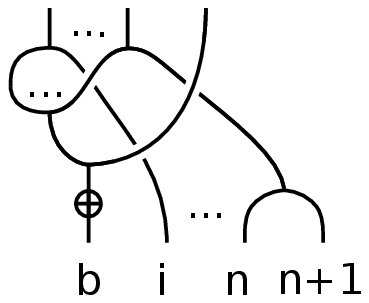}\;
=\,
\psfrag{b}{$1$}
\psfrag{i}{$2$}
\psfrag{e}{$n-1$}
\psfrag{n}{\hspace{0.1cm}$n$}
\psfrag{n+1}{$n+1$}
\rsdraw{0.55}{0.75}{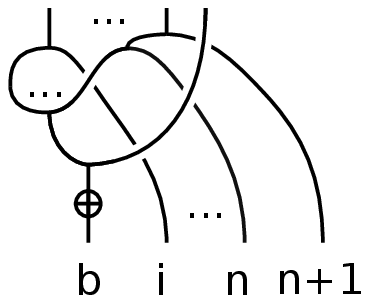}\;=\,
\psfrag{i}{\hspace{-0.25cm}$H^{\tens n}$}
\psfrag{b}{$H$}
\psfrag{r}{\hspace{-0.1cm}$H^{\tens n-2}$}
\psfrag{f}{$\hspace{-0.25cm}\tau_{n}(\varepsilon,u)$}
\rsdraw{0.55}{0.75}{intertaunmult1}\;.\]
Let us show the equality~$(c)$.
It follows from Lemma \ref{recurrencerelations}$(b)$ applied on~$\delta=\varepsilon$ and~$\sigma=u$ and by the fact that multiplication is an algebra morphism:
\[\,
\psfrag{i}{\hspace{-0.25cm}$H^{\tens n-1}$}
\psfrag{b}{$H$}
\psfrag{r}{\hspace{-0.1cm}$H^{\tens n-1}$}
\psfrag{f}{$\hspace{-0.35cm}\tau_{n}(\varepsilon,u)$}
\rsdraw{0.55}{0.75}{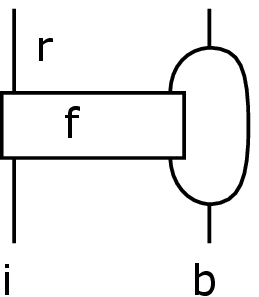}\;=
\,
\psfrag{i}{\hspace{-0.4cm}$H^{\tens n-2}$}
\psfrag{b}{$H$}
\psfrag{f}{$\hspace{-0.7cm}\tau_{n-1}(\varepsilon,u)$}
\psfrag{l}{\hspace{-0.3cm}$H^{\tens n-1}$}
\rsdraw{0.55}{0.75}{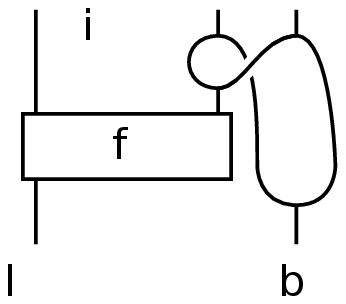}\;=\,
\psfrag{i}{\hspace{-0.4cm}$H^{\tens n-2}$}
\psfrag{b}{$H$}
\psfrag{f}{$\hspace{-0.7cm}\tau_{n-1}(\varepsilon,u)$}
\psfrag{l}{\hspace{-0.3cm}$H^{\tens n-1}$}
\rsdraw{0.55}{0.75}{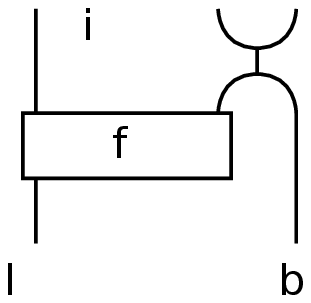}\;.\]
Finally, let us show the equality~$(d)$. Indeed, we have:
\begingroup
\allowdisplaybreaks
\begin{align*}\,
\psfrag{i}{\hspace{-0.15cm}$H^{\tens n}$}
\psfrag{b}{$H$}
\psfrag{l}{\hspace{-0.15cm}$H^{\tens n-1}$}
\psfrag{r}{\hspace{-0.1cm}$H^{\tens n-2}$}
\psfrag{f}{$\hspace{-0.35cm}\tau_{n-1}(\varepsilon,u)$}
\psfrag{g}{$\hspace{-0.35cm}\tau_{n}(\varepsilon,u)$}
\rsdraw{0.55}{0.75}{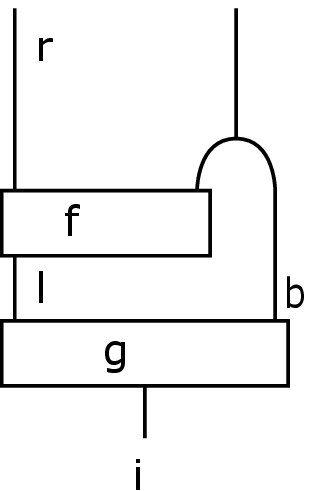}\;&\overset{(i)}{=}
\,
\psfrag{i}{\hspace{-0.15cm}$H^{\tens n-1}$}
\psfrag{b}{$H$}
\psfrag{l}[cc][cc][0.75]{\hspace{0.55cm}$H^{\tens n-2}$}
\psfrag{f}{$\id_{H^{\tens n-1}}$}
\psfrag{g}{$\hspace{-0.35cm}\tau_{n}(\varepsilon,u)$}
\rsdraw{0.55}{0.75}{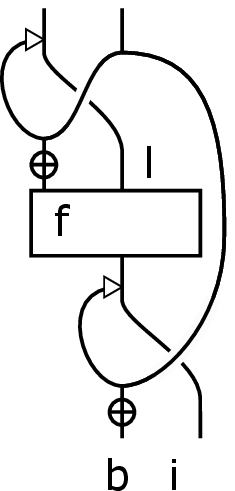}\;\overset{(ii)}{=}
\,
\psfrag{i}{\hspace{-0.15cm}$H^{\tens n-2}$}
\psfrag{b}{$H$}
\psfrag{g}{$\hspace{-0.35cm}\tau_{n}(\varepsilon,u)$}
\rsdraw{0.55}{0.75}{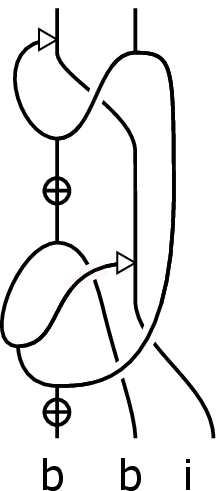}\;\overset{(iii)}{=}
\,
\psfrag{i}{\hspace{-0.15cm}$H^{\tens n-2}$}
\psfrag{b}{$H$}
\psfrag{g}{$\hspace{-0.35cm}\tau_{n}(\varepsilon,u)$}
\rsdraw{0.55}{0.75}{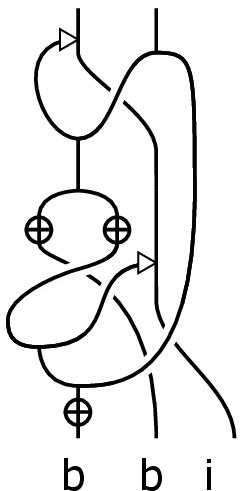}\;\overset{(iv)}{=}
\,
\psfrag{i}{\hspace{-0.15cm}$H^{\tens n-2}$}
\psfrag{b}{$H$}
\psfrag{g}{$\hspace{-0.35cm}\tau_{n}(\varepsilon,u)$}
\rsdraw{0.55}{0.75}{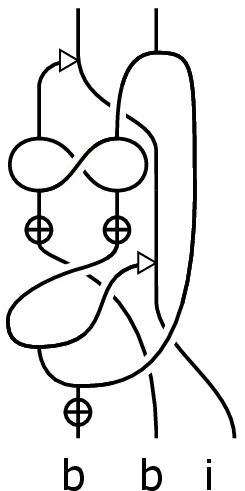}\;\overset{(v)}{=}\\&\overset{(v)}{=} \,
\psfrag{i}{\hspace{-0.15cm}$H^{\tens n-2}$}
\psfrag{b}{$H$}
\psfrag{g}{$\hspace{-0.35cm}\tau_{n}(\varepsilon,u)$}
\rsdraw{0.55}{0.75}{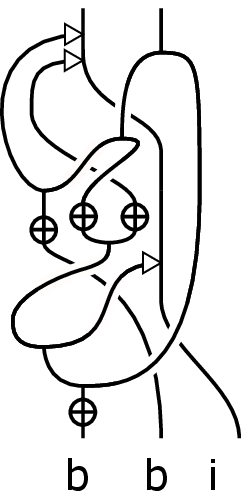}\;\overset{(vi)}{=}\,
\psfrag{i}{\hspace{-0.15cm}$H^{\tens n-2}$}
\psfrag{b}{$H$}
\psfrag{g}{$\hspace{-0.35cm}\tau_{n}(\varepsilon,u)$}
\rsdraw{0.55}{0.75}{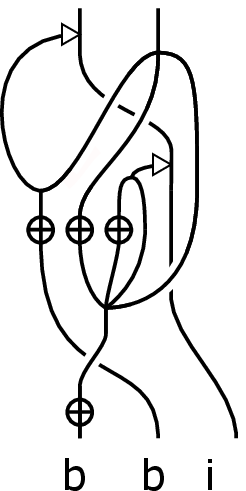}\;\overset{(vii)}{=}\,
\psfrag{i}{\hspace{-0.15cm}$H^{\tens n-2}$}
\psfrag{b}{$H$}
\psfrag{g}{$\hspace{-0.35cm}\tau_{n}(\varepsilon,u)$}
\rsdraw{0.55}{0.75}{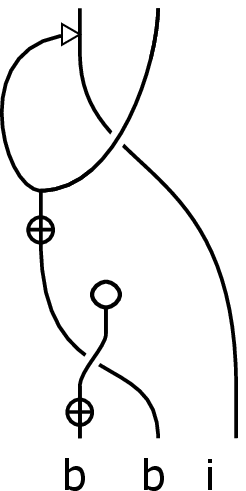}\;\overset{(viii)}{=}\,
\psfrag{i}{\hspace{-0.15cm}$H^{\tens n}$}
\psfrag{b}{$H$}
\psfrag{r}{\hspace{-0.1cm}$H^{\tens n-1}$}
\psfrag{f}{$\hspace{-0.35cm}\tau_{n-1}(\varepsilon,u)$}
\rsdraw{0.55}{0.75}{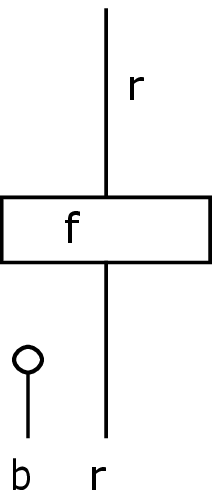}\;.
\end{align*}
\endgroup
Here~$(i)$ follows by definition of~$\tau_n(\varepsilon,u)$ and~$\tau_{n-1}(\varepsilon,u)$,~$(ii)$ follows by inductive definition of left diagonal action,~$(iii)$ follows by the anti-multiplicativity of the antipode,~$(iv)$ follows by the fact that multiplication is an algebra morphism,~$(v)$ follows by the axiom of a module and the anti-comultiplicativity of the antipode,~$(vi)$ follows by the naturality of the braiding, the (co)associativity, and by the axiom of a module,~$(vii)$ follows by applying the antipode axiom twice and by the axiom of a module,~$(viii)$ follows by the fact that~$\varepsilon S=\varepsilon$, the naturality of the braiding, and definition of~$\tau_{n-1}(\varepsilon,u)$.
\end{proof}
The equalities from the following lemma show how the endomorphism $m(\id_H\tens \sigma)$ interacts with the paracocyclic operator $\tau_n(\varepsilon,u)$.
These equalities are intensively used while proving Formula \eqref{n+1thpower} by using Formula \eqref{kthpower} of Theorem \ref{powers}.
\begin{lem} \label{flow} We have:
\begingroup
\allowdisplaybreaks
\begin{enumerate}
\item[\textit{(a)}]  For any $n\ge 1$,
$$\,
\psfrag{j}{$H$}
\psfrag{e}{$H^{\tens n-1}$}
\psfrag{s}{$\sigma$}
\psfrag{p}{$H^{\tens n}$}
\psfrag{f}{\hspace{-0.5cm}$\tau_{n}(\varepsilon,u)$}
\rsdraw{0.55}{0.75}{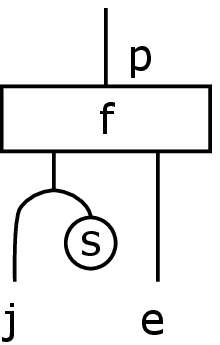}\;=\,
\psfrag{m}{$H^{\tens n}$}
\psfrag{r}{$H^{\tens n}$}
\psfrag{s}{$\sigma$}
\psfrag{f}{$\hspace{-0.5cm}\tau_{n}(\varepsilon,u)$}
\rsdraw{0.55}{0.75}{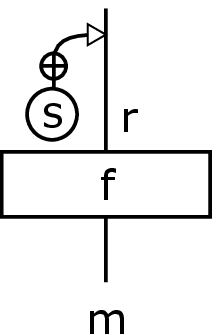}\;.$$

\item[\textit{(b)}] For $2\le j\le n$,
$$\,
\psfrag{b}{$1$}
\psfrag{n}{$2$}
\psfrag{j}{$j$}
\psfrag{e}{$n$}
\psfrag{s}{$\sigma$}
\psfrag{p}{$H^{\tens n}$}
\psfrag{f}{$\hspace{0.45cm}\tau_{n}(\varepsilon,u)$}
\rsdraw{0.55}{0.75}{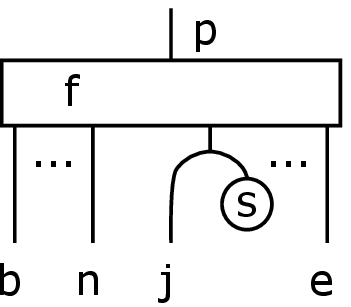}\;=\,
\psfrag{b}{$1$}
\psfrag{n}{\hspace{-0.4cm}$n-1$}
\psfrag{r}{\hspace{-0.35cm}$j-1$}
\psfrag{e}{$n$}
\psfrag{s}{\hspace{-0.05cm}$\sigma$}
\psfrag{m}{$H^{\tens n}$}
\psfrag{f}{$\hspace{-0.5cm}\tau_{n}(\varepsilon,u)$}
\rsdraw{0.55}{0.75}{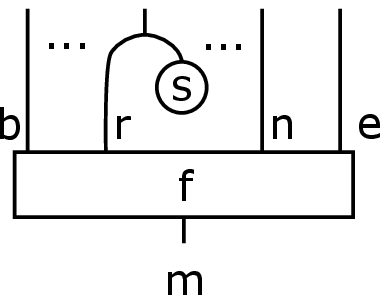}\;.$$
\end{enumerate}
\endgroup
\end{lem}

\begin{proof} Let us prove the part~$(a)$. Indeed, we have:
\[\,
\psfrag{j}{$H$}
\psfrag{e}{\hspace{-0.5cm}$H^{\tens n-1}$}
\psfrag{s}{$\sigma$}
\psfrag{p}{$H^{\tens n}$}
\psfrag{f}{\hspace{-0.5cm}$\tau_{n}(\varepsilon,u)$}
\rsdraw{0.55}{0.75}{flow2}\;\overset{(i)}{=}\,
\psfrag{j}{$H$}
\psfrag{e}{\hspace{-0.5cm}$H^{\tens n-1}$}
\psfrag{s}{$\sigma$}
\rsdraw{0.55}{0.75}{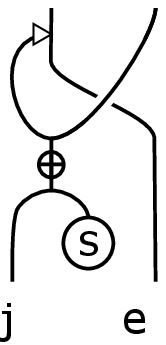}\;\overset{(ii)}{=}\,
\psfrag{j}{\hspace{-0.15cm}$H$}
\psfrag{e}{\hspace{-0.32cm}$H^{\tens n-1}$}
\psfrag{s}{$\sigma$}
\rsdraw{0.55}{0.75}{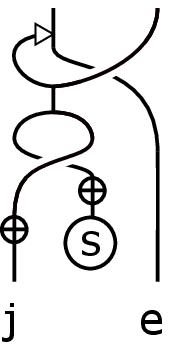}\;\overset{(iii)}{=}\,
\psfrag{j}{\hspace{-0.15cm}$H$}
\psfrag{e}{\hspace{-0.32cm}$H^{\tens n-1}$}
\psfrag{s}{$\sigma$}
\rsdraw{0.55}{0.75}{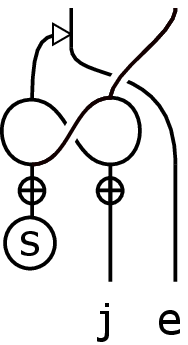}\;\overset{(iv)}{=}\,
\psfrag{j}{\hspace{-0.15cm}$H$}
\psfrag{e}{\hspace{-0.32cm}$H^{\tens n-1}$}
\psfrag{s}{$\sigma$}
\rsdraw{0.55}{0.75}{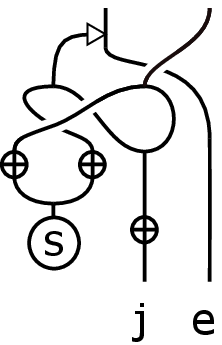}\;\overset{(v)}{=}\,
\psfrag{j}{\hspace{-0.15cm}$H$}
\psfrag{e}{\hspace{-0.32cm}$H^{\tens n-1}$}
\psfrag{s}{$\sigma$}
\rsdraw{0.55}{0.75}{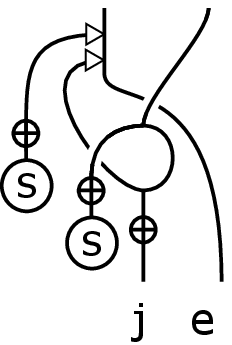}\;\overset{(vi)}{=}
\,
\psfrag{m}{$H^{\tens n}$}
\psfrag{r}{$H^{\tens n}$}
\psfrag{s}{$\sigma$}
\psfrag{f}{$\hspace{-0.5cm}\tau_{n}(\varepsilon,u)$}
\rsdraw{0.55}{0.75}{flow3}\;. \]
Here~$(i)$ follows by definition of~$\tau_n(\varepsilon, u)$,~$(ii)$ follows by the anti-multiplicativity of antipode,~$(iii)$ follows by the naturality of the braiding and the fact that multiplication is an algebra morphism,~$(iv)$ follows by the anti-comultiplicativity of antipode,~$(v)$ follows by the fact that~$\sigma$ is a coalgebra morphism and by the axiom of a module and~$(vi)$ follows by definition of~$\tau_n(\varepsilon, u)$ and the naturality of the braiding.

Let us now show the part~$(b)$.
Indeed, by definition of~$\tau_n(\varepsilon,u)$ and the left diagonal action, by the naturality of the braiding, and the associativity, we have:
\[\,
\psfrag{b}{$1$}
\psfrag{n}{$2$}
\psfrag{j}{$j$}
\psfrag{e}{$n$}
\psfrag{s}{$\sigma$}
\psfrag{p}{$H^{\tens n}$}
\psfrag{f}{$\hspace{0.45cm}\tau_{n}(\varepsilon,u)$}
\rsdraw{0.55}{0.75}{flow}\;=
\,
\psfrag{b}{$1$}
\psfrag{i}{$2$}
\psfrag{e}{$j$}
\psfrag{r}{$n$}
\psfrag{s}{$\sigma$}
\rsdraw{0.55}{0.75}{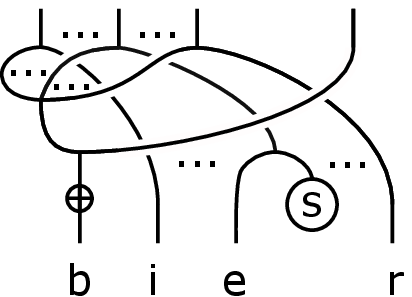}\;=\,
\psfrag{b}{$1$}
\psfrag{i}{$2$}
\psfrag{e}{$j$}
\psfrag{r}{$n$}
\psfrag{s}{$\sigma$}
\rsdraw{0.55}{0.75}{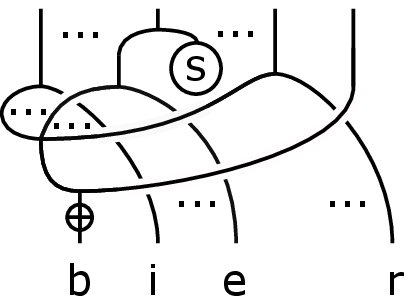}\;=\,
\psfrag{b}{$1$}
\psfrag{n}{\hspace{-0.4cm}$n-1$}
\psfrag{r}{\hspace{-0.35cm}$j-1$}
\psfrag{e}{$n$}
\psfrag{s}{\hspace{-0.05cm}$\sigma$}
\psfrag{m}{$H^{\tens n}$}
\psfrag{f}{$\hspace{-0.5cm}\tau_{n}(\varepsilon,u)$}
\rsdraw{0.55}{0.75}{flow1}\;.\]
\end{proof}
Before passing to the proof of Theorem \ref{powers}, let us state another auxilary lemma.
\begin{lem} \label{auxilia} We have the following assertions:
\begin{itemize}
\item[\textit{(a)}]
The morphism $\,\psfrag{b}{$H$}
\psfrag{l}{\hspace{-0.07cm}$\delta$}
\rsdraw{0.55}{0.75}{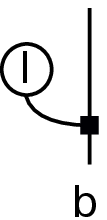}\;\co H\to H$ is a bialgebra morphism. \vspace{0.25cm}
\item[\textit{(b)}] For all $n\ge 1$,
$$\,\psfrag{b}{$1$}
\psfrag{e}{$n$}
\psfrag{h}{\hspace{-0.1cm}$H$}
\psfrag{n}{$H^{\tens n}$}
\psfrag{f}{$\id_{H^{\tens n}}$}
\psfrag{l}{\hspace{-0.05cm}$\delta$}
\rsdraw{0.55}{0.75}{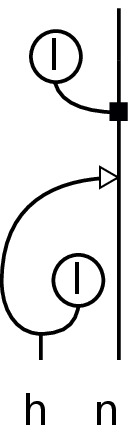}\;=\,\psfrag{b}{$1$}
\psfrag{e}{$n$}
\psfrag{h}{\hspace{-0.1cm}$H$}
\psfrag{n}{$H^{\tens n}$}
\psfrag{f}{$\id_{H^{\tens n}}$}
\psfrag{l}{\hspace{-0.05cm}$\delta$}
\rsdraw{0.55}{0.75}{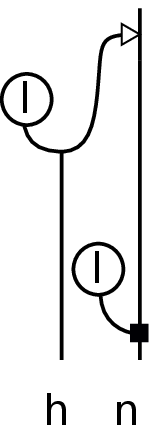}\;.$$
\item[\textit{(c)}] For all $n\ge 1$,
$$\,
\psfrag{r}{\hspace{-0.35cm}$H^{\tens n}$}
\psfrag{s}{$\sigma$}
\rsdraw{0.55}{0.75}{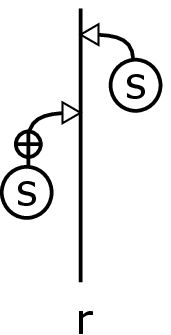}\;=\,
\psfrag{r}{\hspace{-0.35cm}$H^{\tens n}$}
\psfrag{s}{$\sigma$}
\rsdraw{0.55}{0.75}{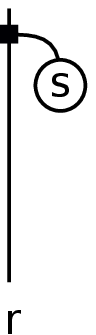}\;.$$
\end{itemize}
\end{lem}

The equality $(b)$ from Lemma \ref{auxilia} is intensively used while proving both of the equalities from Theorem \ref{powers}. The equality $(c)$ from Lemma \ref{auxilia} is particularly used in final steps of the computation of $\tau_n(\delta,\sigma)^{n+1}.$
\begin{proof} Let us first show the part $(a)$.
We first show that the morphism from part $(a)$ is an algebra morphism.
By using definition of the left coadjoint coaction, the bialgebra compatibility axiom, the fact that $\delta$ is an algebra morphism, the anti-multiplicativity of the antipode of $H$, and the naturality of the braiding, we have:
\[\,\psfrag{b}{$H$}
\psfrag{l}{\hspace{-0.07cm}$\delta$}
\rsdraw{0.55}{0.75}{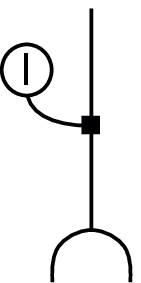}\;=\,\psfrag{b}{$H$}
\psfrag{l}{\hspace{-0.07cm}$\delta$}
\rsdraw{0.55}{0.75}{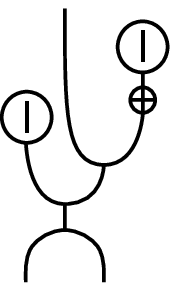}\;=\,\psfrag{b}{$H$}
\psfrag{l}{\hspace{-0.07cm}$\delta$}
\rsdraw{0.55}{0.75}{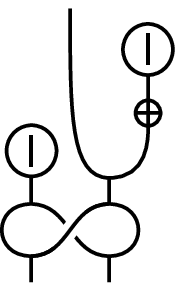}\;=\,\psfrag{b}{$H$}
\psfrag{l}{\hspace{-0.07cm}$\delta$}
\rsdraw{0.55}{0.75}{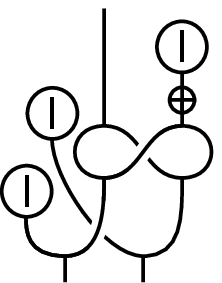}\;=\,\psfrag{b}{$H$}
\psfrag{l}{\hspace{-0.07cm}$\delta$}
\rsdraw{0.55}{0.75}{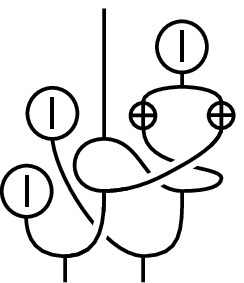}\;=\,\psfrag{b}{$H$}
\psfrag{l}{\hspace{-0.07cm}$\delta$}
\rsdraw{0.55}{0.75}{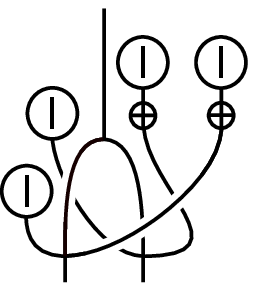}\;=\,\psfrag{b}{$H$}
\psfrag{l}{\hspace{-0.07cm}$\delta$}
\rsdraw{0.55}{0.75}{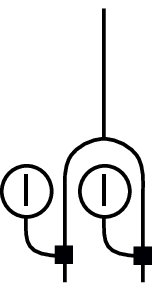}\;.\]
Similarly, by using definition of the left coadjoint coaction, the fact that unit is a coalgebra morphism, the fact that $\delta$ is an algebra morphism, and by the fact that $Su=u$, we have:
\[\,\psfrag{b}{$H$}
\psfrag{l}{\hspace{-0.07cm}$\delta$}
\rsdraw{0.55}{0.75}{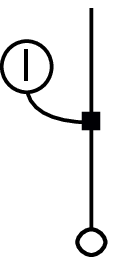}\;=\,\psfrag{b}{$H$}
\psfrag{l}{\hspace{-0.07cm}$\delta$}
\rsdraw{0.55}{0.75}{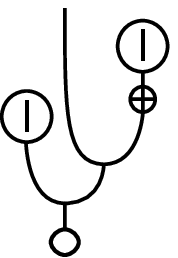}\;=\,\psfrag{b}{$H$}
\psfrag{l}{\hspace{-0.07cm}$\delta$}
\rsdraw{0.55}{0.75}{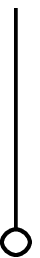}\;.\]

Let us now show that the morphism from $(a)$ is a coalgebra morphism. Indeed, by definition of the left coadjoint coaction, the fact that $\delta$ is an algebra morphism, the naturality of the braiding, the coassociativity, the antipode axiom, and the (co)unitality, we have:
$$
\,
\psfrag{l}{\hspace{-0.07cm}$\delta$}
\psfrag{b}{\hspace{-0.1cm}$H$}
\rsdraw{0.55}{0.75}{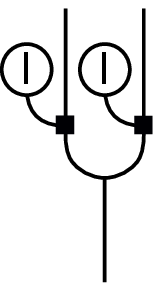}
\;=\,
\psfrag{l}{\hspace{-0.07cm}$\delta$}
\psfrag{b}{\hspace{-0.1cm}$H$}
\rsdraw{0.55}{0.75}{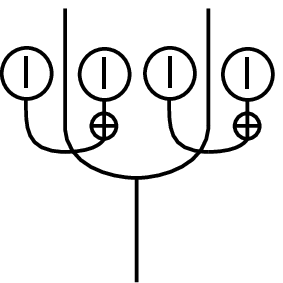}
\;=\,
\psfrag{l}{\hspace{-0.07cm}$\delta$}
\psfrag{b}{\hspace{-0.1cm}$H$}
\rsdraw{0.55}{0.75}{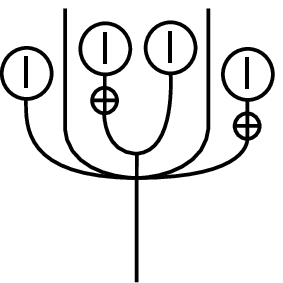}
\;=\,
\psfrag{l}{\hspace{-0.07cm}$\delta$}
\psfrag{b}{\hspace{-0.1cm}$H$}
\rsdraw{0.55}{0.75}{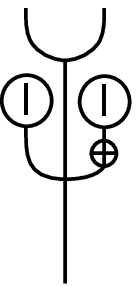}
\;=\,
\psfrag{l}{\hspace{-0.07cm}$\delta$}
\psfrag{b}{\hspace{-0.1cm}$H$}
\rsdraw{0.55}{0.75}{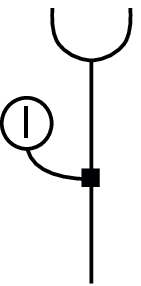}
\;.
$$
Furthermore, by definition of the left coadjoint coaction, the fact that $\delta$ is an algebra morphism, the naturality of the braiding, the (co)unitality, and the antipode axiom we have:
$$\,
\psfrag{l}{\hspace{-0.07cm}$\delta$}
\psfrag{b}{\hspace{-0.1cm}$H$}
\rsdraw{0.55}{0.75}{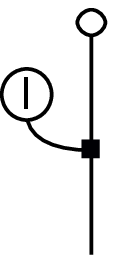}
\;=\,
\psfrag{l}{\hspace{-0.07cm}$\delta$}
\psfrag{b}{\hspace{-0.1cm}$H$}
\rsdraw{0.55}{0.75}{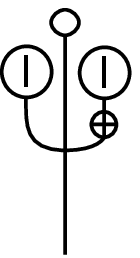}
\;=\,
\psfrag{l}{\hspace{-0.07cm}$\delta$}
\psfrag{b}{\hspace{-0.1cm}$H$}
\rsdraw{0.55}{0.75}{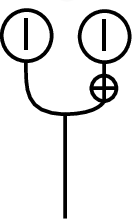}
\;=\,
\psfrag{l}{\hspace{-0.07cm}$\delta$}
\psfrag{b}{\hspace{-0.1cm}$H$}
\rsdraw{0.55}{0.75}{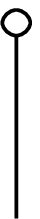}
\;.$$
This completes the proof of the part $(a)$.

Let us show the part $(b)$ by induction. For $n=1$, we prove the statement as follows. By using the part $(a)$, the definition of the coadjoint coaction, the fact that $\delta$ is an algebra morphism, the naturality of the braiding, the coassociativity, the antipode axiom, and the (co)unitality we have:
\[
\,
\psfrag{b}{$H$}
\psfrag{l}{\hspace{-0.07cm}$\delta$}
\rsdraw{0.55}{0.75}{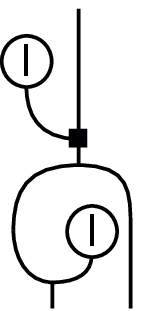}\;=\,
\psfrag{b}{$H$}
\psfrag{l}{\hspace{-0.07cm}$\delta$}
\rsdraw{0.55}{0.75}{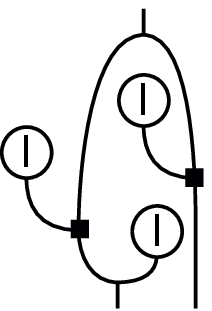}\;=\,
\psfrag{b}{$H$}
\psfrag{l}{\hspace{-0.07cm}$\delta$}
\rsdraw{0.55}{0.75}{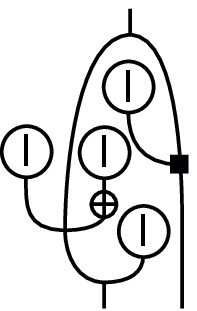}\;=\,
\psfrag{b}{$H$}
\psfrag{l}{\hspace{-0.07cm}$\delta$}
\rsdraw{0.55}{0.75}{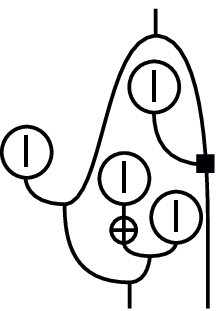}\;=\,
\psfrag{b}{$H$}
\psfrag{l}{\hspace{-0.07cm}$\delta$}
\rsdraw{0.55}{0.75}{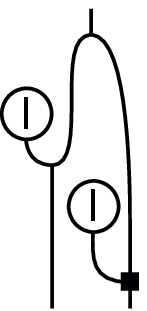}\;
\]
Suppose that the statement is true for an $n\ge 1$ and let us show it for $n+1$. We have:
\begingroup
\allowdisplaybreaks
\begin{align*}
\,\psfrag{h}{\hspace{-0.1cm}$H$}
\psfrag{n}{\hspace{-0.3cm}$H^{\tens n+1}$}
\psfrag{l}{\hspace{-0.05cm}$\delta$}
\rsdraw{0.55}{0.75}{adjauxadvanced0}\;&\overset{(i)}{=}\,
\psfrag{h}{\hspace{-0.1cm}$H$}
\psfrag{b}{$H$}
\psfrag{n}{\hspace{-0.1cm}$H^{\tens n}$}
\psfrag{l}{\hspace{-0.05cm}$\delta$}
\rsdraw{0.55}{0.75}{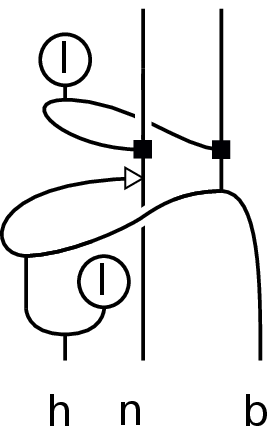}\;\overset{(ii)}{=}\,
\psfrag{h}{\hspace{-0.1cm}$H$}
\psfrag{b}{$H$}
\psfrag{n}{\hspace{-0.1cm}$H^{\tens n}$}
\psfrag{l}{\hspace{-0.05cm}$\delta$}
\rsdraw{0.55}{0.75}{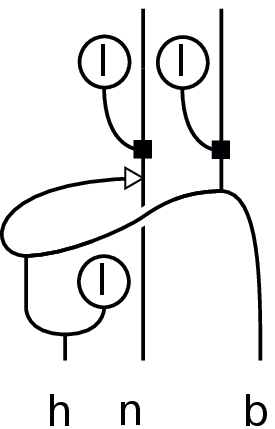}\;\overset{(iii)}{=}\,
\psfrag{h}{\hspace{-0.1cm}$H$}
\psfrag{b}{$H$}
\psfrag{n}{\hspace{-0.1cm}$H^{\tens n}$}
\psfrag{l}{\hspace{-0.05cm}$\delta$}
\rsdraw{0.55}{0.75}{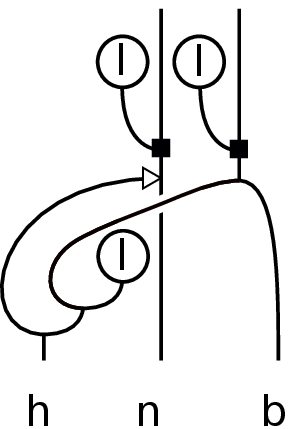}\;\overset{(iv)}{=}\,
\psfrag{h}{\hspace{-0.1cm}$H$}
\psfrag{b}{$H$}
\psfrag{n}{\hspace{-0.1cm}$H^{\tens n}$}
\psfrag{l}{\hspace{-0.05cm}$\delta$}
\rsdraw{0.55}{0.75}{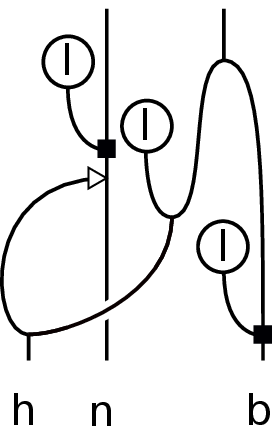}\;\overset{(v)}{=} \\
  &\overset{(v)}{=}\,
\psfrag{h}{\hspace{-0.1cm}$H$}
\psfrag{b}{$H$}
\psfrag{n}{\hspace{-0.1cm}$H^{\tens n}$}
\psfrag{l}{\hspace{-0.05cm}$\delta$}
\rsdraw{0.55}{0.75}{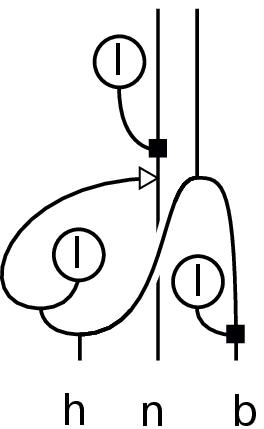}\;\overset{(vi)}{=}\,
\psfrag{h}{\hspace{-0.1cm}$H$}
\psfrag{b}{$H$}
\psfrag{n}{\hspace{-0.1cm}$H^{\tens n}$}
\psfrag{l}{\hspace{-0.05cm}$\delta$}
\rsdraw{0.55}{0.75}{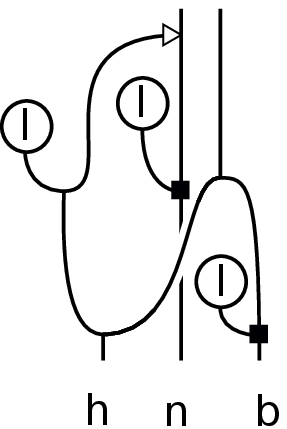}\;\overset{(vii)}{=}\,
\psfrag{h}{\hspace{-0.1cm}$H$}
\psfrag{b}{$H$}
\psfrag{n}{\hspace{-0.1cm}$H^{\tens n}$}
\psfrag{l}{\hspace{-0.05cm}$\delta$}
\rsdraw{0.55}{0.75}{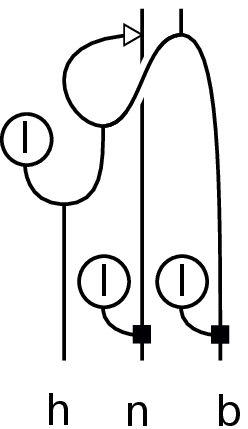}\;\overset{(viii)}{=}\,
\psfrag{h}{\hspace{-0.1cm}$H$}
\psfrag{b}{$H$}
\psfrag{n}{\hspace{-0.3cm}$H^{\tens n+1}$}
\psfrag{l}{\hspace{-0.05cm}$\delta$}
\rsdraw{0.55}{0.75}{adjauxadvanced1}\;,
\end{align*}
\endgroup
which shows the desired statement.
Here~$(i)$ and~$(viii)$ both follow by inductive definition of left diagonal action and left coadjoint coaction,~$(ii)$ follows by the fact that~$\delta$ is an algebra morphism,~$(iii)$ follows by the coassociativity,~$(iv)$ follows by the naturality of the braiding and the case~$n=1$,~$(v)$ and~$(vii)$ both follow by the naturality of the braiding and by the coassociativity and~$(vi)$ follows by the induction hypothesis.

Finally, we show the part~$(c)$ by induction.
For~$n=1$, the statement follows by definition of right adjoint action and the fact that~$\sigma$ is a coalgebra morphism.
Suppose that the statement is true for~$n\ge 1$ and let us show it for~$n+1$.
Indeed, by using inductive definition of the left and the right diagonal actions, the anti-multiplicativity of the antipode, the fact that~$\sigma$ is a coalgebra morphism, the naturality of the braiding, and the induction hypothesis, we have:
\[\,
\psfrag{r}{\hspace{-0.35cm}$H^{\tens n+1}$}
\psfrag{s}{$\sigma$}
\rsdraw{0.55}{0.75}{diagadj0}\;=\,
\psfrag{b}{$H$}
\psfrag{r}{\hspace{-0.35cm}$H^{\tens n}$}
\psfrag{s}{$\sigma$}
\rsdraw{0.55}{0.75}{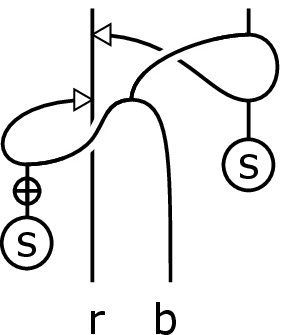}\;=\,
\psfrag{b}{$H$}
\psfrag{r}{\hspace{-0.35cm}$H^{\tens n}$}
\psfrag{s}{$\sigma$}
\rsdraw{0.55}{0.75}{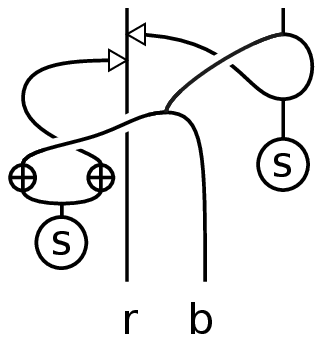}\;=\,
\psfrag{b}{$H$}
\psfrag{r}{\hspace{-0.35cm}$H^{\tens n}$}
\psfrag{s}{$\sigma$}
\rsdraw{0.55}{0.75}{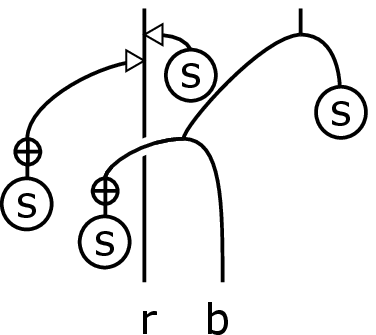}\;=\,
\psfrag{r}{\hspace{-0.35cm}$H^{\tens n+1}$}
\psfrag{s}{$\sigma$}
\rsdraw{0.55}{0.75}{diagadj1}\;\]
\end{proof}
\begin{rmk} \label{twistedantipodesigma} \normalfont If $(\delta, \sigma)$ is a modular pair, then $\,
\psfrag{b}{$H$}
\psfrag{v}[tl][tl]{$\thicksim$}
\psfrag{c}{$\sigma$}
\rsdraw{0.35}{0.75}{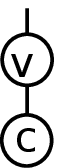}\;=
\,
\psfrag{b}{$H$}
\psfrag{c}{$\sigma$}
\rsdraw{0.35}{0.75}{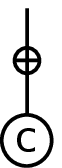}\;.$ Indeed, this follows by the definition of $\tilde{S}$, the fact that $\sigma$ is a coalgebra morphism, and since the $(\delta,\sigma)$ is a modular pair:
\[\,
\psfrag{b}{$H$}
\psfrag{v}[tl][tl]{$\thicksim$}
\psfrag{c}{$\sigma$}
\rsdraw{0.45}{0.75}{twistedantipodesigmanolab}\;=\,
\psfrag{b}{$H$}
\psfrag{c}{$\sigma$}
\psfrag{l}{\hspace{-0.05cm}$\delta$}
\rsdraw{0.45}{0.75}{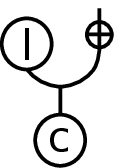}\;=\,
\psfrag{b}{$H$}
\psfrag{c}{$\sigma$}
\psfrag{l}{\hspace{-0.05cm}$\delta$}
\rsdraw{0.45}{0.75}{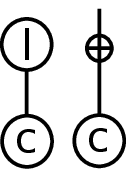}\;=\,
\psfrag{b}{$H$}
\psfrag{c}{$\sigma$}
\rsdraw{0.45}{0.75}{twistedantipodesigma1nolab}\;.\]
\end{rmk}
\subsection{Proof of Formula \eqref{kthpower}} \label{subsec: kthpower}

The proof of Formula \eqref{kthpower} of Theorem \ref{powers} is divided into several steps. For $n=k=2$, it suffices to calculate the square of $\tau_2(\delta, \sigma)$. For $n\ge 3$, we first calculate the square and then derive formulas for the remaining powers.
\subsubsection{Squares of $\tau_n(\delta,\sigma)$ for $n\ge 2$} \label{squares}
Let us first show that Formula \eqref{kthpower} is true in the case~$n=k=2$. Indeed, we have:
$$\left(\tau_2(\delta,\sigma)\right)^2\overset{(i)}{=}\,
\psfrag{b}{\hspace{-0.08cm}$H$}
\psfrag{i}{\hspace{-0.17cm}$H$}
\psfrag{v}[tl][tl]{$\thicksim$}
\psfrag{c}{$\sigma$}
\rsdraw{0.55}{0.75}{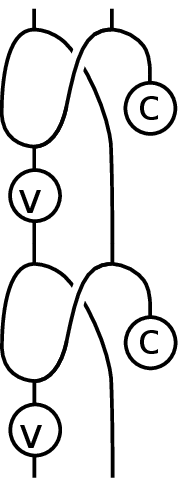}\;\overset{(ii)}{=}\,
\psfrag{b}{\hspace{-0.08cm}$H$}
\psfrag{i}{\hspace{-0.17cm}$H$}
\psfrag{v}[tl][tl]{$\thicksim$}
\psfrag{c}{$\sigma$}
\rsdraw{0.55}{0.75}{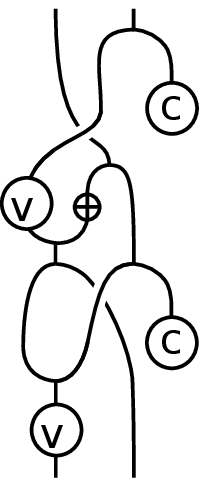}\;\overset{(iii)}{=}\,
\psfrag{b}{\hspace{-0.08cm}$H$}
\psfrag{i}{\hspace{-0.17cm}$H$}
\psfrag{v}[tl][tl]{$\thicksim$}
\psfrag{c}{$\sigma$}
\rsdraw{0.55}{0.75}{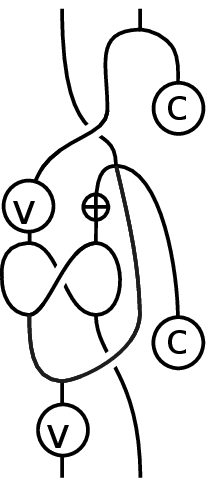}\;\overset{(iv)}{=}\,
\psfrag{b}{\hspace{-0.08cm}$H$}
\psfrag{i}{\hspace{-0.17cm}$H$}
\psfrag{v}[tl][tl]{$\thicksim$}
\psfrag{c}{$\sigma$}
\rsdraw{0.55}{0.75}{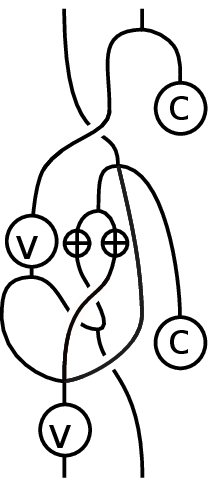}\;\overset{(v)}{=}\,
\psfrag{b}{\hspace{-0.08cm}$H$}
\psfrag{i}{\hspace{-0.17cm}$H$}
\psfrag{v}[tl][tl]{$\thicksim$}
\psfrag{c}{$\sigma$}
\rsdraw{0.55}{0.75}{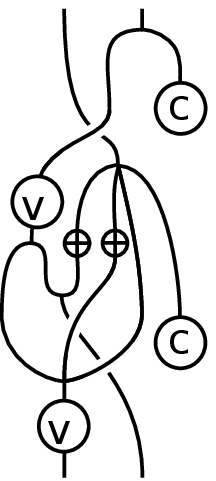}\;\overset{(vi)}{=}\,
\psfrag{b}{\hspace{-0.08cm}$H$}
\psfrag{i}{\hspace{-0.17cm}$H$}
\psfrag{v}[tl][tl]{$\thicksim$}
\psfrag{c}{$\sigma$}
\psfrag{f}{\hspace{-0.4cm}$\tau_1(\varepsilon,u)$}
\rsdraw{0.55}{0.75}{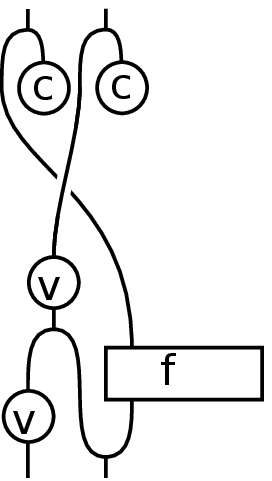}\;.
$$ Here~$(i)$ follows by definition of~$\tau_2(\delta,\sigma)$,~$(ii)$ follows from Lemma~\ref{recurrencerelations}$(a)$ for~$n=2$ and since~$\tau_1(\varepsilon,u)=S$,~$(iii)$ follows by the fact that multiplication is an algebra morphism and associativity,~$(iv)$ by the anti-multiplicativity of the antipode,~$(v)$ follows by the associativity and the naturality of the braiding and~$(vi)$ follows by the antipode axiom, the naturality of the braiding, the (co)unitality, and the fact that~$\tau_1(\varepsilon,u)=S$.

From now on, let us assume that $n\ge 3$. Let us calculate the square of $\tau_n(\delta,\sigma)$. We have
\begingroup
\allowdisplaybreaks
\begin{align*}
   \left(\tau_n(\delta,\sigma)\right)^2&\overset{(i)}{=}\,
\psfrag{b}{$1$}
\psfrag{cx}{\hspace{-0.25cm}$n-1$}
\psfrag{v}[tl][tl]{$\thicksim$}
\psfrag{r}{\hspace{-0.1cm}$2$}
\psfrag{s}{$\sigma$}
\psfrag{n}{$n$}
\psfrag{d}{\hspace{-0.05cm}$H^{\tens n-2}$}
\psfrag{h}{\hspace{-0.9cm}$H^{\tens n-3}$}
\psfrag{l}{\hspace{-0.1cm}$\delta$}
\psfrag{f}{$\tau_{n-1}(\varepsilon,u)$}
\rsdraw{0.45}{0.75}{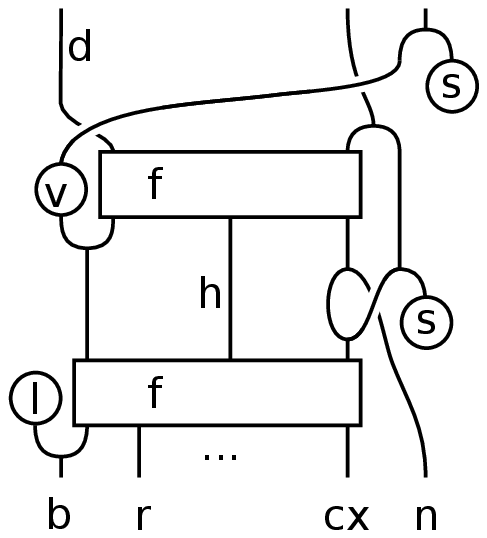}\;\overset{(ii)}{=}
\,
\psfrag{b}{$1$}
\psfrag{c}{$2$}
\psfrag{cx}{\hspace{-0.25cm}$n-1$}
\psfrag{v}[tl][tl]{$\thicksim$}
\psfrag{s}{$\sigma$}
\psfrag{n}{$n$}
\psfrag{d}{\hspace{-0.05cm}$H^{\tens n-2}$}
\psfrag{l}{\hspace{-0.1cm}$\delta$}
\psfrag{f}{$\tau_{n-1}(\varepsilon,u)$}
\rsdraw{0.45}{0.75}{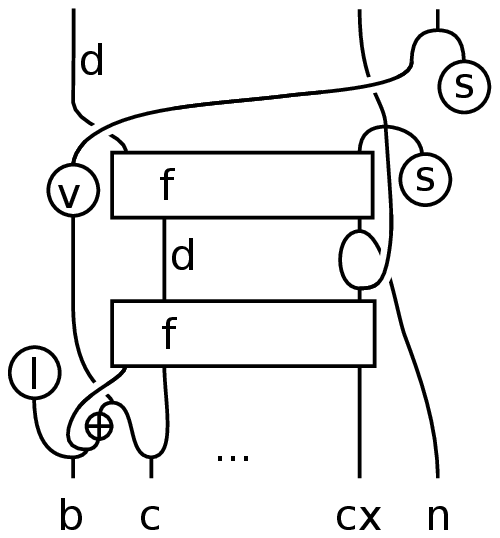}\;  \overset{(iii)}{=} \,
\psfrag{b}{$1$}
\psfrag{c}{$2$}
\psfrag{cx}{\hspace{-0.25cm}$n-1$}
\psfrag{v}[tl][tl]{$\thicksim$}
\psfrag{s}{$\sigma$}
\psfrag{n}{$n$}
\psfrag{d}{\hspace{-0.05cm}$H^{\tens n-2}$}
\psfrag{h}{\hspace{-0.05cm}$H^{\tens n-3}$}
\psfrag{l}{\hspace{-0.1cm}$\delta$}
\psfrag{f}{$\tau_{n-1}(\varepsilon,u)$}
\rsdraw{0.45}{0.75}{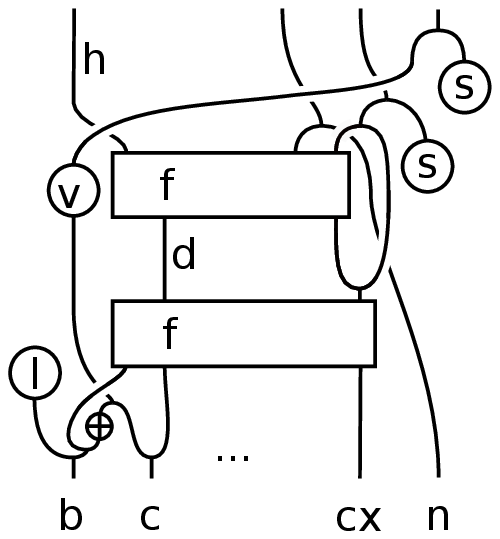}\; \overset{(iv)}{=} \\
   &\overset{(iv)}{=} \,
\psfrag{b}{$1$}
\psfrag{c}{$2$}
\psfrag{cx}{\hspace{-0.25cm}$n-1$}
\psfrag{v}[tl][tl]{$\thicksim$}
\psfrag{s}{$\sigma$}
\psfrag{n}{$n$}
\psfrag{d}{$H^{\tens n-2}$}
\psfrag{h}{\hspace{-0.05cm}$H^{\tens n-3}$}
\psfrag{l}{\hspace{-0.1cm}$\delta$}
\psfrag{f}{$\tau_{n-1}(\varepsilon,u)$}
\psfrag{g}{$\tau_{n-2}(\varepsilon,u)$}
\rsdraw{0.45}{0.75}{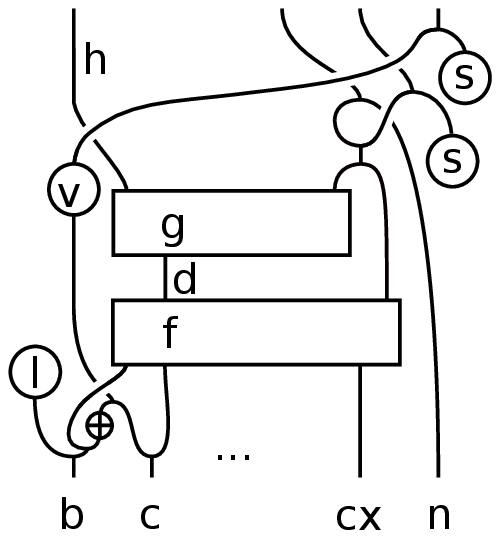}\; \overset{(v)}{=} \,
\psfrag{b}{$1$}
\psfrag{c}{$2$}
\psfrag{cx}{\hspace{-0.25cm}$n-1$}
\psfrag{v}[tl][tl]{$\thicksim$}
\psfrag{s}{$\sigma$}
\psfrag{n}{$n$}
\psfrag{h}{\hspace{-0.05cm}$H^{\tens n-3}$}
\psfrag{l}{\hspace{-0.1cm}$\delta$}
\psfrag{f}{$\tau_{n-2}(\varepsilon,u)$}
\rsdraw{0.45}{0.75}{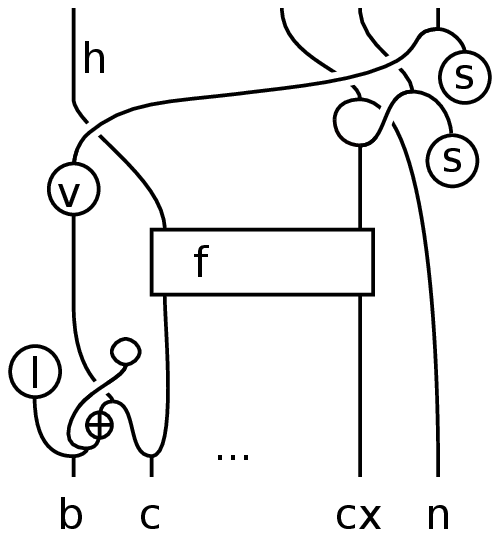}\; \overset{(vi)}{=} \,
\psfrag{b}{$1$}
\psfrag{c}{$2$}
\psfrag{cx}{\hspace{-0.25cm}$n-1$}
\psfrag{v}[tl][tl]{$\thicksim$}
\psfrag{s}{$\sigma$}
\psfrag{n}{$n$}
\psfrag{k}{\hspace{-0.2cm}$H^{\tens n-2}$}
\psfrag{l}{\hspace{-0.1cm}$\delta$}
\psfrag{f}{$\tau_{n-1}(\varepsilon,u)$}
\rsdraw{0.45}{0.75}{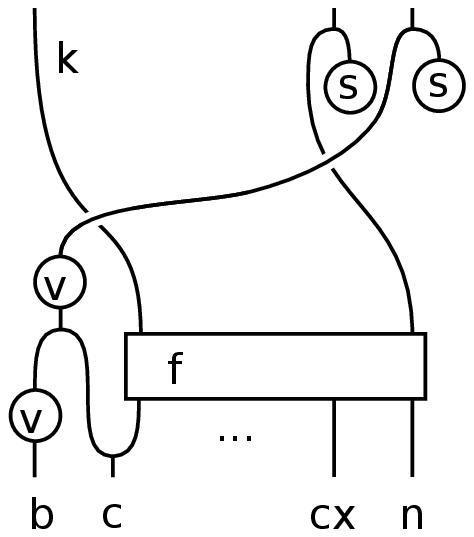}\;,
\end{align*}
\endgroup
which indeed shows Formula~\eqref{kthpower} in the case~$n\ge 3$ and $k=2$.
Here~$(i)$ follows by parts~$(a)$ and~$(b)$  of Lemma~\ref{recurrencerelations},~$(ii)$ follows by the associativity and from Lemma~\ref{squarelemma}$(a)$ for~$n-1$,~$(iii)$ follows by the associativity and from Lemma~\ref{squarelemma}$(b)$ for~$n-1$,~$(iv)$ follows from Lemma~\ref{squarelemma}$(c)$ for~$n-1$,~$(v)$ follows from Lemma~\ref{squarelemma}$(d)$ for~$n-1$ and~$(vi)$ follows by the co\-unitality, de\-fi\-ni\-tion of the twisted an\-ti\-pode~$\tilde{S}$, and from Lemma~\ref{recurrencerelations}$(b)$ applied on~$\delta=\varepsilon$ and~$\sigma=u$ for~$n-1$.

\subsubsection{Passing from $\left(\tau_n(\delta,\sigma)\right)^2$ to $\left(\tau_n(\delta,\sigma)\right)^3$, $n\ge 3$} \label{subsec: cube} From the calculation that has been done in Section \ref{squares}, we can easily deduce Formula \eqref{kthpower} for $n\ge 3$ and $k=3$.
\begingroup
\allowdisplaybreaks
\begin{align*}
  \left(\tau_n(\delta,\sigma)\right)^3 &\overset{(i)}{=} \,
\psfrag{b}{\hspace{-0.1cm}$1$}
\psfrag{c}{$2$}
\psfrag{cx}{\hspace{0.1cm}$3$}
\psfrag{v}[tl][tl]{$\thicksim$}
\psfrag{s}{$\sigma$}
\psfrag{n}{$n$}
\psfrag{k}{\hspace{-0.2cm}$H^{\tens n-2}$}
\psfrag{l}{\hspace{-0.1cm}$\delta$}
\psfrag{f}{$\tau_{n-1}(\varepsilon,u)$}
\rsdraw{0.45}{0.75}{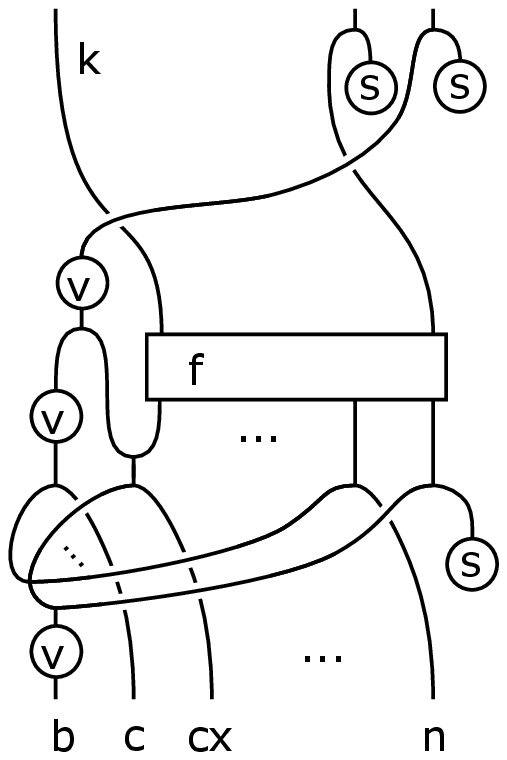}\;\overset{(ii)}{=} \,
\psfrag{b}{\hspace{-0.1cm}$1$}
\psfrag{c}{$2$}
\psfrag{cx}{\hspace{0.1cm}$3$}
\psfrag{v}[tl][tl]{$\thicksim$}
\psfrag{s}{$\sigma$}
\psfrag{n}{$n$}
\psfrag{k}{\hspace{-0.05cm}$H^{\tens n-2}$}
\psfrag{l}{\hspace{-0.1cm}$\delta$}
\psfrag{f}{$\tau_{n-1}(\varepsilon,u)$}
\rsdraw{0.45}{0.75}{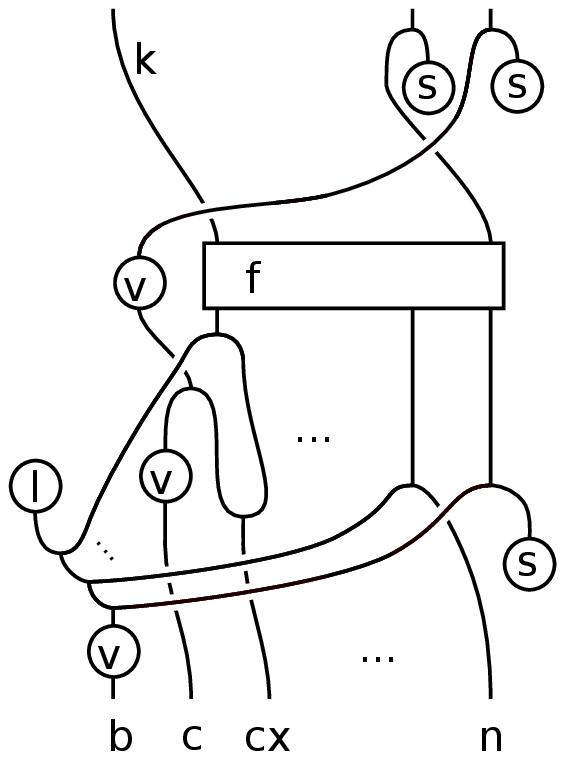}\;\overset{(iii)}{=}\,
\psfrag{b}{\hspace{-0.1cm}$1$}
\psfrag{c}{$2$}
\psfrag{cx}{\hspace{0.1cm}$3$}
\psfrag{v}[tl][tl]{$\thicksim$}
\psfrag{s}{$\sigma$}
\psfrag{n}{$n$}
\psfrag{k}{\hspace{-0.05cm}$H^{\tens n-2}$}
\psfrag{l}{\hspace{-0.1cm}$\delta$}
\psfrag{f}{$\tau_{n-1}(\varepsilon,u)$}
\rsdraw{0.45}{0.75}{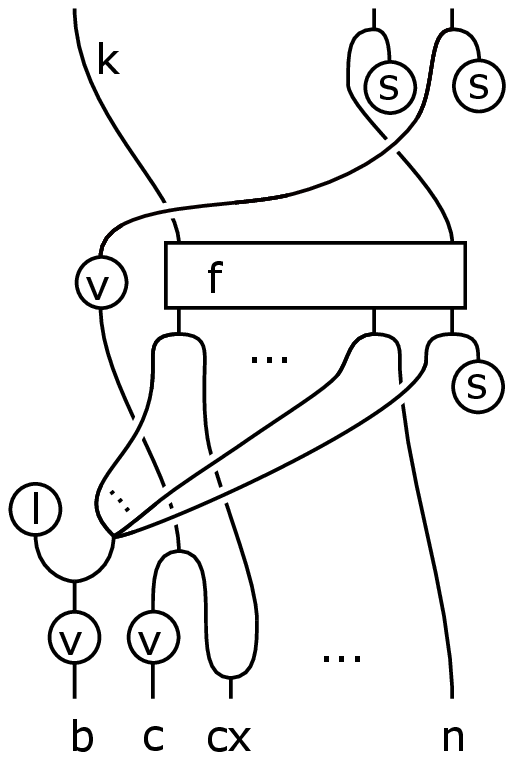}\;\overset{(iv)}{=}\\
   &\overset{(iv)}{=}\,
\psfrag{b}{\hspace{-0.1cm}$1$}
\psfrag{c}{$2$}
\psfrag{cx}{\hspace{0.1cm}$3$}
\psfrag{v}[tl][tl]{$\thicksim$}
\psfrag{s}{$\sigma$}
\psfrag{n}{$n$}
\psfrag{k}{\hspace{-0.05cm}$H^{\tens n-3}$}
\psfrag{l}{\hspace{-0.1cm}$\delta$}
\psfrag{f}{$\tau_{n-1}(\varepsilon,u)$}
\rsdraw{0.45}{0.75}{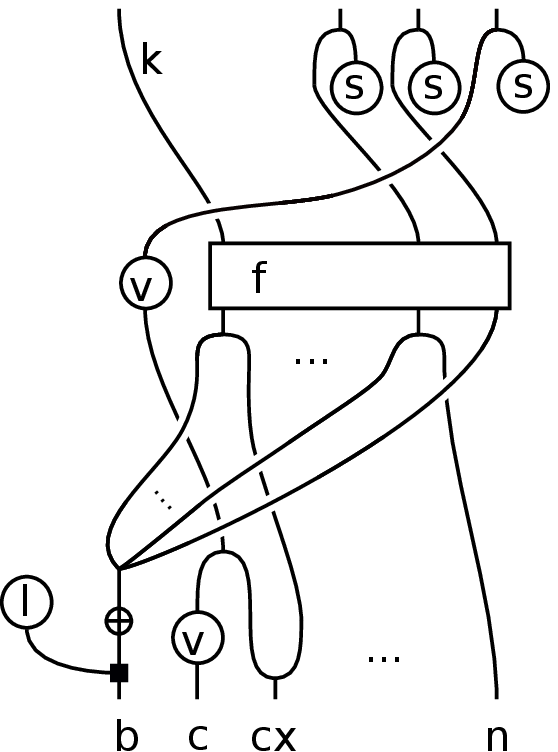}\;\overset{(v)}{=}\,
\psfrag{b}{\hspace{-0.1cm}$1$}
\psfrag{c}{$2$}
\psfrag{cx}{\hspace{0.1cm}$3$}
\psfrag{v}[tl][tl]{$\thicksim$}
\psfrag{s}{$\sigma$}
\psfrag{n}{$n$}
\psfrag{k}{\hspace{-0.05cm}$H^{\tens n-3}$}
\psfrag{l}{\hspace{-0.1cm}$\delta$}
\psfrag{f}{$\left(\tau_{n-1}(\varepsilon,u)\right)^2$}
\psfrag{d}{\hspace{-0.15cm}$H^{\tens 2}$}
\rsdraw{0.45}{0.75}{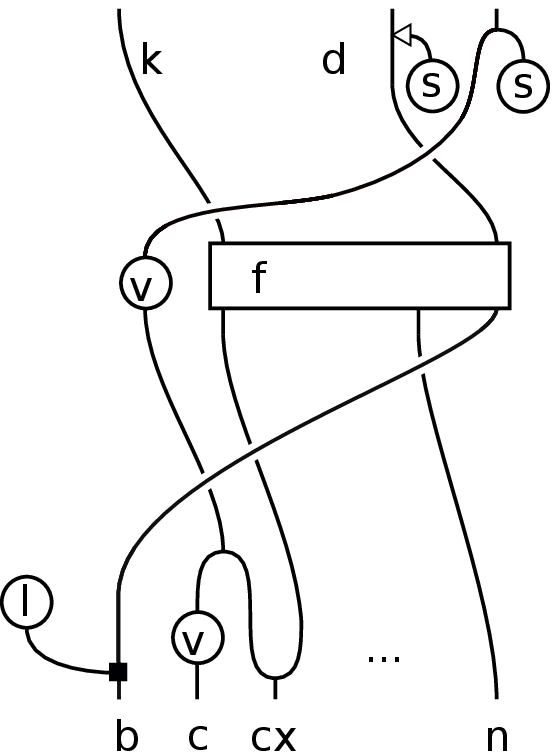}\;.
\end{align*}
\endgroup
Here~$(i)$ follows by developing~$\tau_n(\delta,\sigma)^3$ as~$\tau_n(\delta,\sigma)^2\tau_n(\delta,\sigma)$, by definition of~$\tau_n(\delta,\sigma)$ and by the computation of~$\tau_n(\delta,\sigma)^2$ that has been done in~\ref{squares},~$(ii)$ follows by the coassociativity and Lemma~\ref{algebraicppties}$(f)$,~$(iii)$ follows by the naturality of the braiding and the coassociativity,~$(iv)$ follows from Lemma~\ref{flow}$(b)$ for~$j=n-1$ and from Lemma~\ref{algebraicppties}$(d)$,~$(v)$ follows by the naturality of the braiding and the definition of~$\tau_{n-1}(\varepsilon,u).$
\subsubsection{Computation of $\left(\tau_n(\delta,\sigma)\right)^{j+1}$, $2\le j \le n-1$}
Note that by now, we have completely shown Formula~\eqref{kthpower} in the cases~$n=2$ and~$n=3$. Also, the square and the cube of~$\tau_n(\delta,\sigma)$ are calculated for each~$n\ge 3$.
In this section, we finish the proof of~\eqref{kthpower}, by focusing on the case~$n\ge 4$.
As it has been already noted, for~$j=2$,~$\left(\tau_n(\delta,\sigma)\right)^{j+1}$ is already computed in Section~\ref{subsec: cube}.
From now on, we assume that~$3\le j \le n-1$.
If~\eqref{kthpower} is established for~$k=j$, then:
\begingroup
\allowdisplaybreaks
\begin{align*}
(\tau_n(\delta, \sigma))^{j+1} & \overset{(i)}{=} \,
\psfrag{b}{\hspace{-0.1cm}$1$}
\psfrag{c}{$2$}
\psfrag{cx}{\hspace{-0.25cm}$j-1$}
\psfrag{e}{$j$}
\psfrag{e+1}{$j+1$}
\psfrag{v}[tl][tl]{$\thicksim$}
\psfrag{s}{$\sigma$}
\psfrag{n}{$n$}
\psfrag{k}{\hspace{-0.05cm}$H^{\tens n-j}$}
\psfrag{l}{\hspace{-0.1cm}$\delta$}
\psfrag{f}{$\left(\tau_{n-1}(\varepsilon,u)\right)^{j-1}$}
\psfrag{d}{\hspace{-0.7cm}$H^{\tens j-1}$}
\rsdraw{0.45}{0.75}{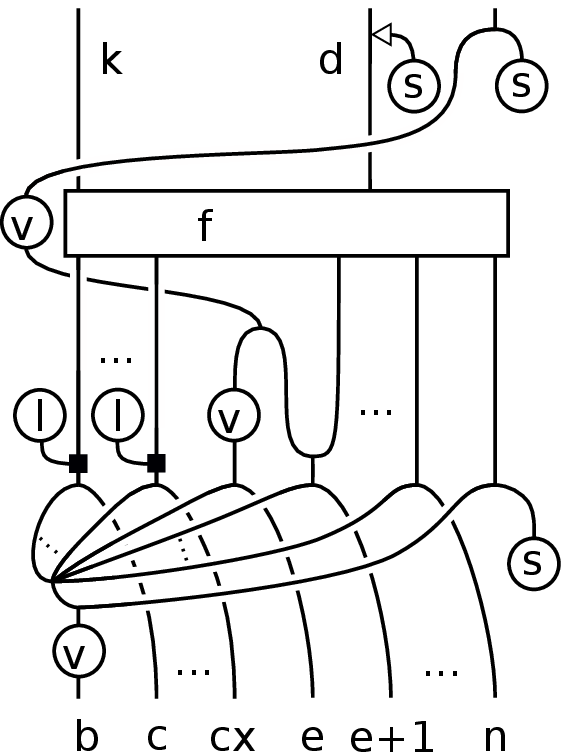}\;\overset{(ii)}{=}\,
\psfrag{b}{\hspace{-0.1cm}$1$}
\psfrag{c}{$2$}
\psfrag{cx}{\hspace{-0.25cm}$j-1$}
\psfrag{e}{$j$}
\psfrag{e+1}{$j+1$}
\psfrag{v}[tl][tl]{$\thicksim$}
\psfrag{s}{$\sigma$}
\psfrag{n}{$n$}
\psfrag{k}{\hspace{-0.05cm}$H^{\tens n-j}$}
\psfrag{l}{\hspace{-0.1cm}$\delta$}
\psfrag{f}{$\left(\tau_{n-1}(\varepsilon,u)\right)^{j-1}$}
\psfrag{d}{\hspace{-0.7cm}$H^{\tens j-1}$}
\rsdraw{0.45}{0.75}{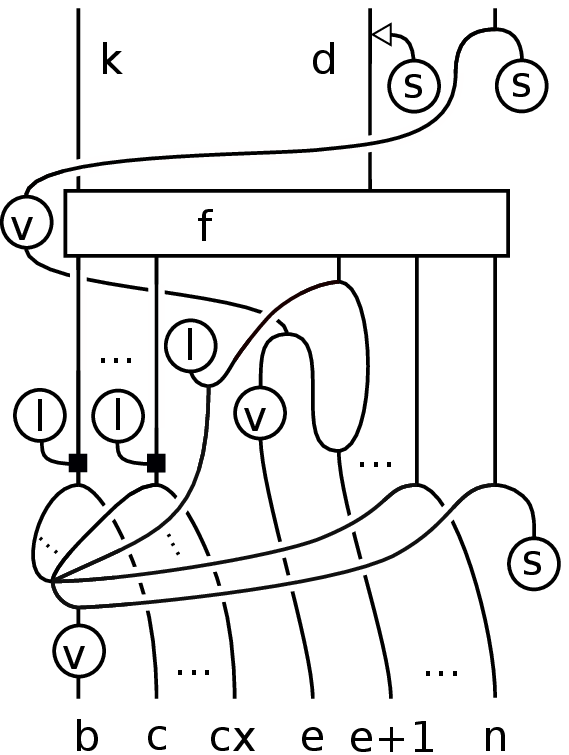}\;\overset{(iii)}{=} \\&\overset{(iii)}{=}\,
\psfrag{b}{\hspace{-0.1cm}$1$}
\psfrag{c}{$2$}
\psfrag{cx}{\hspace{-0.25cm}$j-1$}
\psfrag{e}{$j$}
\psfrag{e+1}{$j+1$}
\psfrag{v}[tl][tl]{$\thicksim$}
\psfrag{s}{$\sigma$}
\psfrag{n}{$n$}
\psfrag{k}{\hspace{-0.05cm}$H^{\tens n-j}$}
\psfrag{l}{\hspace{-0.1cm}$\delta$}
\psfrag{f}{$\left(\tau_{n-1}(\varepsilon,u)\right)^{j-1}$}
\psfrag{d}{\hspace{-0.7cm}$H^{\tens j-1}$}
\rsdraw{0.45}{0.75}{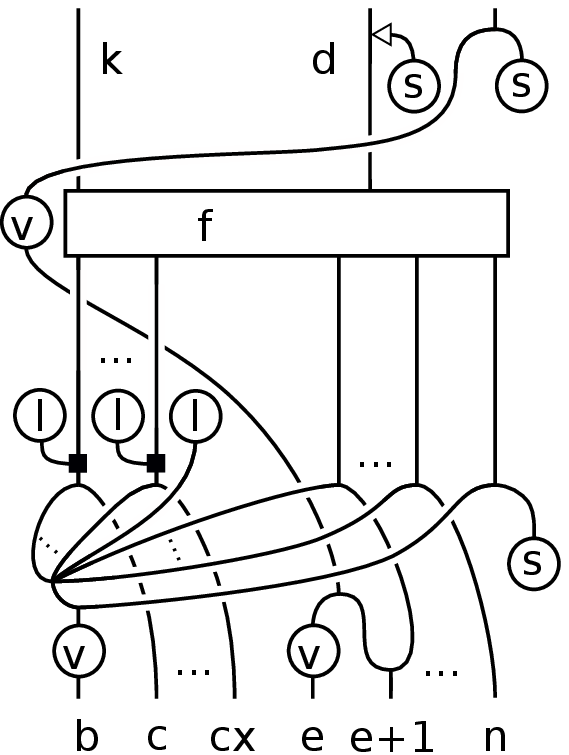}\;\overset{(iv)}{=}\,
\psfrag{b}{\hspace{-0.1cm}$1$}
\psfrag{c}{$2$}
\psfrag{cx}{\hspace{-0.25cm}$j-1$}
\psfrag{e}{$j$}
\psfrag{e+1}{$j+1$}
\psfrag{v}[tl][tl]{$\thicksim$}
\psfrag{s}{$\sigma$}
\psfrag{n}{$n$}
\psfrag{k}{\hspace{-0.05cm}$H^{\tens n-j-1}$}
\psfrag{l}{\hspace{-0.1cm}$\delta$}
\psfrag{f}{$\left(\tau_{n-1}(\varepsilon,u)\right)^{j-1}$}
\psfrag{d}{$H^{\tens j-1}$}
\rsdraw{0.45}{0.75}{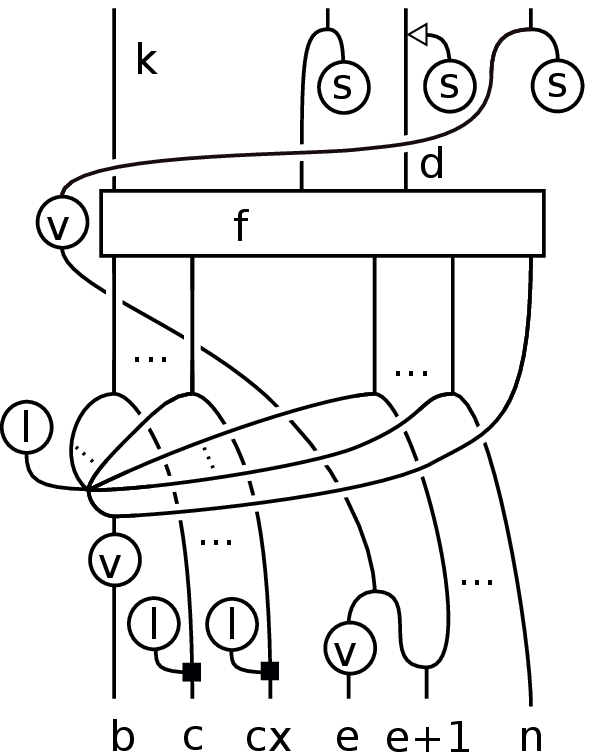}\;\overset{(v)}{=} \\
&\overset{(v)}{=}\,
\psfrag{b}{\hspace{-0.1cm}$1$}
\psfrag{c}{$2$}
\psfrag{cx}{\hspace{-0.25cm}$j-1$}
\psfrag{e}{$j$}
\psfrag{e+1}{$j+1$}
\psfrag{v}[tl][tl]{$\thicksim$}
\psfrag{s}{$\sigma$}
\psfrag{n}{$n$}
\psfrag{k}{\hspace{-0.05cm}$H^{\tens n-j-1}$}
\psfrag{l}{\hspace{-0.1cm}$\delta$}
\psfrag{f}{$\left(\tau_{n-1}(\varepsilon,u)\right)^{j-1}$}
\psfrag{d}{\hspace{-0.4cm}$H^{\tens j}$}
\rsdraw{0.45}{0.75}{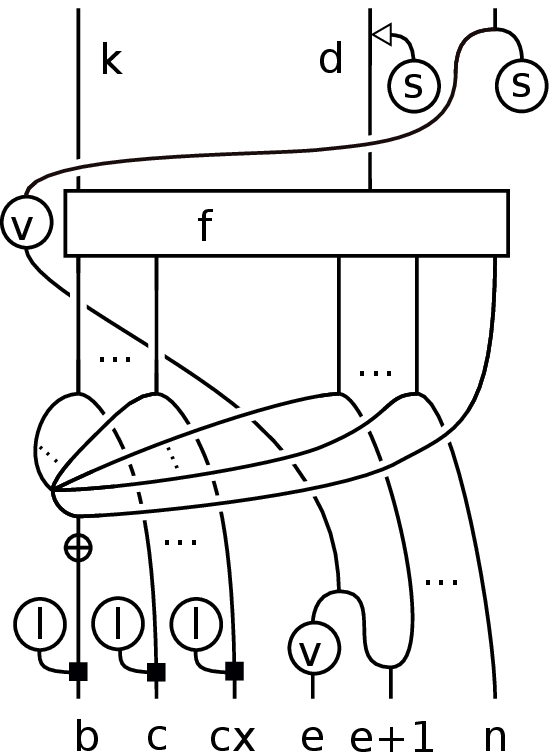}\;\overset{(vi)}{=}
\,
\psfrag{b}{\hspace{-0.1cm}$1$}
\psfrag{c}{$2$}
\psfrag{cx}{\hspace{-0.25cm}$j-1$}
\psfrag{e}{$j$}
\psfrag{e+1}{$j+1$}
\psfrag{v}[tl][tl]{$\thicksim$}
\psfrag{s}{$\sigma$}
\psfrag{n}{$n$}
\psfrag{k}{\hspace{-0.05cm}$H^{\tens n-j-1}$}
\psfrag{l}{\hspace{-0.1cm}$\delta$}
\psfrag{f}{$\left(\tau_{n-1}(\varepsilon,u)\right)^{j}$}
\psfrag{d}{\hspace{-0.4cm}$H^{\tens j}$}
\rsdraw{0.45}{0.75}{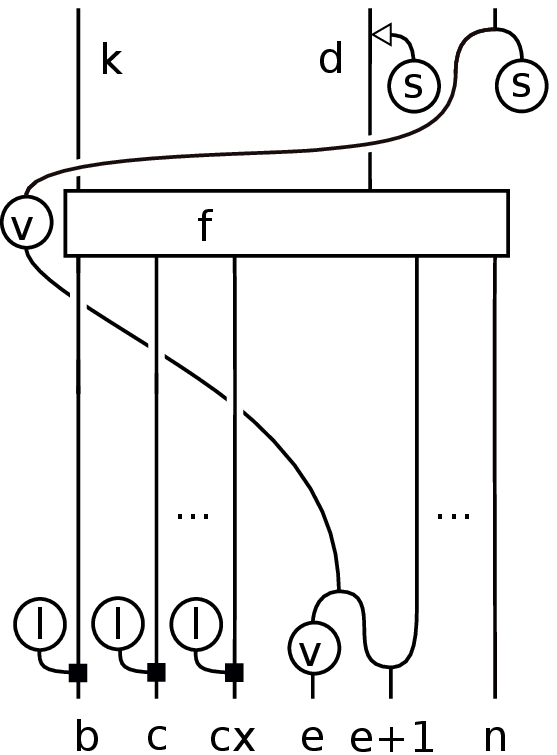}\;,
\end{align*}
\endgroup
which finishes the proof of~\eqref{kthpower}.
Here~$(i)$ follows by decomposing~$(\tau_n(\delta,\sigma))^{j+1}$ in the composition~$(\tau_n(\delta,\sigma))^j\tau_n(\delta,\sigma)$, by definition of~$\tau_n(\delta,\sigma)$ and by the hypothesis that~\eqref{kthpower} is established for~$k=j$,~$(ii)$ follows by the coassociativity and Lemma~\ref{algebraicppties}$(f)$,~$(iii)$ follows by the naturality of the braiding and the coassociativity,~$(iv)$ follows from~Lemma~\ref{auxilia}$(b)$ and by applying $j-1$ times Lemma \ref{flow}$(b)$,~$(v)$ follows from Lemma~\ref{algebraicppties}$(d)$ and the coassociativity and finally,~$(vi)$ follows by the naturality of the braiding and the definition of~$\tau_{n-1}(\varepsilon,u).$

\subsection{Proof of Formula \eqref{n+1thpower}} \label{finalstep}

In order to show Formula~\eqref{n+1thpower} of Theorem \ref{powers}, we will separately consider three cases:~$n=1$,~$n=2$, and~$n\ge 3$.

If~$n=1$, then we have
\[\left(\tau_1(\delta,\sigma)\right)^2\overset{(i)}{=}\,
\psfrag{b}{\hspace{-0.2cm}$H$}
\psfrag{c}{$\sigma$}
\psfrag{v}[tl][tl]{$\thicksim$}
\rsdraw{0.45}{0.75}{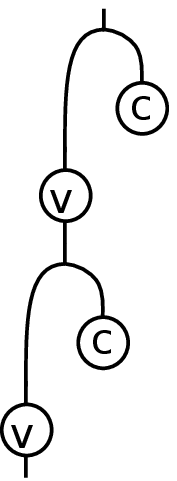}\;\overset{(ii)}{=}\,
\psfrag{b}{\hspace{-0.2cm}$H$}
\psfrag{c}{$\sigma$}
\psfrag{v}[tl][tl]{$\thicksim$}
\rsdraw{0.45}{0.75}{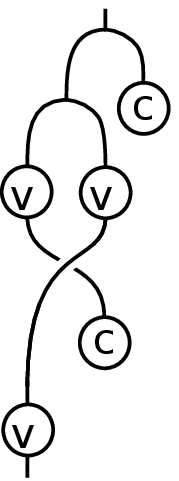}\;\overset{(iii)}{=}\,
\psfrag{b}{\hspace{-0.2cm}$H$}
\psfrag{c}{$\sigma$}
\psfrag{v}[tl][tl]{$\thicksim$}
\rsdraw{0.45}{0.75}{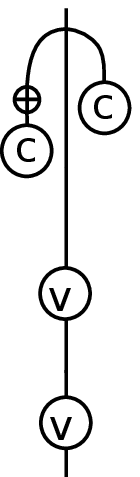}\;\overset{(iv)}{=}\,
\psfrag{b}{\hspace{-0.2cm}$H$}
\psfrag{v}{$\sigma$}
\psfrag{l}{\hspace{-0.1cm}$\delta$}
\rsdraw{0.45}{0.75}{tau123DSnolab}\;.\]
Here~$(i)$ fol\-lows by de\-fi\-ni\-tion~$\tau_1(\delta,\sigma)$,~$(ii)$ fol\-lows from Le\-mma~\ref{algebraicppties}$(b)$,~$(iii)$ follows by the na\-tu\-ra\-li\-ty of the brai\-ding, the as\-so\-ci\-a\-ti\-vi\-ty, and Remark~\ref{twistedantipodesigma},~$(iv)$ follows by definition of right adjoint action, the fact that~$\sigma$ is a coalgebra morphism, and Lemma~\ref{algebraicppties}$(e)$.

If~$n=2$, then we have
\[\left(\tau_2(\delta,\sigma)\right)^3\overset{(i)}{=}\,
\psfrag{b}{\hspace{-0.2cm}$H$}
\psfrag{i}{\hspace{-0.2cm}$H$}
\psfrag{c}{$\sigma$}
\psfrag{v}[tl][tl]{$\thicksim$}
\rsdraw{0.55}{0.75}{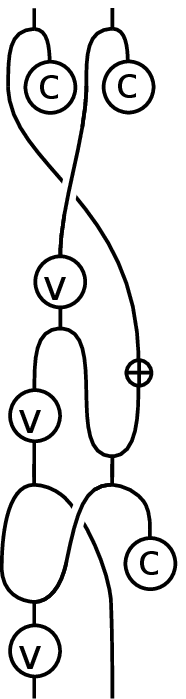}\;\overset{(ii)}{=}\,
\psfrag{b}{\hspace{-0.2cm}$H$}
\psfrag{i}{\hspace{-0.2cm}$H$}
\psfrag{c}{$\sigma$}
\psfrag{v}[tl][tl]{$\thicksim$}
\psfrag{l}{\hspace{-0.05cm}$\delta$}
\rsdraw{0.55}{0.75}{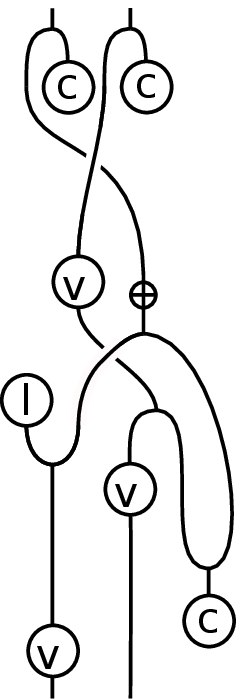}\;\overset{(iii)}{=}\,
\psfrag{b}{\hspace{-0.2cm}$H$}
\psfrag{i}{\hspace{-0.2cm}$H$}
\psfrag{c}{$\sigma$}
\psfrag{v}[tl][tl]{$\thicksim$}
\psfrag{l}{\hspace{-0.05cm}$\delta$}
\rsdraw{0.55}{0.75}{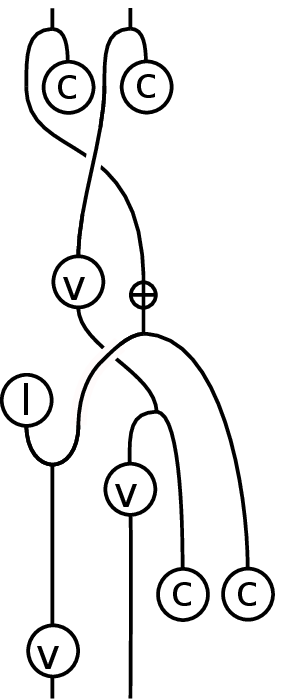}\;\hspace{-0.4cm}\overset{(iv)}{=}\,
\psfrag{b}{\hspace{-0.2cm}$H$}
\psfrag{i}{\hspace{-0.2cm}$H$}
\psfrag{g}{\hspace{-1.1cm}$\left(\tau_1(\delta, \sigma)\right)^2$}
\psfrag{c}{$\sigma$}
\psfrag{v}[tl][tl]{$\thicksim$}
\psfrag{l}{\hspace{-0.1cm}$\delta$}
\rsdraw{0.55}{0.75}{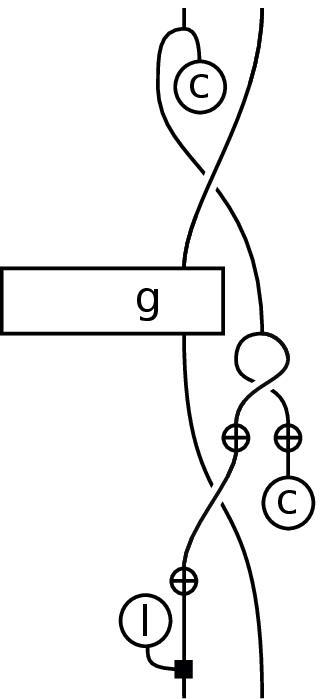}\;\overset{(v)}{=}
\,
\psfrag{b}{\hspace{-0.2cm}$H$}
\psfrag{i}{\hspace{-0.2cm}$H$}
\psfrag{g}{\hspace{-1.1cm}$\left(\tau_1(\delta, \sigma)\right)^2$}
\psfrag{c}{$\sigma$}
\psfrag{v}[tl][tl]{$\thicksim$}
\psfrag{l}{$\delta$}
\rsdraw{0.55}{0.75}{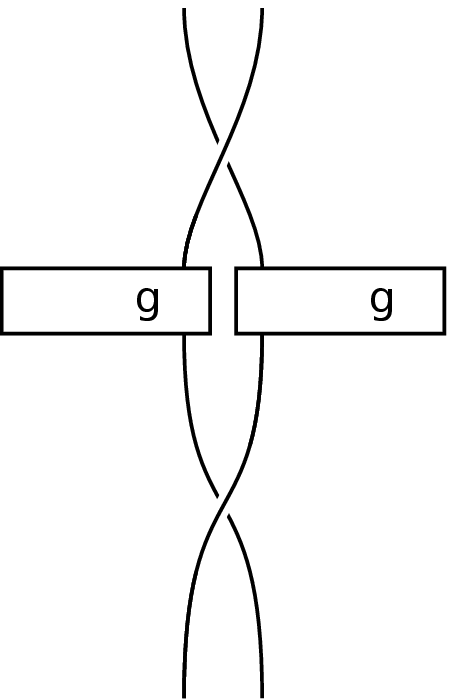}\;.
\]
Here~$(i)$ follows by expanding~$(\tau_2(\delta,\sigma))^3$ as~$(\tau_2(\delta,\sigma))^2\tau_2(\delta,\sigma)$, using the definition of~$\tau_2(\delta,\sigma)$, and Formula~\eqref{kthpower} in the case~$n=k=2$,~$(ii)$ follows from Lemma~\ref{algebraicppties}$(f)$,~$(iii)$ by using the fact that~$\sigma$ is a coalgebra morphism,~$(iv)$ by the naturality of the braiding, definition of~$\tau_1(\delta,\sigma)$, and the anti-multiplicativity of the antipode,~$(v)$ follows by the naturality of the braiding, by definition of right adjoint action, and the case~$n=1$ of Formula~$\eqref{n+1thpower}$, which is shown above.

From now on, let us assume that $n\ge 3$. We have
\begingroup
\allowdisplaybreaks
\begin{align*}\left(\tau_n(\delta,\sigma)\right)^{n+1}\overset{(i)}{=}\,
\psfrag{b}{$1$}
\psfrag{r}{\hspace{-0.05cm}$2$}
\psfrag{f}{\hspace{-1cm}$\left(\tau_{n-1}(\varepsilon,u)\right)^{n-1}$}
\psfrag{g}{\hspace{-0.35cm}$\tau_{n-1}(\varepsilon, u)$}
\psfrag{cx}{$n-1$}
\psfrag{n}{$n$}
\psfrag{s}{$\sigma$}
\psfrag{v}[tl][tl]{$\thicksim$}
\psfrag{l}{$\delta$}
\psfrag{d}{\hspace{-0.8cm}$H^{\tens n-1}$}
\rsdraw{0.55}{0.75}{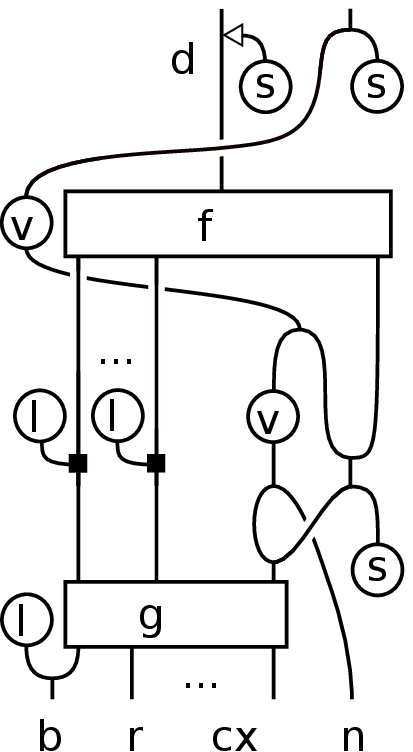}\;&\overset{(ii)}{=} \,
\psfrag{b}{$1$}
\psfrag{f}{\hspace{-1cm}$\left(\tau_{n-1}(\varepsilon,u)\right)^{n-1}$}
\psfrag{g}{\hspace{-0.35cm}$\tau_{n-1}(\varepsilon, u)$}
\psfrag{r}{\hspace{-0.05cm}$2$}
\psfrag{cx}{$n-1$}
\psfrag{n}{$n$}
\psfrag{s}{$\sigma$}
\psfrag{v}[tl][tl]{$\thicksim$}
\psfrag{l}{$\delta$}
\psfrag{d}{\hspace{-0.8cm}$H^{\tens n-1}$}
\rsdraw{0.55}{0.75}{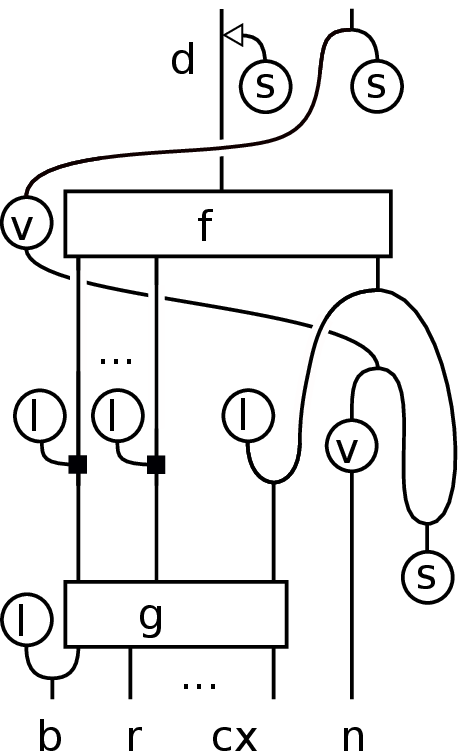}\; \overset{(iii)}{=} \,
\psfrag{b}{$1$}
\psfrag{f}{\hspace{-1cm}$\left(\tau_{n-1}(\varepsilon,u)\right)^{n-1}$}
\psfrag{g}{\hspace{-0.35cm}$\tau_{n-1}(\varepsilon, u)$}
\psfrag{cx}{$n-1$}
\psfrag{r}{\hspace{-0.05cm}$2$}
\psfrag{n}{$n$}
\psfrag{s}{$\sigma$}
\psfrag{v}[tl][tl]{$\thicksim$}
\psfrag{l}{$\delta$}
\psfrag{d}{\hspace{-0.8cm}$H^{\tens n-1}$}
\rsdraw{0.55}{0.75}{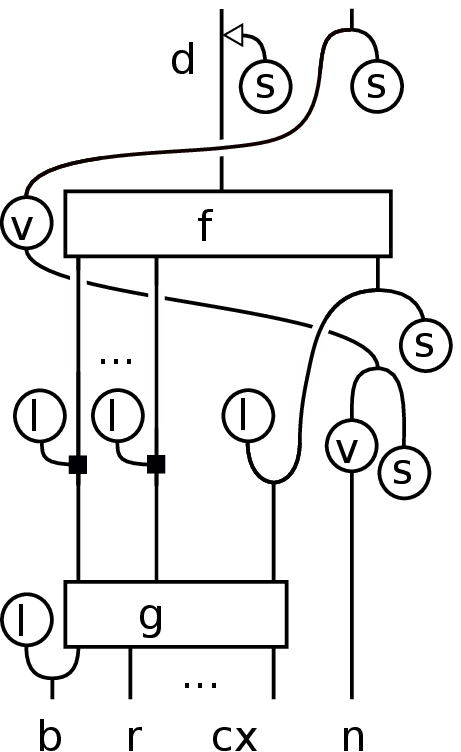}\;\overset{(iv)}{=} &\\ \overset{(iv)}{=}
 \,
\psfrag{b}{$1$}
\psfrag{i}{\hspace{-0.05cm}$2$}
\psfrag{f}{\hspace{-1.35cm}$\left(\tau_{n-1}(\varepsilon,u)\right)^{n-1}$}
\psfrag{g}{\hspace{-1.1cm}$\left(\tau_{1}(\delta, \sigma)\right)^2$}
\psfrag{cx}{\hspace{-0.05cm}$n-1$}
\psfrag{n}{\hspace{-0.05cm}$n$}
\psfrag{s}{$\sigma$}
\psfrag{v}[tl][tl]{$\thicksim$}
\psfrag{l}{\hspace{-0.1cm}$\delta$}
\psfrag{d}{\hspace{-0.8cm}$H^{\tens n-1}$}
\rsdraw{0.55}{0.75}{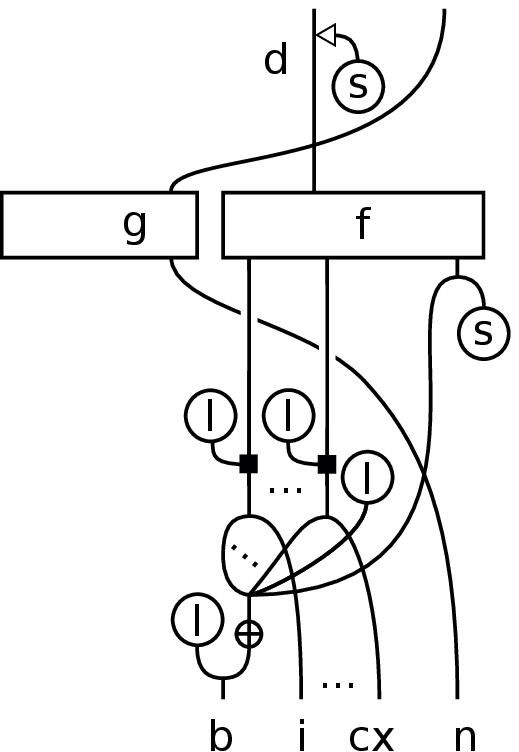}\; &\overset{(v)}{=} \,
\psfrag{b}{$1$}
\psfrag{i}{\hspace{-0.05cm}$2$}
\psfrag{f}{\hspace{-1.35cm}$\left(\tau_{n-1}(\varepsilon,u)\right)^{n-2}$}
\psfrag{h}{\hspace{-0.8cm}$\tau_{n-1}(\varepsilon,u)$}
\psfrag{g}{\hspace{-1.1cm}$\left(\tau_{1}(\delta, \sigma)\right)^2$}
\psfrag{k}{\hspace{-0.2cm}$H^{\tens n-2}$}
\psfrag{cx}{\hspace{-0.05cm}$n-1$}
\psfrag{n}{\hspace{-0.05cm}$n$}
\psfrag{s}{$\sigma$}
\psfrag{v}[tl][tl]{$\thicksim$}
\psfrag{l}{\hspace{-0.1cm}$\delta$}
\psfrag{d}{\hspace{-0.8cm}$H^{\tens n-1}$}
\rsdraw{0.55}{0.75}{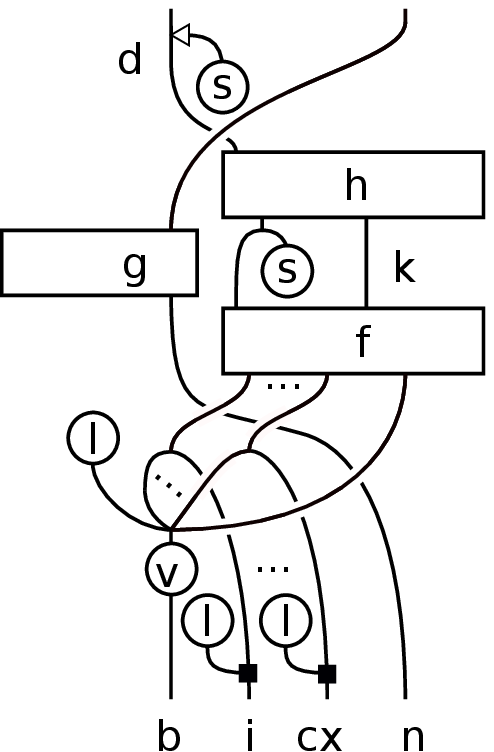}\; \overset{(vi)}{=} \,
\psfrag{b}{$1$}
\psfrag{i}{\hspace{-0.05cm}$2$}
\psfrag{f}{\hspace{-1.35cm}$\left(\tau_{n-1}(\varepsilon,u)\right)^{n-2}$}
\psfrag{h}{\hspace{-0.8cm}$\tau_{n-1}(\varepsilon,u)$}
\psfrag{g}{\hspace{-1.1cm}$\left(\tau_{1}(\delta, \sigma)\right)^2$}
\psfrag{k}{\hspace{-0.2cm}$H^{\tens n-1}$}
\psfrag{cx}{\hspace{-0.05cm}$n-1$}
\psfrag{n}{\hspace{-0.05cm}$n$}
\psfrag{s}{$\sigma$}
\psfrag{v}[tl][tl]{$\thicksim$}
\psfrag{l}{\hspace{-0.1cm}$\delta$}
\psfrag{d}{\hspace{-0.8cm}$H^{\tens n-1}$}
\rsdraw{0.55}{0.75}{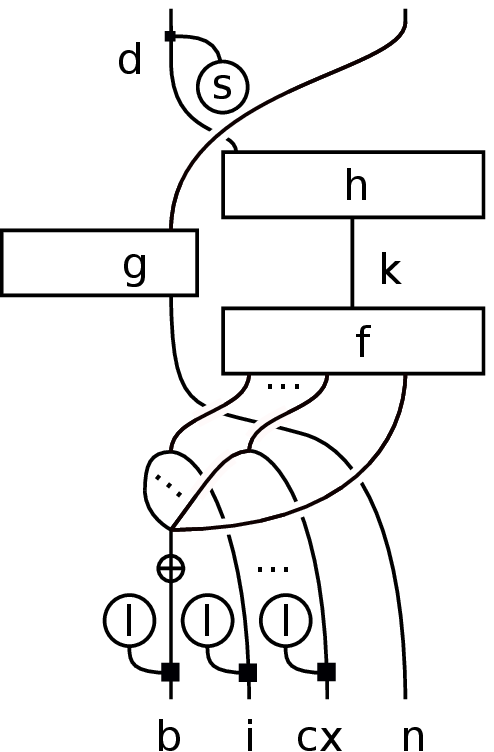}\;\overset{(vii)}{=}  &\\  &\overset{(vii)}{=} \,
\psfrag{b}{$1$}
\psfrag{i}{\hspace{-0.05cm}$2$}
\psfrag{f}{\hspace{-1.15cm}$\left(\tau_{n-1}(\varepsilon,u)\right)^{n}$}
\psfrag{h}{\hspace{-0.8cm}$\tau_{n-1}(\varepsilon,u)$}
\psfrag{g}{\hspace{-1.1cm}$\left(\tau_{1}(\delta, \sigma)\right)^2$}
\psfrag{k}{\hspace{-0.2cm}$H^{\tens n-2}$}
\psfrag{cx}{\hspace{-0.05cm}$n-1$}
\psfrag{n}{\hspace{-0.05cm}$n$}
\psfrag{s}{$\sigma$}
\psfrag{v}[tl][tl]{$\thicksim$}
\psfrag{l}{\hspace{-0.1cm}$\delta$}
\psfrag{d}{\hspace{-0.8cm}$H^{\tens n-1}$}
\rsdraw{0.55}{0.75}{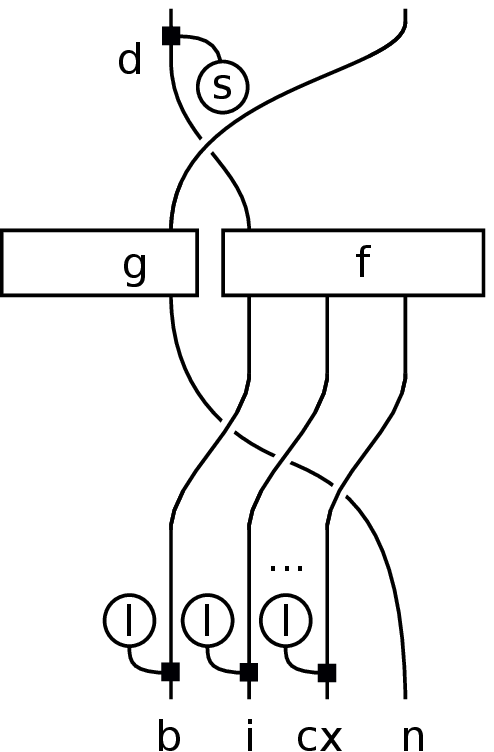}\;,&
\end{align*}\endgroup
which shows Formula~\eqref{n+1thpower}.
Here~$(i)$ follows by decomposing~$(\tau_n(\delta,\sigma))^{n+1}$ in the composition~$(\tau_n(\delta,\sigma))^n\tau_n(\delta,\sigma)$, by using Lemma~\ref{recurrencerelations}$(b)$, and Formula~\eqref{kthpower} for~$k=n$,~$(ii)$ follows from Lemma~\ref{algebraicppties}$(f)$,~$(iii)$ by the fact that~$\sigma$ is a coalgebra morphism,~$(iv)$ by defi\-nition of~$\tau_{n-1}(\varepsilon,u)$, the naturality of the braiding, definition of~$\tau_1(\delta,\sigma)$, and the coassociati\-vi\-ty,~$(v)$ by definition of the twisted antipode, Lemma~\ref{auxilia}$(b)$, and by applying~$n-2$ times Lemma~\ref{flow}$(b)$,~$(vi)$ follows by the naturality of the braiding, by combining Lemma~\ref{flow}$(a)$ with Lemma~\ref{auxilia}$(c)$ for~$n-1$, by the coassociativity, and Lemma~\ref{algebraicppties}$(d)$,~$(vii)$ follows by the naturality of the braiding and definition of~$\tau_{n-1}(\varepsilon,u)$.

\section{Proof of Theorem \ref*{CMtrace}} \label{pfCMtrace}
In order to show that the family~$\{\alpha_n \co H^{\tens n} \to C^{\tens n+1}\}_{n\in \N}$ is a morphism between the paracocyclic objects~$\textbf{CM}_\bullet(H,\delta,\sigma)$ and~$\textbf{C}_\bullet(C)$ in $\mc B$, we will directly check that
\begin{align}
\label{commfacCM}\alpha_n \delta_i^n(\sigma)&= \delta_i^n \alpha_{n-1} \quad \text{for} \quad 0\le i \le n, n\ge 1,  \\
\label{commdegCM}\alpha_{n}\sigma_j^n&= \sigma_j^n \alpha_{n+1}\quad \text{for} \quad 0\le j \le n, n\ge 0, \\
\label{commcycCM}\alpha_{n}\tau_n(\delta,\sigma) &= \tau_n \alpha_n\quad \text{for} \quad n\ge 0.
\end{align}
Note that we abusively use the same notation for cofaces, codegeneracies, and paracocyclic operators of two different constructions. These should be understood from context.
Roughly described, the equalities $\eqref{commfacCM}$ and $\eqref{commdegCM}$ follow by the fact that $C$ is a coalgebra in the category of right $H$-modules.
In order to show the equality \eqref{commcycCM} in the case $n\ge 2$, we will need the following computation:

\begin{lem}\label{kindaflow} If $n\ge 2$, then
$$\,
\psfrag{b}{\hspace{-0.15cm}$H$}
\psfrag{A}{$\color{red}C$}
\psfrag{e}{\hspace{-0.07cm}$H^{\tens n}$}
\psfrag{id}{\hspace{-0.32cm}$\id_{H^{\tens n}}$}
\rsdraw{0.55}{0.75}{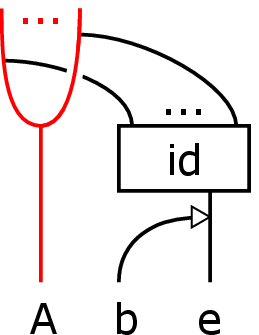}  \;=\,
\psfrag{b}{$1$}
\psfrag{A}{$\color{red}C$}
\psfrag{e}{$2$}
\psfrag{n}{\hspace{-0.07cm}$n$}
\psfrag{id}{\hspace{-0.32cm}$\id_{H^{\tens n}}$}
\rsdraw{0.55}{0.75}{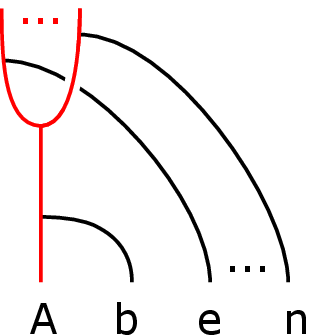}  \;.
$$
\end{lem}
\begin{proof}
We prove the claim by induction. Let us first show it for~$n=2$.
Indeed, by the right module axiom, the naturality of braiding and the fact that the comultiplication~$\Delta_C\co C\to C\tens C$ is an~$H$-linear morphism, we have
$$\,
\psfrag{b}{\hspace{-0.15cm}$H$}
\psfrag{A}{$C$}
\rsdraw{0.55}{0.75}{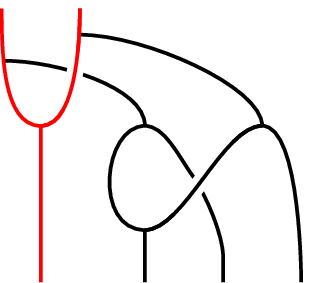}  \;=\,
\psfrag{b}{\hspace{-0.15cm}$H$}
\psfrag{A}{$C$}
\rsdraw{0.55}{0.75}{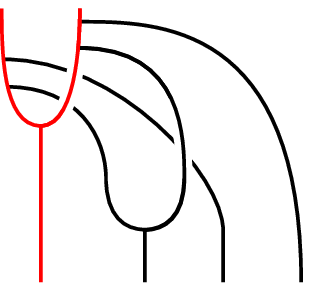}  \;=\,
\psfrag{b}{\hspace{-0.15cm}$H$}
\psfrag{A}{$C$}
\rsdraw{0.55}{0.75}{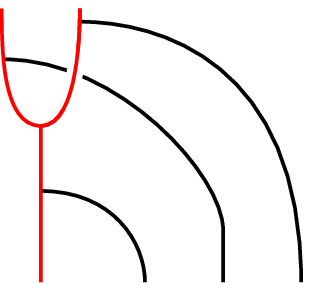}  \;.$$
Suppose that the claim is true for an~$n\ge 2$ and let us show it for~$n+1$.
We have
\begingroup
\allowdisplaybreaks
\begin{align*}\,
\psfrag{b}{\hspace{-0.15cm}$H$}
\psfrag{A}{$\color{red} C$}
\psfrag{e}{\hspace{-0.07cm}$H^{\tens n+1}$}
\psfrag{id}{\hspace{-0.3cm}$\id_{H^{\tens n}}$}
\rsdraw{0.55}{0.75}{flowtypelemma}  \;&\overset{(i)}{=}
\,
\psfrag{b}{\hspace{-0.15cm}$H$}
\psfrag{A}{$\color{red}C$}
\psfrag{id}{\hspace{-0.3cm}$\id_{H^{\tens n}}$}
\psfrag{e}{\hspace{-0.07cm}$H^{\tens n}$}
\rsdraw{0.55}{0.75}{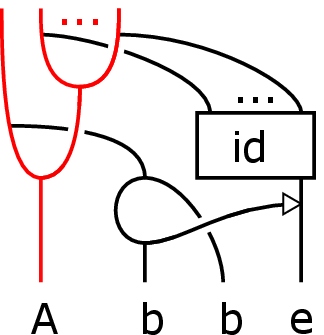}  \;\overset{(ii)}{=}
\,
\psfrag{b}{\hspace{-0.02cm}$1$}
\psfrag{A}{$\color{red}C$}
\psfrag{e}{$3$}
\psfrag{n}{\hspace{-0.5cm}$n+1$}
\psfrag{l}{\hspace{-0.07cm}$2$}
\rsdraw{0.55}{0.75}{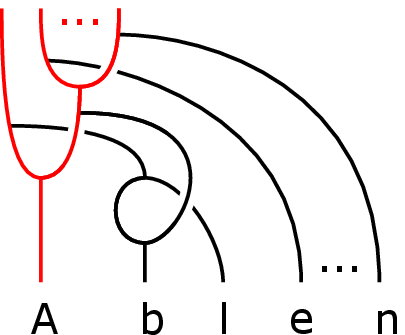}  \;\overset{(iii)}{=} \\ &\overset{(iii)}{=}
\,
\psfrag{b}{\hspace{-0.13cm}$1$}
\psfrag{A}{$\color{red}C$}
\psfrag{e}{$3$}
\psfrag{n}{\hspace{-0.5cm}$n+1$}
\psfrag{l}{\hspace{-0.07cm}$2$}
\rsdraw{0.55}{0.75}{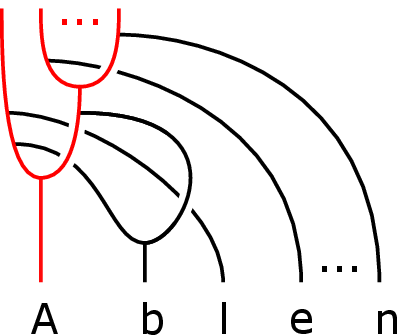}  \;\overset{(iv)}{=}\,
\psfrag{b}{\hspace{0.08cm}$1$}
\psfrag{A}{$\color{red}C$}
\psfrag{e}{$2$}
\psfrag{n}{\hspace{-0.4cm}$n+1$}
\psfrag{id}{\hspace{-0.32cm}$\id_{H^{\tens n}}$}
\rsdraw{0.55}{0.75}{flowtypelemma2}  \;, \end{align*}
\endgroup
which in\-deed proves the claim for $n+1$.
Here $(i)$ fol\-lows by the in\-duc\-tive de\-fi\-ni\-tion of the left di\-a\-go\-nal ac\-tion, $(ii)$ follows by the in\-duc\-tion hy\-po\-the\-sis, $(iii)$ follows by the right mo\-dule a\-xi\-om and $(iv)$ fol\-lows by the na\-tu\-ra\-li\-ty of the brai\-ding and the fact that the comul\-ti\-pli\-ca\-tion~$\Delta_C\co C\to C\tens C$ is an~$H$-li\-ne\-ar mor\-phism.
\end{proof}

\subsection{Proof of the equality \eqref{commfacCM}}
Let us show the equality~\eqref{commfacCM}.
If~$n=1$ and~$i=0$, then the equality~\eqref{commfacCM} writes as~$\alpha_1\delta_0^1=\delta_0^1\alpha_0,$ which follows by definitions and the right module axiom.
Let~$n=i=1$. In this case, the equality~\eqref{commfacCM} writes as~$\alpha_1\delta_1^1=\delta_1^1\alpha_0$, which is exactly the condition that~$\alpha\co \uu \to C$ is a~$\sigma$-trace. This shows the equality~\eqref{commfacCM} for~$n=1$ and~$0\le i \le 1$.

Now let $n\ge 2$ and $i=0$. By definitions, the naturality of the braiding, the axiom of a right module, and the coassociativity, we have:
$$\alpha_n\delta_0^n =
\, \psfrag{b}{\hspace{0.1cm}$1$}
\psfrag{A}{\hspace{-0.1cm}$C$}
\psfrag{e}{\hspace{-0.3cm}$n-1$}
\psfrag{g}{$\color{red}\alpha$}
\rsdraw{0.55}{0.75}{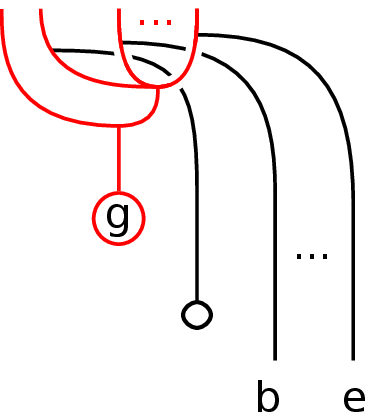} \; =\, \psfrag{b}{\hspace{0.1cm}$1$}
\psfrag{A}{\hspace{-0.1cm}$C$}
\psfrag{e}{\hspace{-0.3cm}$n-1$}
\psfrag{g}{$\color{red}\alpha$}
\rsdraw{0.55}{0.75}{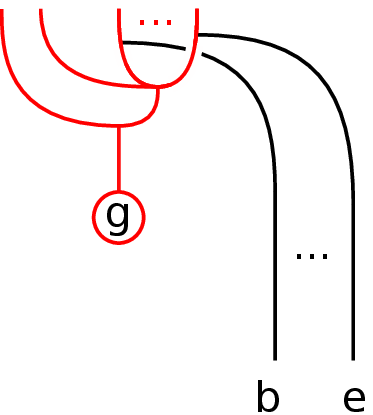} \;=\, \psfrag{b}{\hspace{0.1cm}$1$}
\psfrag{A}{\hspace{-0.1cm}$C$}
\psfrag{e}{\hspace{-0.3cm}$n-1$}
\psfrag{g}{$\color{red}\alpha$}
\rsdraw{0.55}{0.75}{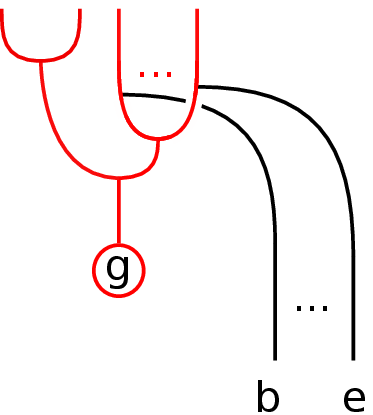} \; = \delta_0^n\alpha_{n-1}.$$
Let~$n\ge 2$ and~$1\le i \le n-1$.
In this case, equation \eqref{commfacCM} follows by definitions, the coassociativity, and the fact that~$\Delta_C\co C\to C\tens C$ is an~$H$-linear morphism:
\begingroup
\allowdisplaybreaks
\begin{align*}\alpha_n\delta_i^n&=\, \psfrag{b}{\hspace{-0.05cm}$1$}
\psfrag{A}{\hspace{-0.1cm}$C$}
\psfrag{j}{$i$}
\psfrag{e}{\hspace{-0.3cm}$n-1$}
\psfrag{g}{$\color{red}\alpha$}
\rsdraw{0.55}{0.75}{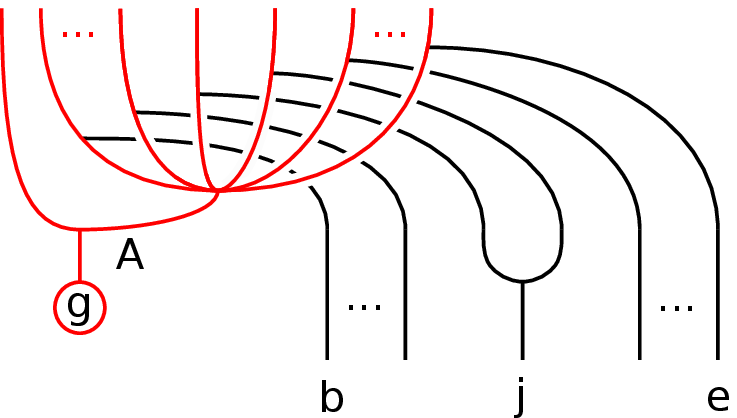} \;= \, \psfrag{b}{\hspace{-0.05cm}$1$}
\psfrag{A}{\hspace{-0.1cm}$C$}
\psfrag{j}{$i$}
\psfrag{e}{\hspace{-0.3cm}$n-1$}
\psfrag{g}{$\color{red}\alpha$}
\rsdraw{0.55}{0.75}{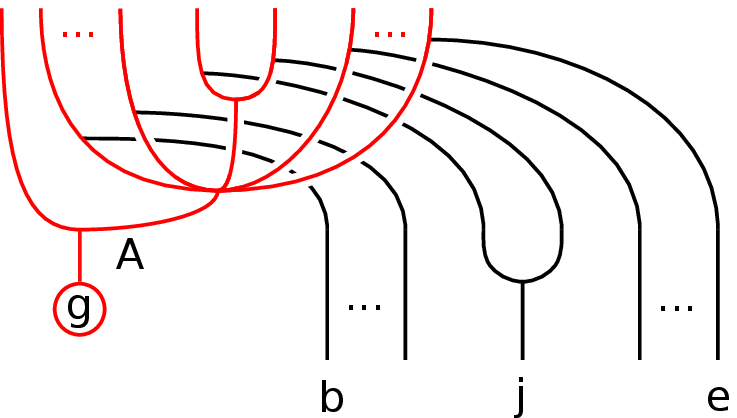} \; \\&=\, \psfrag{b}{\hspace{-0.05cm}$1$}
\psfrag{A}{\hspace{-0.1cm}$C$}
\psfrag{j}{$i$}
\psfrag{e}{\hspace{-0.3cm}$n-1$}
\psfrag{g}{$\color{red} \alpha$}
\rsdraw{0.55}{0.75}{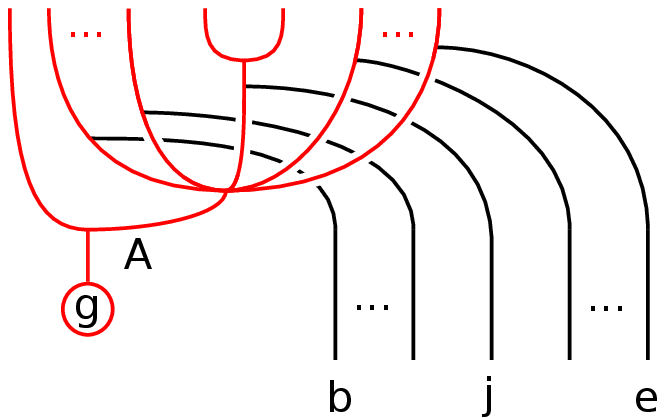} \; =\delta_i^n\alpha_{n-1}.\end{align*}
\endgroup
Finally, let~$n\ge 2$ and~$i=n$.
Equation~\eqref{commfacCM} in this case follows by the case~$n\ge 2$ and~$i=0$, which is written above, the equation~\eqref{commcycCM}, which is proven in Section \ref{proofcommcyccm}, and by the paracyclic compatibility relation~$\tau_n\delta_0^n=\delta_n^n$ (see Section \ref{paracyccat}).
Indeed, we have
\begin{align*}
\alpha_n\delta_n^n&=\alpha_n(\tau_n(\delta,\sigma)\delta_0^n)=(\alpha_n\tau_n(\delta,\sigma))\delta_0^n=(\tau_n\alpha_n)\delta_0^n=\\
&=\tau_n(\alpha_n\delta_0^n)=\tau_n(\delta_0^n\alpha_{n-1})=(\tau_n\delta_0^n)\alpha_{n-1}=\delta_n^n\alpha_{n-1}.
\end{align*}

\subsection{Proof of the equality \eqref{commdegCM}} Let us show the equality~\eqref{commdegCM}. We consider the three following cases:~$j=0$,~$1\le j\le n-1$, and~$j=n$.
In each case, the desired equality follows by definition, the counitality, the fact that the counit~$\varepsilon_C\co C\to \uu$ is an~$H$-linear morphism, and the naturality of the braiding.

Indeed, if $j=0$, then we have
$$\alpha_n\sigma_0^n=\,
\psfrag{b}{\hspace{-0.1cm}$1$}
\psfrag{A}{$C$}
\psfrag{e}{\hspace{-0.1cm}$n$}
\psfrag{n}{\hspace{-0.1cm}$0$}
\psfrag{g}{$\color{red}\alpha$}
\rsdraw{0.55}{0.75}{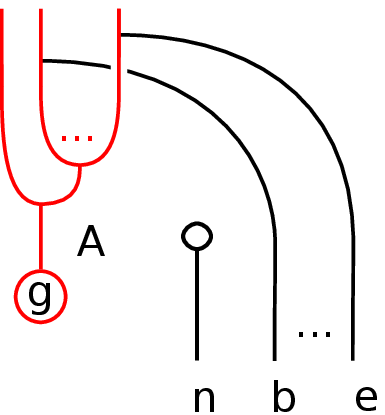}  \; = \,
\psfrag{b}{\hspace{-0.1cm}$1$}
\psfrag{A}{$C$}
\psfrag{e}{\hspace{-0.1cm}$n$}
\psfrag{n}{\hspace{-0.1cm}$0$}
\psfrag{g}{$\color{red}\alpha$}
\rsdraw{0.55}{0.75}{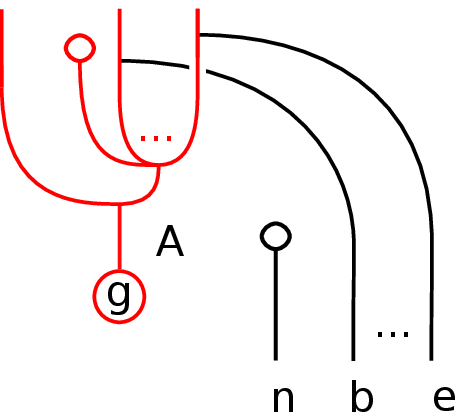}  \;= \,
\psfrag{b}{\hspace{-0.1cm}$1$}
\psfrag{A}{$C$}
\psfrag{e}{\hspace{-0.1cm}$n$}
\psfrag{n}{\hspace{-0.1cm}$0$}
\psfrag{g}{$\color{red}\alpha$}
\rsdraw{0.55}{0.75}{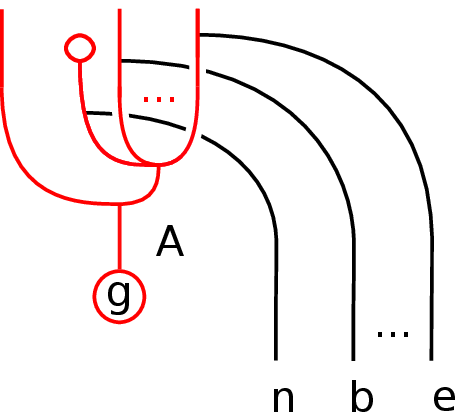}  \; = \sigma_0^n\alpha_{n+1}.$$
If $1\le j \le n-1$, then
\begingroup
\allowdisplaybreaks
\begin{align*}\alpha_n\sigma_j^n&=\,
\psfrag{b}{\hspace{-0.1cm}$1$}
\psfrag{A}{$C$}
\psfrag{e}{$n+1$}
\psfrag{i}{\hspace{-0.55cm}$j+1$}
\psfrag{n}{\hspace{-0.2cm}$j$}
\psfrag{m}{\hspace{-0.2cm}$j+2$}
\psfrag{g}{$\color{red}\alpha$}
\rsdraw{0.55}{0.75}{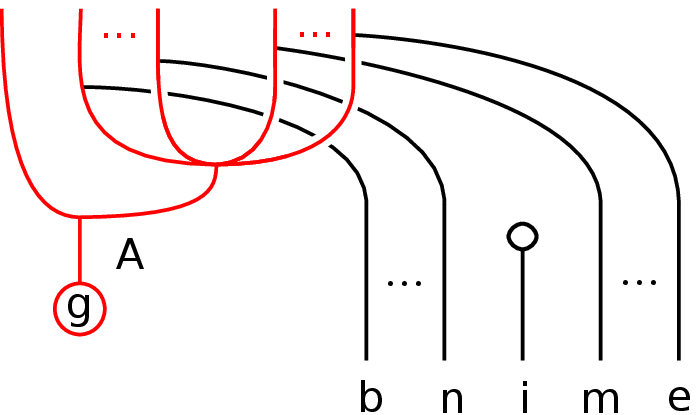}  \;=\,
\psfrag{b}{\hspace{-0.1cm}$1$}
\psfrag{A}{$C$}
\psfrag{e}{$n+1$}
\psfrag{i}{\hspace{-0.55cm}$j+1$}
\psfrag{n}{\hspace{-0.2cm}$j$}
\psfrag{m}{\hspace{-0.2cm}$j+2$}
\psfrag{g}{$\color{red}\alpha$}
\rsdraw{0.55}{0.75}{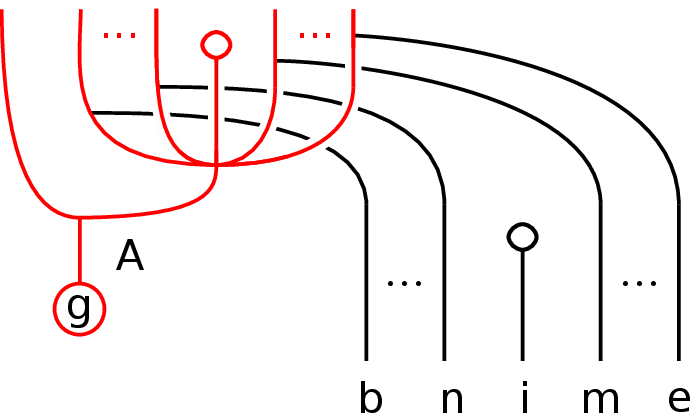}  \; = \\ &= \,
\psfrag{b}{\hspace{-0.1cm}$1$}
\psfrag{A}{$C$}
\psfrag{e}{$n+1$}
\psfrag{i}{\hspace{-0.55cm}$j+1$}
\psfrag{n}{\hspace{-0.2cm}$j$}
\psfrag{m}{\hspace{-0.2cm}$j+2$}
\psfrag{g}{$\color{red}\alpha$}
\rsdraw{0.55}{0.75}{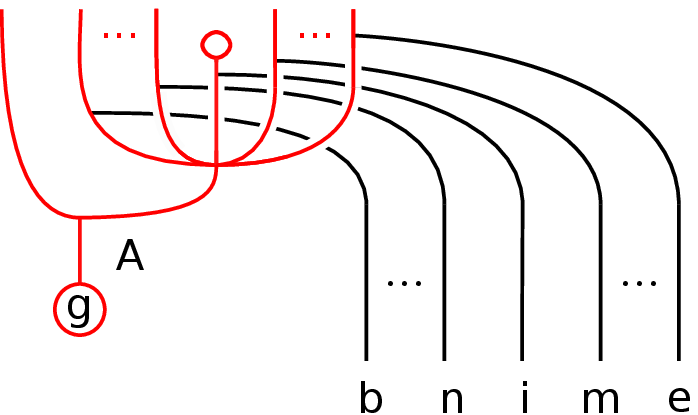}  \;=\sigma_j^n\alpha_{n+1}.\end{align*}
\endgroup
Finally, if $j=n$, then we have
$$\alpha_n\sigma_n^n
=\,
\psfrag{b}{\hspace{0.1cm}$1$}
\psfrag{A}{$C$}
\psfrag{e}{$n$}
\psfrag{n}{\hspace{-0.5cm}$n+1$}
\psfrag{g}{$\color{red}\alpha$}
\rsdraw{0.55}{0.75}{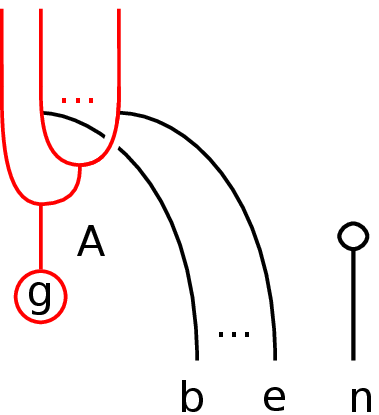}  \;=\,
\psfrag{b}{\hspace{0.1cm}$1$}
\psfrag{A}{$C$}
\psfrag{e}{$n$}
\psfrag{n}{\hspace{-0.5cm}$n+1$}
\psfrag{g}{$\color{red}\alpha$}
\rsdraw{0.55}{0.75}{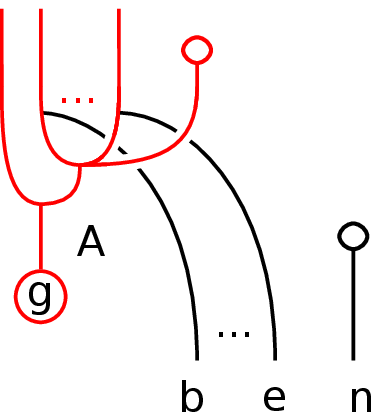}  \; =\,
\psfrag{b}{\hspace{0.1cm}$1$}
\psfrag{A}{$C$}
\psfrag{e}{$n$}
\psfrag{n}{\hspace{-0.5cm}$n+1$}
\psfrag{g}{$\color{red}\alpha$}
\rsdraw{0.55}{0.75}{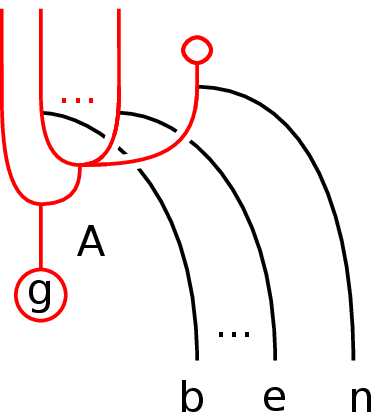}  \; = \sigma_n^n\alpha_{n+1}.$$

\subsection{Proof of the equality \eqref{commcycCM}} \label{proofcommcyccm} Let us verify that equation \eqref{commcycCM} holds.
When $n=0$, this holds since twist morphisms are natural, $\theta_\uu=\id_\uu$, and since $\tau_0(\delta,\sigma)=\id_\uu$.
Indeed,
$$\tau_0 \alpha_0 = \theta_C \alpha_0 =\alpha_0= \alpha_0 \tau_0(\delta,\sigma).$$
Let us check it for the case $n=1$. Indeed, we have
\begingroup
\allowdisplaybreaks
\begin{align*} \alpha_1\tau_1(\delta,\sigma)&\overset{(i)}{=}\,
\psfrag{A}{\hspace{-0.15cm}$C$}
\psfrag{g}{$\color{red}\alpha$}
\psfrag{c}{$\sigma$}
\psfrag{b}{\hspace{-0.2cm}$H$}
\psfrag{l}{\hspace{-0.07cm}$\delta$}
\rsdraw{0.55}{0.75}{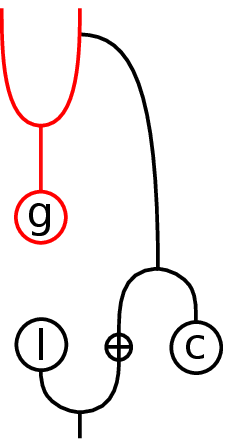}  \; \overset{(ii)}{=} \,
\psfrag{A}{$C$}
\psfrag{g}{$\color{red}\alpha$}
\psfrag{c}{$\sigma$}
\psfrag{b}{\hspace{-0.2cm}$H$}
\psfrag{l}{\hspace{-0.1cm}$\delta$}
\rsdraw{0.55}{0.75}{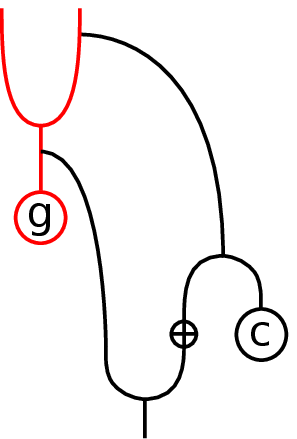}  \;\overset{(iii)}{=}\,
\psfrag{A}{$C$}
\psfrag{g}{$\color{red}\alpha$}
\psfrag{c}{$\sigma$}
\psfrag{b}{\hspace{-0.2cm}$H$}
\psfrag{l}{\hspace{-0.1cm}$\delta$}
\rsdraw{0.55}{0.75}{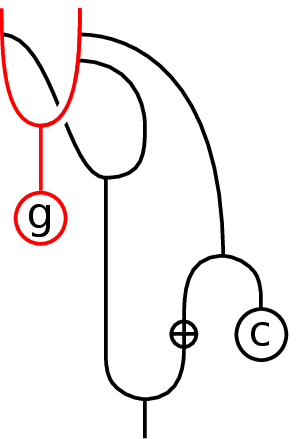}  \; \overset{(iv)}{=} \,
\psfrag{A}{$C$}
\psfrag{g}{$\color{red}\alpha$}
\psfrag{c}{$\sigma$}
\psfrag{b}{\hspace{-0.2cm}$H$}
\psfrag{l}{\hspace{-0.1cm}$\delta$}
\rsdraw{0.55}{0.75}{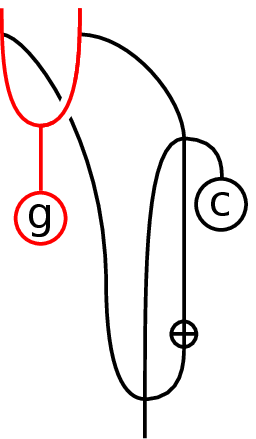}  \; \overset{(v)}{=}\\ &\overset{(v)}{=}\hspace{0.5cm}\,
\psfrag{A}{$C$}
\psfrag{g}{$\color{red}\alpha$}
\psfrag{c}{$\sigma$}
\psfrag{b}{\hspace{-0.2cm}$H$}
\psfrag{l}{\hspace{-0.1cm}$\delta$}
\rsdraw{0.55}{0.75}{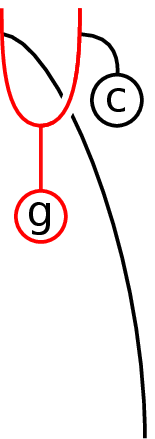}  \; \overset{(vi)}{=} \hspace{0.5cm}\,
\psfrag{A}{$C$}
\psfrag{g}{$\color{red}\alpha$}
\psfrag{c}{$\sigma$}
\psfrag{b}{\hspace{-0.2cm}$H$}
\psfrag{l}{\hspace{-0.1cm}$\delta$}
\rsdraw{0.55}{0.75}{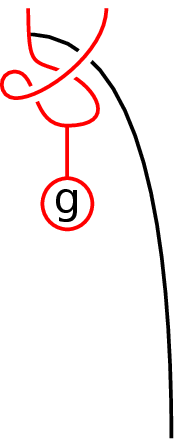}  \;\overset{(vii)}{=}\hspace{0.5cm}\,
\psfrag{A}{$C$}
\psfrag{g}{$\color{red}\alpha$}
\psfrag{c}{$\sigma$}
\psfrag{b}{\hspace{-0.2cm}$H$}
\psfrag{l}{\hspace{-0.1cm}$\delta$}
\rsdraw{0.55}{0.75}{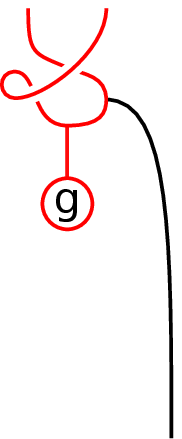}  \;\hspace{0.5cm}\overset{(viii)}{=} \tau_1\alpha_1.\end{align*}
\endgroup
Here~$(i)$ and~$(viii)$ follow by definition,~$(ii)$ follows by the fact that~$\alpha$ is~$\delta$-invariant,~$(iii)$ follows from the fact that the comultiplication~$\Delta_C\co C\to C\tens C$ is~$H$-linear.
The e\-qua\-li\-ty~$(iv)$ follows by the (co)associativity,~$(v)$ follows by the antipode axiom and the (co)unitality,~$(vi)$ follows from the fact that~$\alpha$ is a~$\sigma$-trace and~$(vii)$ follows by the naturality of the braiding.

Finally, let us check the equality \eqref{commcycCM} when $n\ge 2$. Indeed, we have
\begingroup
\allowdisplaybreaks
\begin{align*}\alpha_n\tau_n(\delta,\sigma)&\overset{(i)}{=}\,
\psfrag{A}{$C$}
\psfrag{g}{$\color{red}\alpha$}
\psfrag{c}{$\sigma$}
\psfrag{i}{\hspace{-0.1cm}$2$}
\psfrag{e}{$n$}
\psfrag{b}{$1$}
\psfrag{l}{\hspace{-0.1cm}$\delta$}
\rsdraw{0.55}{0.75}{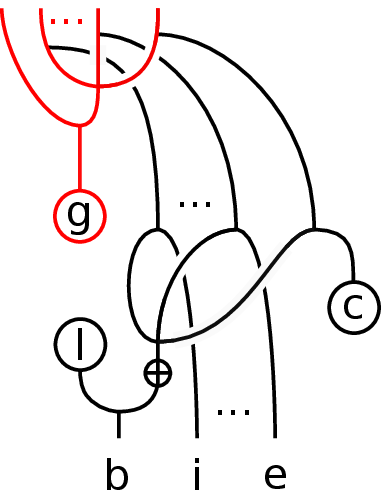}  \;\overset{(ii)}{=}\,
\psfrag{A}{$C$}
\psfrag{b}{\hspace{0.1cm}$1$}
\psfrag{i}{\hspace{-0.1cm}$2$}
\psfrag{g}{$\color{red}\alpha$}
\psfrag{c}{$\sigma$}
\psfrag{b}{$1$}
\psfrag{l}{\hspace{-0.1cm}$\delta$}
\rsdraw{0.55}{0.75}{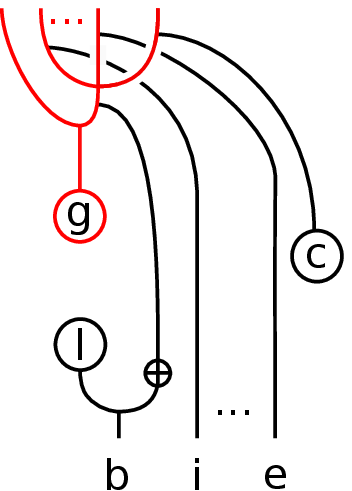}  \;\overset{(iii)}{=}\,
\psfrag{A}{$C$}
\psfrag{b}{\hspace{0.1cm}$1$}
\psfrag{i}{\hspace{-0.1cm}$2$}
\psfrag{g}{$\color{red}\alpha$}
\psfrag{c}{$\sigma$}
\psfrag{b}{$1$}
\psfrag{l}{\hspace{-0.1cm}$\delta$}
\rsdraw{0.55}{0.75}{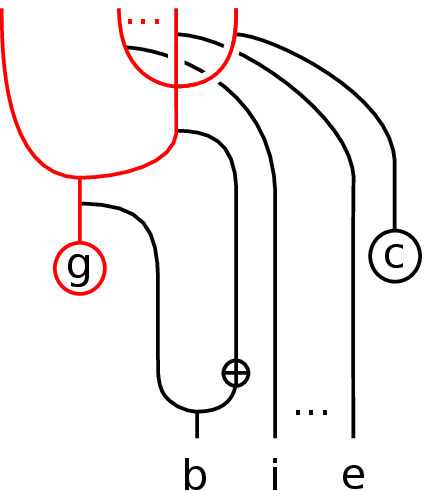}  \;\overset{(iv)}{=} \\ &\overset{(iv)}{=}\,
\psfrag{A}{$C$}
\psfrag{b}{\hspace{0.1cm}$1$}
\psfrag{i}{\hspace{-0.1cm}$2$}
\psfrag{g}{$\color{red}\alpha$}
\psfrag{c}{$\sigma$}
\psfrag{e}{$n$}
\psfrag{b}{$1$}
\psfrag{l}{\hspace{-0.1cm}$\delta$}
\rsdraw{0.55}{0.75}{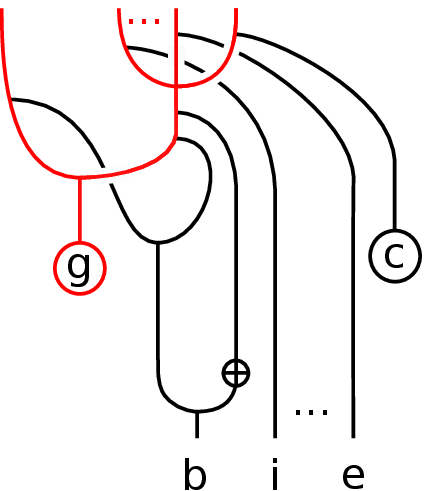}  \; \overset{(v)}{=} \,
\psfrag{A}{$C$}
\psfrag{b}{\hspace{0.1cm}$1$}
\psfrag{i}{\hspace{-0.1cm}$2$}
\psfrag{g}{$\color{red}\alpha$}
\psfrag{c}{$\sigma$}
\psfrag{b}{$1$}
\psfrag{e}{$n$}
\psfrag{l}{\hspace{-0.1cm}$\delta$}
\rsdraw{0.55}{0.75}{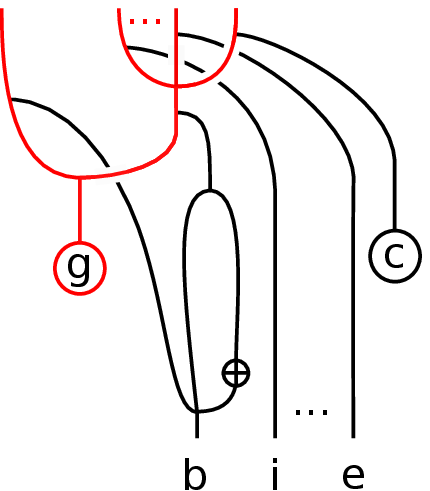} \;\overset{(vi)}{=}\,
\psfrag{A}{$C$}
\psfrag{b}{\hspace{0.1cm}$1$}
\psfrag{i}{\hspace{-0.1cm}$2$}
\psfrag{g}{$\color{red}\alpha$}
\psfrag{c}{$\sigma$}
\psfrag{b}{$1$}
\psfrag{e}{$n$}
\psfrag{l}{\hspace{-0.1cm}$\delta$}
\rsdraw{0.55}{0.75}{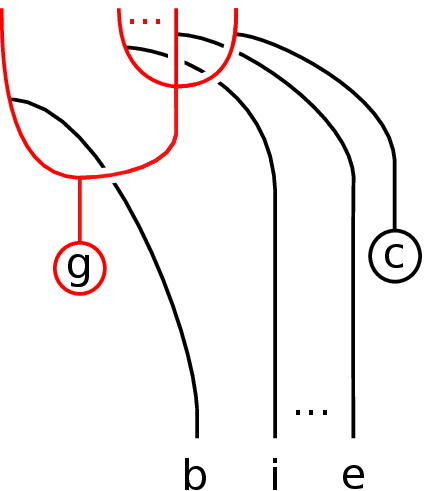} \; \overset{(vii)}{=} \\
&\overset{(vii)}{=}\,
\psfrag{A}{$C$}
\psfrag{b}{\hspace{0.1cm}$1$}
\psfrag{i}{\hspace{-0.1cm}$2$}
\psfrag{g}{$\color{red}\alpha$}
\psfrag{c}{$\sigma$}
\psfrag{b}{$1$}
\psfrag{e}{$n$}
\psfrag{l}{\hspace{-0.1cm}$\delta$}
\rsdraw{0.55}{0.75}{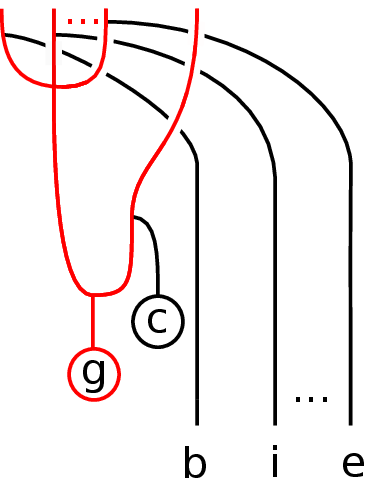} \; \overset{(viii)}{=}\,
\psfrag{A}{$C$}
\psfrag{b}{\hspace{0.1cm}$1$}
\psfrag{i}{\hspace{-0.1cm}$2$}
\psfrag{g}{$\color{red}\alpha$}
\psfrag{c}{$\sigma$}
\psfrag{b}{$1$}
\psfrag{e}{$n$}
\psfrag{l}{\hspace{-0.1cm}$\delta$}
\rsdraw{0.55}{0.75}{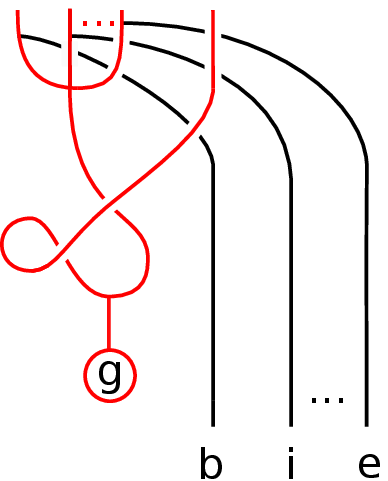} \;\overset{(ix)}{=}\,
\psfrag{A}{$C$}
\psfrag{b}{\hspace{0.1cm}$1$}
\psfrag{i}{\hspace{-0.1cm}$2$}
\psfrag{g}{$\color{red}\alpha$}
\psfrag{c}{$\sigma$}
\psfrag{b}{$1$}
\psfrag{e}{$n$}
\psfrag{l}{\hspace{-0.1cm}$\delta$}
\rsdraw{0.55}{0.75}{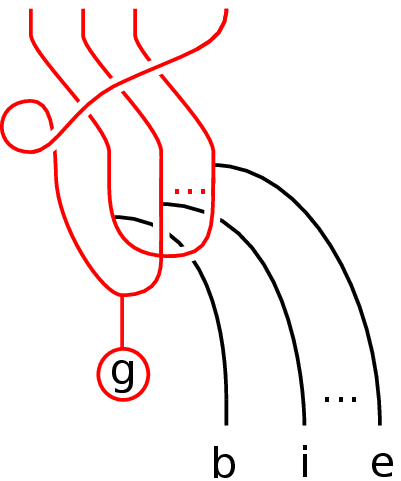} \;\overset{(x)}{=}\tau_n\alpha_n.
\end{align*}
\endgroup
Here~$(i)$ and~$(x)$ follow by de\-fi\-ni\-tion,~$(ii)$ follows by ap\-plying Lemma~\ref{kindaflow},~$(iii)$ follows from the fact that~$\alpha$ is~$\delta$-invariant,~$(iv)$ follows from the fact that the comultiplication~$\Delta_C\co C\to C\tens C$ is~$H$-linear.
The equality~$(v)$ follows from the coassociativity and the right module axiom,~$(vi)$ follows by the antipode axiom, the (co)unitality, and the right module axiom,~$(vii)$ follows by the coassociativity and the naturality of the braiding,~$(viii)$ follows from the fact that~$\alpha$ is a~$\sigma$-trace and finally,~$(ix)$ follows by the naturality of the braiding.

\section{Appendix} \label{appendix}
In this appendix, we ve\-ri\-fy pa\-ra\-cy\-clic com\-pa\-ti\-bi\-li\-ty re\-la\-tions~\eqref{pcr} of the pa\-ra\-co\-cy\-clic ob\-ject~$\textbf{CM}_\bullet(H,\delta,\sigma)$ defined in Section~\ref{recall}.
Note that the ve\-ri\-fi\-ca\-tion of sim\-pli\-cial re\-la\-tions~\eqref{sr} for this ob\-ject is an easy task.
One can show it gra\-phi\-cal\-ly, by using the level-exchange pro\-per\-ty (see Sec\-tion~\ref{cg}), the co\-as\-so\-ci\-a\-ti\-vi\-ty, and the co\-u\-ni\-ta\-li\-ty. Also, the re\-la\-tion~$\tau_n\sigma_i^n=\sigma_{i-1}^n\tau_{n+1}$ for~$1\le i \le n$ fol\-lows from the bi\-al\-ge\-bra axiom and by the na\-tu\-ra\-li\-ty of the brai\-ding.
In this appendix, we show that
\begin{enumerate}
  \item \label{tndi} $\tau_n \delta_i^n= \delta_{i-1}^n \tau_{n-1}$ for $n\in \N^*$ and ~$1\le i \le n$,
  \item \label{tnd0} $\tau_n\delta_0^n=\delta_n^n$ for $n\in \N^*$,
  \item \label{tns0} $\tau_n(\delta,\sigma)\sigma_0^n=\sigma_n^n(\tau_{n+1}(\delta,\sigma))^2$ for~$n\in \N$.
\end{enumerate}
Verifications of~\eqref{tndi}-\eqref{tns0} are somewhat involved. Note that~\eqref{tns0} is proved for~$n=2$ in~\cite[Theorem 7.1.]{khalkhalipourkia}. Recall the notion of twisted antipodes and the corresponding notation from Section~\ref{proofmain}.

\subsection{The relation~$\tau_n \delta_i^n= \delta_{i-1}^n \tau_{n-1}$}

If $n=1$ and $i=1$, the relation rewrites as~$\tau_1 \delta_1^1= \delta_{0}^1 \tau_{0}$ and it follows by definitions, the fact that $\sigma$ is a coalgebra morphism, the fact that $(\delta,\sigma)$ is a modular pair and by the antipode axiom:
$$\tau_1 \delta_1^1=\,
\psfrag{c}{$\sigma$}
\psfrag{l}{\hspace{-0.1cm}$\delta$}
\rsdraw{0.55}{0.75}{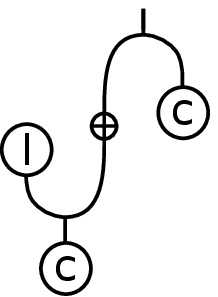}  \; = \,
\psfrag{c}{$\sigma$}
\psfrag{l}{\hspace{-0.1cm}$\delta$}
\rsdraw{0.55}{0.75}{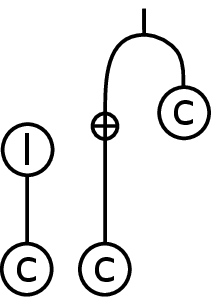}  \; =
\,
\psfrag{c}{$\sigma$}
\psfrag{l}{\hspace{-0.1cm}$\delta$}
\rsdraw{0.55}{0.75}{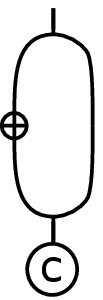}  \; = u = \delta_0^1\tau_0.$$
If~$n=2$ and~$i=1$, the relation rewrites as~$\tau_2\delta_1^2=\delta_0^2\tau_1$ and it follows by definitions, Lemma~\ref{algebraicppties}$a)$, the coassociativity, the antipode axiom, and the naturality of the braiding:
$$\tau_2 \delta_1^2=\,
\psfrag{b}{\hspace{0.1cm}$H$}
\psfrag{c}{$\sigma$}
\psfrag{l}{\hspace{-0.1cm}$\delta$}
\psfrag{v}[tl][tl]{$\thicksim$}
\rsdraw{0.55}{0.75}{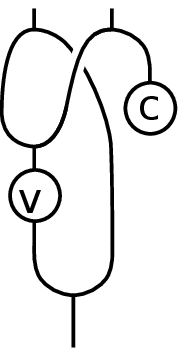}  \; =
\,
\psfrag{b}{\hspace{0.1cm}$H$}
\psfrag{c}{$\sigma$}
\psfrag{v}[tl][tl]{$\thicksim$}
\psfrag{l}{\hspace{-0.1cm}$\delta$}
\rsdraw{0.55}{0.75}{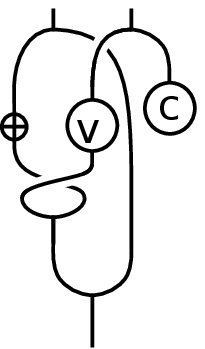}  \; =
\,
\psfrag{b}{\hspace{0.1cm}$H$}
\psfrag{c}{$\sigma$}
\psfrag{v}[tl][tl]{$\thicksim$}
\psfrag{l}{\hspace{-0.1cm}$\delta$}
\rsdraw{0.55}{0.75}{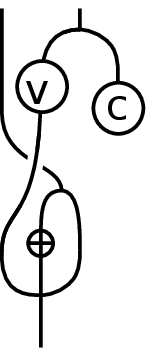}  \; =
\,
\psfrag{b}{\hspace{0.1cm}$H$}
\psfrag{v}[tl][tl]{$\thicksim$}
\psfrag{c}{$\sigma$}
\psfrag{l}{\hspace{-0.1cm}$\delta$}
\rsdraw{0.55}{0.75}{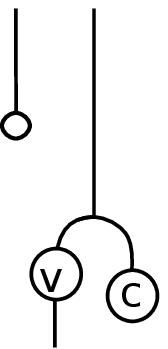}  \;  = \delta_0^2\tau_1.$$
If $n=2$ and $i=2$, the relation rewrites as $\tau_2\delta_2^2=\delta_1^2\tau_1$ and it follows by definitions, the fact that $\sigma$ is a coalgebra morphism, and the fact that multiplication is an algebra morphism:
$$\tau_2\delta_2^2= \,
\psfrag{b}{\hspace{0.1cm}$H$}
\psfrag{v}[tl][tl]{$\thicksim$}
\psfrag{c}{$\sigma$}
\rsdraw{0.55}{0.75}{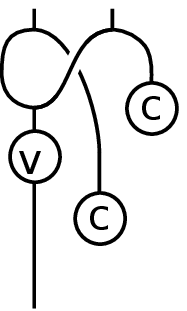}  \;=\,
\psfrag{b}{\hspace{0.1cm}$H$}
\psfrag{v}[tl][tl]{$\thicksim$}
\psfrag{c}{$\sigma$}
\rsdraw{0.55}{0.75}{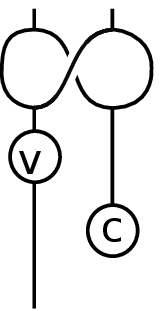}  \;=\,
\psfrag{b}{\hspace{0.1cm}$H$}
\psfrag{v}[tl][tl]{$\thicksim$}
\psfrag{c}{$\sigma$}
\rsdraw{0.55}{0.75}{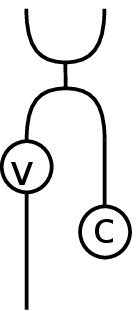}  \;=\delta_1^2\tau_1.$$
Now let $n\ge 3$.
For $i=1$, the relation rewrites as $\tau_n\delta_1^n=\delta_0^n\tau_{n-1}$ and it is true since
\begin{align*}
\tau_n\delta_1^n&\overset{(i)}{=}\,
\psfrag{j}{$1$}
\psfrag{b}{\hspace{-0.1cm}$2$}
\psfrag{e}{\hspace{-0.4cm}$n-1$}
\psfrag{v}[tl][tl]{$\thicksim$}
\psfrag{f}{\hspace{-0.6cm}$\id_{H^{\tens n-1}}$}
\psfrag{c}{$\sigma$}
\rsdraw{0.55}{0.75}{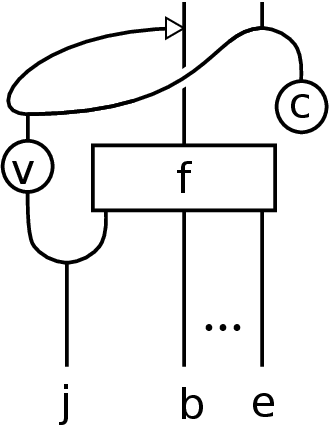}  \;\overset{(ii)}{=}\,
\psfrag{j}{\hspace{-0.15cm}$H$}
\psfrag{b}{\hspace{-0.1cm}$H^{\tens n-2}$}
\psfrag{e}{\hspace{-0.4cm}$n-1$}
\psfrag{v}[tl][tl]{$\thicksim$}
\psfrag{c}{$\sigma$}
\rsdraw{0.55}{0.75}{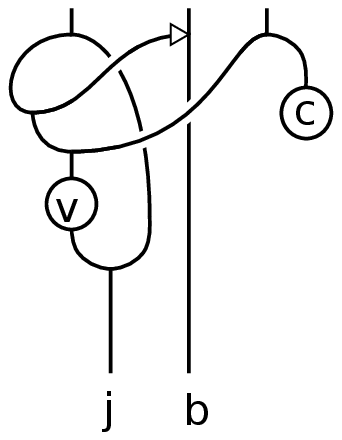}  \;\overset{(iii)}{=}\,
\psfrag{j}{\hspace{-0.15cm}$H$}
\psfrag{b}{\hspace{-0.3cm}$H^{\tens n-2}$}
\psfrag{e}{\hspace{-0.4cm}$n-1$}
\psfrag{v}[tl][tl]{$\thicksim$}
\psfrag{c}{$\sigma$}
\psfrag{f}{\hspace{-0.55cm}$\id_{H^{\tens n-2}}$}
\rsdraw{0.55}{0.75}{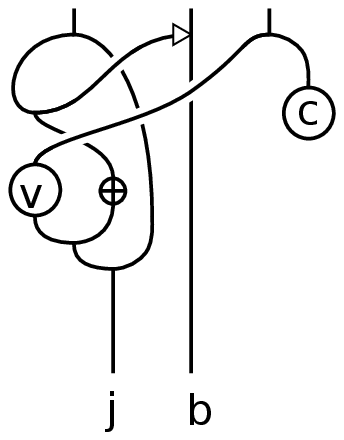}  \;\overset{(iv)}{=}\,
\psfrag{j}{\hspace{-0.15cm}$H$}
\psfrag{b}{\hspace{-0.3cm}$H^{\tens n-2}$}
\psfrag{e}{\hspace{-0.4cm}$n-1$}
\psfrag{v}[tl][tl]{$\thicksim$}
\psfrag{c}{$\sigma$}
\psfrag{f}{\hspace{-0.55cm}$\id_{H^{\tens n-2}}$}
\rsdraw{0.55}{0.75}{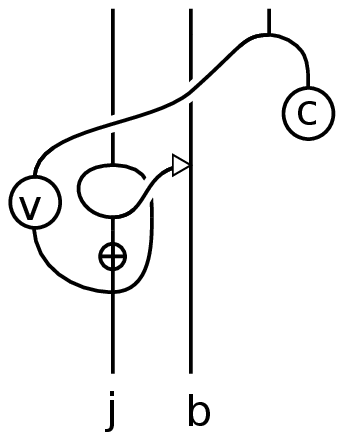}  \;\overset{(v)}{=}\\
&\overset{(v)}{=}\,
\psfrag{j}{\hspace{-0.15cm}$H$}
\psfrag{b}{\hspace{-0.3cm}$H^{\tens n-2}$}
\psfrag{e}{\hspace{-0.4cm}$n-1$}
\psfrag{v}[tl][tl]{$\thicksim$}
\psfrag{c}{$\sigma$}
\psfrag{f}{\hspace{-0.55cm}$\id_{H^{\tens n-2}}$}
\rsdraw{0.55}{0.75}{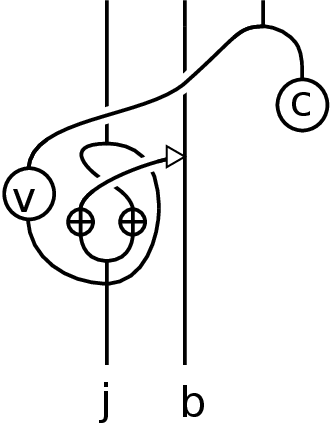}  \;\overset{(vi)}{=}\,
\psfrag{j}{\hspace{-0.15cm}$H$}
\psfrag{b}{\hspace{-0.3cm}$H^{\tens n-2}$}
\psfrag{e}{\hspace{-0.4cm}$n-1$}
\psfrag{v}[tl][tl]{$\thicksim$}
\psfrag{c}{$\sigma$}
\psfrag{f}{\hspace{-0.55cm}$\id_{H^{\tens n-2}}$}
\rsdraw{0.55}{0.75}{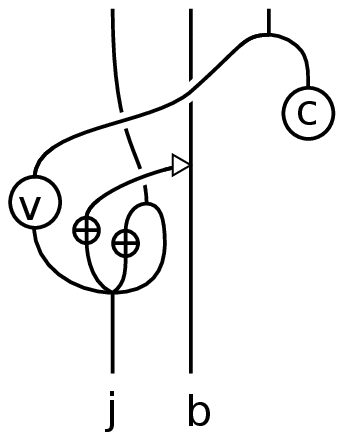}  \;\overset{(vii)}{=}\,
\psfrag{j}{\hspace{-0.15cm}$H$}
\psfrag{b}{\hspace{-0.3cm}$H^{\tens n-2}$}
\psfrag{e}{\hspace{-0.4cm}$n-1$}
\psfrag{v}[tl][tl]{$\thicksim$}
\psfrag{c}{$\sigma$}
\psfrag{f}{\hspace{-0.55cm}$\id_{H^{\tens n-2}}$}
\rsdraw{0.55}{0.75}{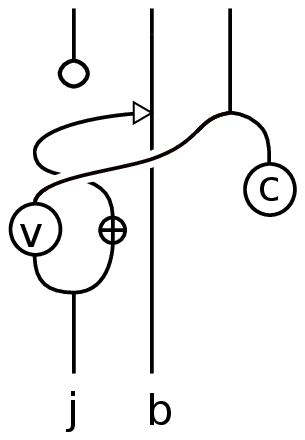}  \;\overset{(viii)}{=}\,
\psfrag{j}{\hspace{-0.15cm}$H$}
\psfrag{b}{\hspace{-0.3cm}$H^{\tens n-2}$}
\psfrag{e}{\hspace{-0.4cm}$n-1$}
\psfrag{v}[tl][tl]{$\thicksim$}
\psfrag{c}{$\sigma$}
\psfrag{f}{\hspace{-0.55cm}$\id_{H^{\tens n-2}}$}
\rsdraw{0.55}{0.75}{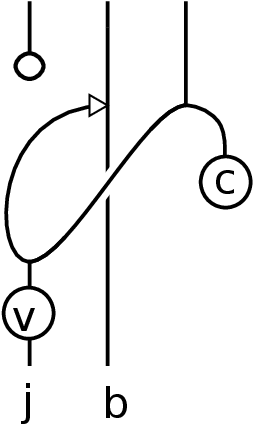}  \; \overset{(ix)}{=} \delta_0^n\tau_{n-1}.\end{align*}
Here $(i)$ and $(ix)$ follow by definition, $(ii)$ follows by inductive definition of the left diagonal action, $(iii)$ and $(viii)$ follow from Lemma \ref{algebraicppties}$(a)$,
$(iv)$ and $(vi)$ follow by the naturality of the braiding and the coassociativity, $(v)$ follows by the anti-comultiplicativity of the antipode, $(vii)$ follow by the antipode axiom, the counitality, and the naturality of the braiding.

\noindent For~$2\le i \le n-1$, the relation~$\tau_n\delta_i^n= \delta_{i-1}^{n}\tau_{n-1}$ is a consequence of the fact that multiplication is an algebra morphism and the naturality of the braiding.
Let us check~$\tau_n \delta_i^n = \delta_{i-1}^n \tau_{n-1}$ for~$i=n$.
The relation follows by definition, Lemma~\ref{recurrencerelations}$b)$, by the fact that~$\sigma$ is a coalgebra morphism, and by the fact that the comultiplication is an algebra morphism:
\[\tau_n\delta_n^n=
\,
\psfrag{b}{$H$}
\psfrag{i}{\hspace{-0.5cm}$H^{\tens n-2}$}
\psfrag{v}{\hspace{0.05cm}$\delta$}
\psfrag{s}{$\sigma$}
\psfrag{f}{$\hspace{-0.55cm}\tau_{n-1}(\varepsilon,u)$}
\rsdraw{0.55}{0.75}{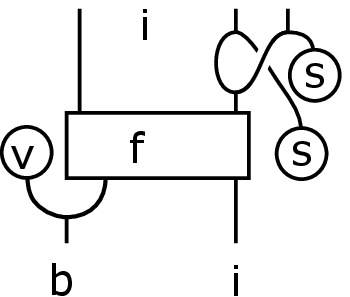}\;=\,
\psfrag{b}{$H$}
\psfrag{i}{\hspace{-0.5cm}$H^{\tens n-2}$}
\psfrag{v}{\hspace{0.05cm}$\delta$}
\psfrag{s}{$\sigma$}
\psfrag{f}{$\hspace{-0.55cm}\tau_{n-1}(\varepsilon,u)$}
\rsdraw{0.55}{0.75}{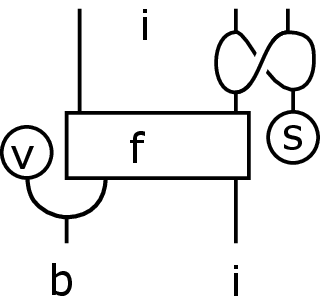}\;=\,
\psfrag{b}{$H$}
\psfrag{i}{\hspace{-0.5cm}$H^{\tens n-2}$}
\psfrag{v}{\hspace{0.05cm}$\delta$}
\psfrag{s}{$\sigma$}
\psfrag{f}{$\hspace{-0.55cm}\tau_{n-1}(\varepsilon,u)$}
\rsdraw{0.55}{0.75}{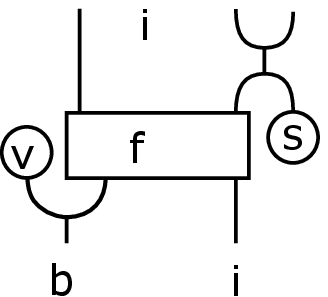}\;=\delta_{n-1}^n\tau_{n-1}.\]

\subsection{The relation~$\tau_n\delta_0^n=\delta_n^n$}
For $n=1$, the relation rewrites as $\tau_1\delta_0^1=\delta_1^1$ and it follows by definition, the fact that $\delta$ is an algebra morphism, by $\varepsilon S=\varepsilon$, and the fact that the unit is a coalgebra morphism:
$$\tau_1\delta_0^1=\,
\psfrag{c}{$\sigma$}
\psfrag{l}{\hspace{-0.1cm}$\delta$}
\rsdraw{0.55}{0.75}{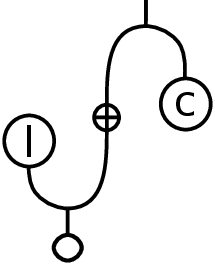}  \;=\,
\psfrag{c}{$\sigma$}
\psfrag{l}{\hspace{-0.1cm}$\delta$}
\rsdraw{0.55}{0.75}{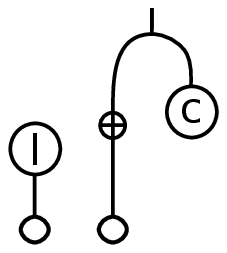}  \;=\sigma =\delta_1^1.$$
Now let $n\ge 2$. By definition, the fact that $\delta$ is an algebra morphism, by $\varepsilon S=\varepsilon$, the fact that the unit is a coalgebra morphism, the left module axiom, and the naturality of the braiding, we have
$$\tau_n\delta_0^n=\,
\psfrag{b}{\hspace{-0.05cm}$1$}
\psfrag{c}{$\sigma$}
\psfrag{f}{\hspace{-0.5cm}$\id_{H^{\tens n-1}}$}
\psfrag{e}{\hspace{-0.4cm}$n-1$}
\psfrag{l}{\hspace{-0.1cm}$\delta$}
\rsdraw{0.55}{0.75}{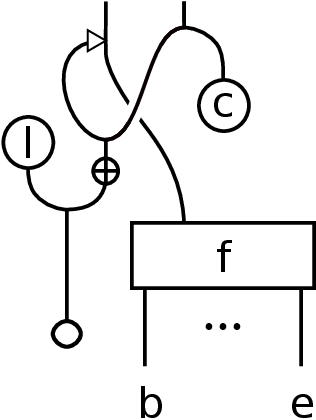}  \;=\,
\psfrag{b}{\hspace{-0.05cm}$1$}
\psfrag{c}{$\sigma$}
\psfrag{f}{\hspace{-0.5cm}$\id_{H^{\tens n-1}}$}
\psfrag{e}{\hspace{-0.4cm}$n-1$}
\psfrag{l}{\hspace{-0.1cm}$\delta$}
\rsdraw{0.55}{0.75}{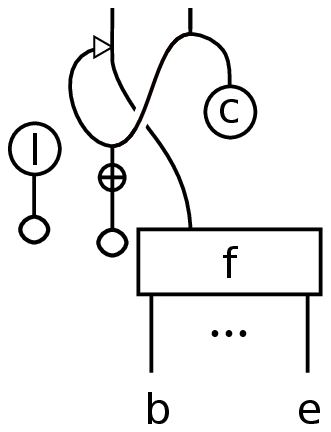}  \;=\,
\psfrag{b}{\hspace{-0.05cm}$1$}
\psfrag{c}{$\sigma$}
\psfrag{f}{\hspace{-0.5cm}$\id_{H^{\tens n-1}}$}
\psfrag{e}{\hspace{-0.4cm}$n-1$}
\psfrag{l}{\hspace{-0.1cm}$\delta$}
\rsdraw{0.55}{0.75}{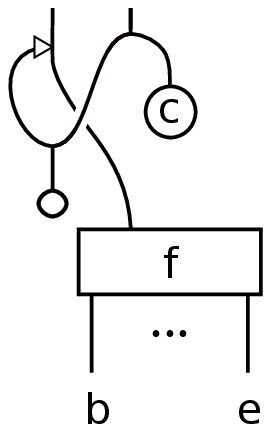}  \;=\,
\psfrag{b}{\hspace{-0.05cm}$1$}
\psfrag{c}{$\sigma$}
\psfrag{f}{\hspace{-0.5cm}$\id_{H^{\tens n-1}}$}
\psfrag{e}{\hspace{-0.4cm}$n-1$}
\psfrag{l}{\hspace{-0.1cm}$\delta$}
\rsdraw{0.55}{0.75}{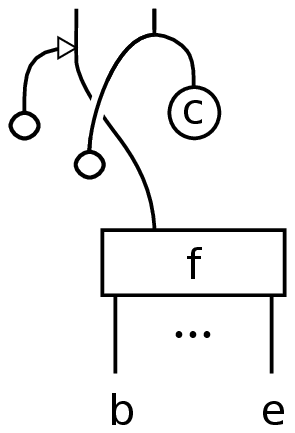}  \;=\delta_n^n.$$

\subsection{The re\-la\-tion~$\tau_n(\delta,\sigma)\sigma_0^n=\sigma_n^n(\tau_{n+1}(\delta,\sigma))^2$} In order to show the relation, one can use Theorem~\ref{powers}.
We first prove the case~$n=0$. It is true since
$$\sigma_0^0(\tau_{1}(\delta,\sigma))^2\overset{(i)}{=} \,
\psfrag{A}{$H$}
\psfrag{l}{\hspace{-0.05cm}$\delta$}
\psfrag{v}{\hspace{0.05cm}$\sigma$}
\psfrag{u}{$H$}
\rsdraw{0.55}{0.75}{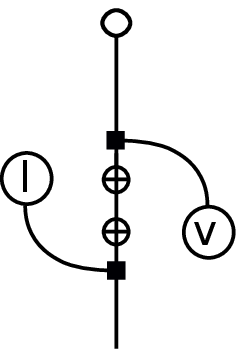}  \;\overset{(ii)}{=}\,
\psfrag{A}{$H$}
\psfrag{l}{\hspace{-0.05cm}$\delta$}
\psfrag{c}{$\sigma$}
\psfrag{u}{$H$}
\rsdraw{0.55}{0.75}{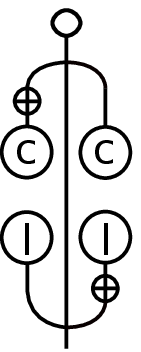}  \;\overset{(iii)}{=}\,
\psfrag{A}{$H$}
\psfrag{l}{\hspace{-0.05cm}$\delta$}
\psfrag{c}{$\sigma$}
\psfrag{u}{$H$}
\rsdraw{0.55}{0.75}{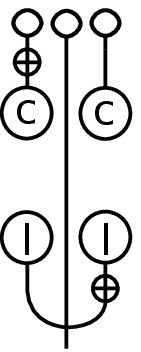}  \;\overset{(iv)}{=}\,
\psfrag{A}{$H$}
\psfrag{l}{\hspace{-0.05cm}$\delta$}
\psfrag{c}{$\sigma$}
\psfrag{u}{$H$}
\rsdraw{0.55}{0.75}{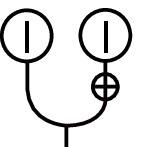}  \;\overset{(v)}{=}\,
\psfrag{A}{$H$}
\rsdraw{0.55}{0.75}{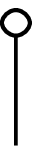}\;\overset{(vi)}{=}\tau_0(\delta,\sigma)\sigma_0^0.$$
Here~$(i)$ follows by using Formula~\eqref{n+1thpower} of Theorem~\ref{powers} and the definition of~$\sigma_0^0$,~$(ii)$ follows from definition of right adjoint action and left coadjoint action of $H$ on itself, the naturality of the braiding, and the fact that~$\delta$ is an algebra morphism and~$\sigma$ is a coalgebra morphism. The equality~$(iii)$ follows by applying twice the fact that the counit is an algebra morphism,~$(iv)$ follows by the fact that $\varepsilon S=\varepsilon$ and since $\sigma$ is a coalgebra morphism, $(v)$ follows by the fact that $\delta$ is an algebra morphism and by the antipode axiom, $(vi)$ follows by definition of $\tau_0(\delta,\sigma)$ and $\sigma_0^0$.

Let us now consider the case when $n\ge 1$. Indeed, the relation still holds since
\begingroup
\allowdisplaybreaks
\begin{align*}
\sigma_n^n(\tau_{n+1}(\delta,\sigma))^2&\overset{(i)}{=}
\,
\psfrag{b}{\hspace{-0.05cm}$1$}
\psfrag{c}{$2$}
\psfrag{n}{\hspace{-0.35cm}$n+1$}
\psfrag{f}{\hspace{0.5cm}$\tau_n(\varepsilon,u)$}
\psfrag{v}[tl][tl]{$\thicksim$}
\psfrag{k}{$H^{\tens n-1}$}
\psfrag{s}{$\sigma$}
\rsdraw{0.55}{0.75}{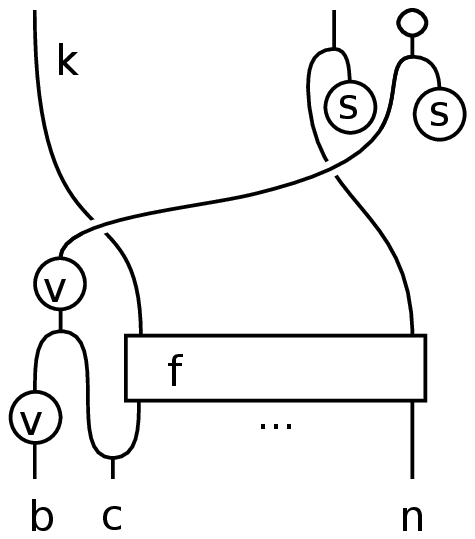}  \;\overset{(ii)}{=}\,
\psfrag{b}{\hspace{-0.05cm}$1$}
\psfrag{c}{$2$}
\psfrag{n}{\hspace{-0.35cm}$n+1$}
\psfrag{f}{\hspace{0.5cm}$\tau_n(\varepsilon,u)$}
\psfrag{v}[tl][tl]{$\thicksim$}
\psfrag{k}{$H^{\tens n-1}$}
\psfrag{s}{$\sigma$}
\rsdraw{0.55}{0.75}{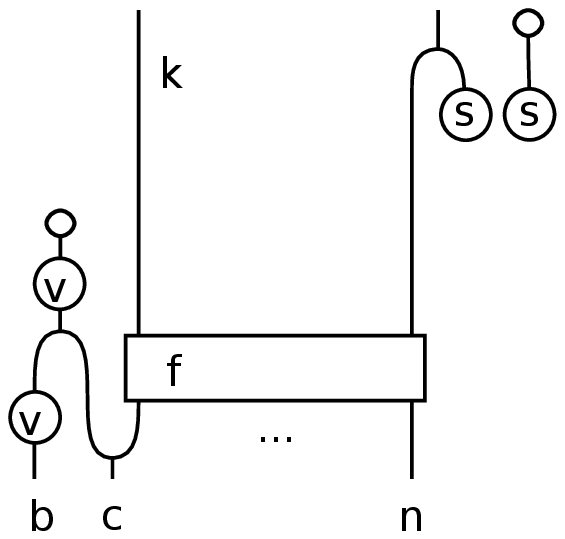}  \;&\overset{(iii)}{=}& \\&\overset{(iii)}{=}
\,
\psfrag{b}{\hspace{-0.05cm}$1$}
\psfrag{c}{$2$}
\psfrag{n}{\hspace{-0.35cm}$n+1$}
\psfrag{f}{\hspace{0.5cm}$\tau_n(\varepsilon,u)$}
\psfrag{v}[tl][tl]{$\thicksim$}
\psfrag{k}{$H^{\tens n-1}$}
\psfrag{s}{$\sigma$}
\rsdraw{0.55}{0.75}{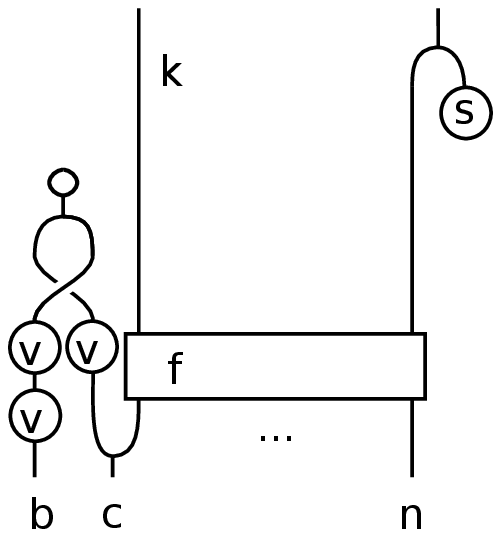}  \;\overset{(iv)}{=}\,
\psfrag{b}{\hspace{-0.05cm}$1$}
\psfrag{c}{$2$}
\psfrag{n}{\hspace{-0.35cm}$n+1$}
\psfrag{f}{\hspace{0.5cm}$\tau_n(\varepsilon,u)$}
\psfrag{v}[tl][tl]{$\thicksim$}
\psfrag{k}{$H^{\tens n-1}$}
\psfrag{s}{$\sigma$}
\psfrag{l}{\hspace{-0.05cm}$\delta$}
\rsdraw{0.55}{0.75}{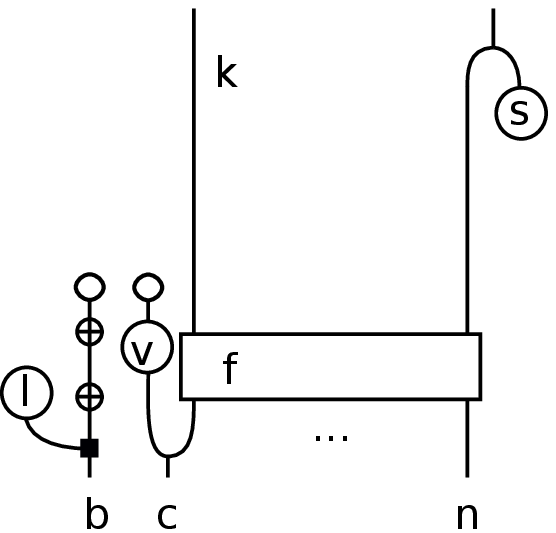}  \;&\overset{(v)}{=}& \\&\overset{(v)}{=}\,
\psfrag{b}{\hspace{-0.05cm}$1$}
\psfrag{c}{$2$}
\psfrag{n}{\hspace{-0.35cm}$n+1$}
\psfrag{f}{\hspace{0.5cm}$\tau_n(\varepsilon,u)$}
\psfrag{v}[tl][tl]{$\thicksim$}
\psfrag{k}{$H^{\tens n-1}$}
\psfrag{s}{$\sigma$}
\psfrag{l}{\hspace{-0.05cm}$\delta$}
\rsdraw{0.55}{0.75}{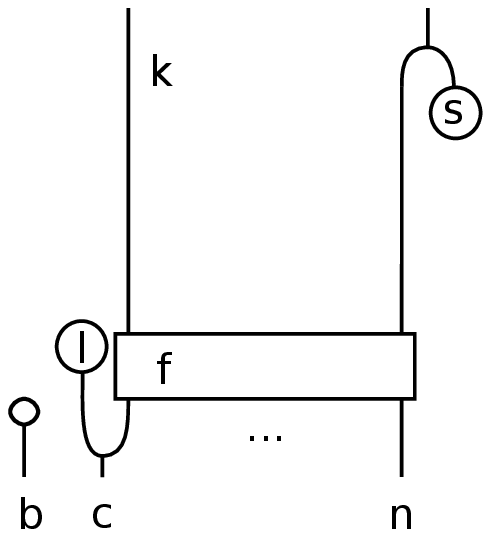}  \; \overset{(vi)}{=} \,
\psfrag{b}{\hspace{-0.05cm}$1$}
\psfrag{c}{$2$}
\psfrag{n}{\hspace{-0.35cm}$n+1$}
\psfrag{f}{\hspace{0.5cm}$\tau_n(\delta,\sigma)$}
\psfrag{k}{$H^{\tens n-1}$}
\psfrag{h}{$H$}
\rsdraw{0.55}{0.75}{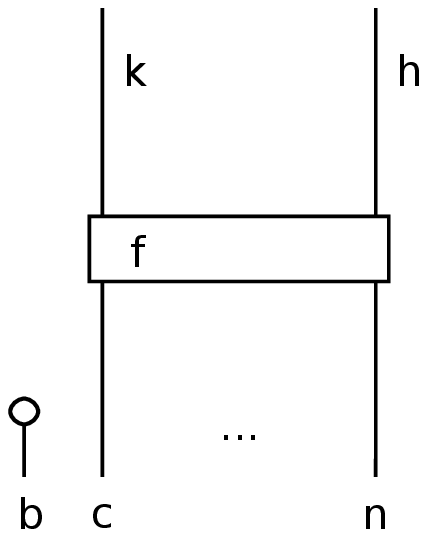}\; &\overset{(vii)}{=}& \tau_n(\delta, \sigma) \sigma_0^n.
\end{align*}
\endgroup
Here~$(i)$ follows by Theorem \ref{powers} applied for~$n+1$ and~$k=2$ and by definition of~$\sigma_n^n$.
The equality~$(ii)$ follows by the fact that counit is an algebra morphism and the naturality of the braiding,~$(iii)$ follows from Lemma~\ref{algebraicppties}$(b)$ and the fact that~$\sigma$ is a coalgebra morphism,~$(iv)$ follows from Lemma~\ref{algebraicppties}$(e)$, by the fact that counit is an algebra morphism, and by the naturality of the braiding.
The equality~$(v)$ follows by the fact that~$\varepsilon S = \varepsilon$, by Remark~\ref{twistedantipodeeps}, and Lemma~\ref{auxilia}$(a)$,~$(vi)$ follows by definition of the paracocyclic operator~$\tau_n(\delta,\sigma)$ and finally,~$(vii)$ follows by definition of~$\tau_n(\delta,\sigma)$ and $\sigma_0^n$.

\bibliographystyle{abbrv}
\bibliography{CMbib1}

\end{document}